\documentclass{article}
\usepackage[utf8]{inputenc}
\usepackage{amssymb,amsmath, amsthm, mathabx}
\usepackage{amsmath,amsfonts,amssymb,amsthm,epsfig,epstopdf,titling,url,array}
\usepackage{tikz}
\usetikzlibrary{shapes,snakes}
\usepackage{color, enumerate}
\usepackage{comment, authblk}
\usepackage{hyperref}
\usepackage{pgfplots}
\usepackage{bm}
\usepackage{stmaryrd}
\usepackage{graphicx}
\usepackage{subfigure}
\usepackage{caption}
\usepackage{a4wide}
\usepackage{titlesec}
\usepackage[a4paper, total={6in, 8in}]{geometry}
\usepackage[normalem]{ulem}

\titleformat*{\subsection}{\large\bfseries}
\numberwithin{equation}{section}

\pgfplotsset{compat=newest}
\pgfplotsset{plot coordinates/math parser=false}
\newlength\figureheight
\newlength\figurewidth

\DeclareMathOperator{\G}{G}

\usepackage{enumerate}
\usepackage{mathrsfs}
\usepackage{wrapfig}

\usepackage{color}

\numberwithin{equation}{section}

\newcommand{\Sig}{\Sigma}

\newcommand{\beq}{\begin{equation}}
\newcommand{\bEq}{\end{equation}}

\newcommand{\al}{\alpha}

\newcommand{\be}{\begin{equation}}
\newcommand{\ee}{\end{equation}}

\newcommand{\e}{{\varepsilon}}

\newcommand{\si}{\sigma}

\newcommand{\fa}{{\mathfrak a}}

\usepackage{amsmath} 
\usepackage{amssymb}
\usepackage{amsthm}

\renewcommand{\cal}{\mathcal}

\newcommand{\wh}{\widehat}
\newcommand{\wt}{\widetilde}

\newcommand{\ii}{\mathrm{i}} 
\newcommand{\dd}{\mathrm{d}}

\renewcommand{\epsilon}{\varepsilon}
\renewcommand{\leq}{\leqslant}
\renewcommand{\geq}{\geqslant}

\renewcommand{\le}{\leq}
\renewcommand{\ge}{\geq}

\newcommand{\E}{\mathbb{E}}

\newcommand{\N}{\mathbb{N}}

\DeclareMathOperator{\im}{Im}

\DeclareMathOperator{\OO}{O}
\DeclareMathOperator{\oo}{o}

\DeclareMathOperator{\bA}{\mathbf{A}}
\DeclareMathOperator{\bC}{\mathbf{C}}

\DeclareMathOperator{\bS}{\mathbf{S}}
\DeclareMathOperator{\bv}{\mathbf{v}}
\DeclareMathOperator{\bu}{\mathbf{u}}
\DeclareMathOperator{\bw}{\mathbf{w}}

\DeclareMathOperator{\bbE}{\mathbb{E}}
\DeclareMathOperator{\bbN}{\mathbb{N}}
\DeclareMathOperator{\bbP}{\mathbb{P}}

\DeclareMathOperator{\sI}{\mathcal{I}}

\DeclareMathOperator{\sW}{\mathcal{W}}

\DeclareMathOperator{\one}{\mathbf{1}}

\theoremstyle{plain} 
\newtheorem{theorem}{Theorem}[section]
\newtheorem*{theorem*}{Theorem}
\newtheorem{lemma}[theorem]{Lemma}
\newtheorem{assumption}[theorem]{Assumption}
\newtheorem*{lemma*}{Lemma}

\newtheorem*{corollary*}{Corollary}
\newtheorem{proposition}[theorem]{Proposition}
\newtheorem*{proposition*}{Proposition}

\newtheorem{definition}[theorem]{Definition}
\newtheorem*{definition*}{Definition}
\theoremstyle{remark}

\newtheorem*{example*}{Example}
\newtheorem{remark}[theorem]{Remark}

\newtheorem*{remark*}{Remark}
\newtheorem*{remarks*}{Remarks}

\renewcommand{\Im}{{\rm{Im}}}
\renewcommand{\Re}{{\rm{Re}}}

\newcommand{\nc}{\normalcolor}

\allowdisplaybreaks

\title{Edge universality of separable covariance matrices}
\author[1]{Fan Yang  \thanks{E-mail: fyang75@wharton.upenn.edu}}
\affil[1]{Department of Statistics, University of Pennsylvania}

\begin{document}
\maketitle

\begin{abstract}

In this paper, we prove the edge universality of largest eigenvalues for separable covariance matrices of the form $\mathcal Q :=A^{1/2}XBX^*A^{1/2}$. Here $X=(x_{ij})$ is an $n\times N$ random matrix with $x_{ij}=N^{-1/2}q_{ij}$, where $q_{ij}$ are $i.i.d.$ random variables with zero mean and unit variance, and $A$ and $B$ are respectively  $n \times n$ and $N\times N$ deterministic non-negative definite symmetric (or Hermitian) matrices. We consider the high-dimensional case, i.e. ${n}/{N}\to d \in (0, \infty)$ as $N\to \infty$. Assuming $\mathbb E q_{ij}^3=0$ and some mild conditions on $A$ and $B$, we prove that the limiting distribution of the largest eigenvalue of $\mathcal Q$ coincide with that of the corresponding Gaussian ensemble (i.e.\;$\mathcal Q$ with $X$ being an $i.i.d.$ Gaussian matrix) as long as we have $\lim_{s \rightarrow \infty}s^4 \mathbb{P}(\vert  q_{ij} \vert \geq s)=0$, which is a sharp moment condition for edge universality. If we take $B=I$, then $\mathcal Q$ becomes the normal sample covariance matrix and the edge universality holds true without the vanishing third moment condition. So far, this is the strongest edge universality result for sample covariance matrices with correlated data (i.e. non-diagonal $A$) and heavy tails, which improves the previous results in \cite{BPZ1,LS} (assuming high moments and diagonal $A$), \cite{Anisotropic} (assuming high moments) and \cite{DY} (assuming diagonal $A$). 
\end{abstract}

\tableofcontents

\begin{section}{Introduction}\label{sec_intro}

Sample covariance matrices are fundamental objects in multivariate statistics. Given a centered random vector $\mathbf y\in \mathbb R^n$ and its $i.i.d.$ copies $\mathbf y_i$, $ i = 1, \cdots, N$, the sample covariance matrix $\mathcal Q := N^{-1}\sum_i \mathbf y_i \mathbf y_i^*$ is the simplest estimator for the covariance matrix $A:=\mathbb E \mathbf y \mathbf y^*$. In fact, if the dimension $n$ of the data is fixed, then $\mathcal Q$ converges almost surely to $\Sigma$ as $N\to \infty$. However, in many modern applications, high dimensional data, i.e. data with $n$ being comparable to or even larger than $N$, is commonly collected in various fields, such as statistics \cite{DT2011,IJ,IJ2,IJ2008}, economics \cite{Onatski2} and population genetics \cite{Genetics}, to name a few. In this setting, $A$ cannot be estimated through $Q$ directly due to the so-called {\it curse of dimensionality}. Yet, some properties of $A$ can be inferred from the eigenvalue statistics of $\mathcal Q$. 

In this paper, we focus on the limiting distribution of the largest eigenvalues of high-dimensional sample covariance matrices, which is of great interest to the principal component analysis. The largest eigenvalue has been widely used in hypothesis testing problems on the structure of covariance matrices, see e.g. \cite{BDMN,Karoui,IJ2,OMH}. Of course the list is very far from being complete, and we refer the reader to \cite{IJ,PA,YZB} for a comprehensive review. Precisely, we will consider sample covariance matrices of the form 
$$\mathcal Q = A^{1/2}XX^{*}A^{1/2},$$ 
where the data matrix $X=(x_{ij})$ is an $n \times N$ random matrix with $i.i.d.$ entries such that $\mathbb E x_{11}=0$ and $\mathbb E |x_{11}|^2 = N^{-1}$, and $A$ is an $n\times n$ deterministic non-negative definite symmetric (or Hermitian) matrix. On dimensionality, we assume that $n/N \to d \in (0,\infty)$ as $N\to \infty$. It is well-known that the empirical spectral distribution (ESD) of $\mathcal Q$ converges to the (deformed) Marchenko-Pastur (MP) law \cite{MP}, whose rightmost edge $\lambda_+$ gives the asymptotic location of the largest eigenvalue. Moreover, it was proved in a series of papers that under an $N^{2/3}$ scaling, the distribution of the largest eigenvalue $\lambda_1(\mathcal Q)$ around $\lambda_+$ converges to the famous Tracy-Widom distribution \cite{TW1,TW}. 
This result is commonly referred to as the {\it{edge universality}}, in the sense that it is independent of the detailed distribution of the entries of $X$. The limiting distribution of $\lambda_1$ was first obtained for $\mathcal Q$ with $X$ consisting of $i.i.d.$ centered Gaussian entries (i.e.\;$XX^*$ is a Wishart matrix) and with trivial covariance (i.e. $A=I$) \cite{IJ2}. The edge universality in the $A=I$ case was later proved for all random matrices $X$ whose entries satisfy a sub-exponential decay \cite{PY}. When $A$ is a non-scalar diagonal matrix, the Tracy-Widom distribution was first proved for the case with $i.i.d.$ Gaussian $X$ in \cite{Karoui} (non-singular $A$ case) and \cite{Onatski} (singular $A$ case). Later the edge universality with general diagonal $A$ was proved in \cite{BPZ1,LS} for $X$ with entries having arbitrarily high moments, and in \cite{DY} for $X$ with entries satisfying the tail condition \eqref{tail_cond} below. The most general case with non-diagonal $A$ is considered in \cite{Anisotropic}, where the edge universality was proved under the arbitrarily high moments assumption. 
 

Without loss of generality, we may assume that the row indices of the data matrix correspond to the spatial locations and the column indices correspond to the observation times. Then the data model $A^{1/2}X$ corresponds to observing independent samples at $N$ different times, and hence is incompetent to model sampling data with time correlations. In fact, the spatio-temporal sampling data is commonly collected in environmental study \cite{MF2006,KJ1999,LMS2008,MG2003} and wireless communications \cite{RMT_Wireless}. Motivated by this fact, we shall consider a separable data model $Y=A^{1/2}XB^{1/2}$, where $A$ and $B$ are respectively $n\times n$ and $N\times N$ deterministic non-negative definite symmetric (or Hermitian) matrices. Here $A$ and $B$ are not necessarily diagonal, which means that the entries are correlated both in space and in time. The name ``separable" is because the joint covariance of $Y$, viewed as an $(Nn)$-dimensional vector, is given by a separable form $A\otimes B$. 
In particular, if the entries of $X$ are Gaussian, then the joint distribution of $Y$ is $\mathcal N_{Nn}(0, A\otimes B)$. Note that the separable model describes a process where the time correlation does not depend on the spatial location and the spatial correlation does not depend on time, i.e. there is no space-time interaction.

The separable covariance matrix is defined as $\mathcal Q:=YY^*=A^{1/2}XB X^* A^{1/2}$. It has been proved to be very useful for various applications. For example, in wireless communications, it was shown in \cite{Verdu} that an estimate of the capacity is directly given by various informations of the largest eigenvalue. The spectral properties of separable covariance matrices have been investigated in some recent works, see e.g. \cite{PRE,  Karoui2009, Separable, WANG2014, Zhang_thesis}. However, the edge universality is much less known compared with sample covariance matrices. It is known that the edge universality generally follows from an optimal local law for the resolvent $G=(\mathcal Q-z)^{-1}$ near the spectral edge, where $z\in \mathbb C_+:=\{z\in \mathbb C: \im z>0\}$ with $\im z\gg N^{-1}$ \cite{BPZ1,DY,Anisotropic, LS}. Consider an $n\times N$ matrix $X$ consisting of independent centered entries with general variance profile $\mathbb E|x_{ij}|^2=\sigma_{ij}/N$, then an optimal local law was prove in \cite{Alt_Gram,AEK_Gram} for the resolvent $(XX^*-z)^{-1}$ under the arbitrarily high moments assumption. Note that this gives the local law for $G$ in the case where both $A$ and $B$ are diagonal. However, if $A$ and $B$ are not diagonal, no such local law is proved so far, let alone the edge universality.  


The goal of this paper is to fill this gap. More precisely, we shall prove that for general (non-diagonal) $A$ and $B$ satisfying some mild assumptions, the limiting distribution of the rescaled largest eigenvalue $N^{\frac{2}{3}}\left(\lambda_1(\mathcal Q)-\lambda_+\right)$ coincides with that of the corresponding Gaussian ensemble (i.e. ${\mathcal Q}^G=A^{1/2}X^G B (X^G)^* A^{1/2}$ with $X^G$ being an $i.i.d.$ Gaussian matrix) as long as the following conditions hold:
\begin{equation}
\lim_{s \rightarrow \infty } s^4 \mathbb{P}\left( \vert \sqrt{N}x_{11} \vert \geq s\right)=0 ,\label{tail_cond}
\end{equation}
and
\begin{equation}
\mathbb Ex_{11}^3 = 0 . \label{assm_3rdmoment}
\end{equation}
For a precise statement, the reader can refer to Theorem \ref{main_thm}. Note that the tail condition  (\ref{tail_cond}) is slightly weaker than the finite fourth moment condition for $\sqrt{N}x_{11}$, and in fact is sharp for the edge universality of the largest eigenvalue, see Remark \ref{sharp} below. Historically, for sample covariance matrices, it was proved in \cite{YBK} that $\lambda_1 \rightarrow \lambda_+$ almost surely in the null case with $A=I$ if the fourth moment exists. Later the finite fourth moment condition is proved to be also necessary for the almost sure convergence of $\lambda_1$ \cite{BSY}. On the other hand, it was proved in \cite{SW} that $\lambda_1 \rightarrow \lambda_+$ in probability under the condition (\ref{tail_cond}). If $A$ is diagonal, it was proved in \cite{DY} that the condition (\ref{tail_cond}) is actually necessary and sufficient for the edge universality of sample covariance matrices to hold. 

On the other hand, the condition \eqref{assm_3rdmoment} is more technical and should be considered to be removed in future works. We now discuss about it briefly. The main difficulty in studying $\mathcal Q=A^{1/2}XB X^* A^{1/2}$ and its resolvent is due to the fact that the entries of $A^{1/2} XB^{1/2}$ are not independent. We assume that $A$ and $B$ have eigendecompositions
$A= U\Sig U^*$ and $ B= V\wt \Sigma V^*.$
Then in the special case where $X\equiv X^{G}$ is $i.i.d.$ Gaussian, it is easy to see that 
$$A^{1/2} X^{G} B (X^{G})^* A^{1/2} \stackrel{d}{=}  U\left(\Sig^{1/2} X^{G} \wt \Sig^{1/2}\right)U^* \sim \Sig^{1/2} X^{G} \wt \Sig^{1/2},$$
which is reduced to a separable covariance matrix with diagonal $\Sig$ and $\wt \Sig$. This case can be handled using the 
current method in \cite{DY}. To extend the result in the Gaussian case to the general $X$ case, we use a self-consistent comparison argument developed in \cite{Anisotropic}. For this argument to work, we need to assume that the third moments of the $X$ entries coincide with that of the Gaussian random variable, i.e. the condition \eqref{assm_3rdmoment}. (Actually it is common that for a comparison argument to work for random matrices, some kind of four moment matching is needed; see e.g. \cite{TaoVu_edge,TaoVu_4moment,TaoVu_local}.) If one of the $A$ and $B$ is diagonal, then a notable argument in \cite[Section 8]{Anisotropic} can remove this requirement by exploring more detailed structures of the resolvents of $\mathcal Q$. However, their argument is quite specific and cannot be adapted to the general case with both $A$ and $B$ being non-diagonal. Nevertheless, this is still a welcome result, which shows that for sample covariance matrices, the condition \eqref{assm_3rdmoment} is not necessary and the edge universality holds as long as \eqref{tail_cond} holds. \nc For a more detailed explanation on why and where the condition \eqref{assm_3rdmoment} is needed, we refer the reader to the discussion following Theorem \ref{LEM_SMALL}.\nc

Finally, we believe that the largest eigenvalue of the Gaussian separable covariance matrix ${\mathcal Q}^G$ should converge to the Tracy-Widom distribution. However, to the best of our knowledge, so far there is no explicit proof for this fact. We will give a proof in another paper \cite{DY2}.

This paper is organized as follows. In Section \ref{main_result}, we first define the limiting spectral distribution of the separable covariance matrix and its rightmost edge $\lambda_+$, which will depend only on the empirical spectral densities (ESD) of $A$ and $B$. Then we will state the main theorem---Theorem \ref{main_thm}--- of this paper. In Section \ref{sec_maintools}, we introduce the notations and collect some tools including the {\it anisotropic local law (Theorem \ref{LEM_SMALL})}, {\it rigidity of eigenvalues} (Theorem \ref{thm_largerigidity}) and a comparison theorem (Theorem \ref{lem_comparison}). In Section \ref{sec_cutoff}, we prove Theorem \ref{main_thm} with these tools. Then Section \ref{sec_Gauss} and Section \ref{sec_comparison} are devoted to proving Theorem \ref{LEM_SMALL}, and Section \ref{sec_Lindeberg} is devoted to proving Theorem \ref{thm_largerigidity} and Theorem \ref{lem_comparison}. 

\vspace{5pt}

\noindent{\bf Conventions.} 
The fundamental large parameter is $N$ and we always assume that $n$ is comparable to $N$. All quantities that are not explicitly constant may depend on $N$, and we usually omit $N$ from our notations. We use $C$ to denote a generic large positive constant, whose value may change from one line to the next. Similarly, we use $\epsilon$, $\tau$, $\delta$ and $c$ to denote generic small positive constants. If a constant depends on a quantity $a$, we use $C(a)$ or $C_a$ to indicate this dependence. We use $\tau>0$ in various assumptions to denote a small positive constant. All constants appear in the statements or proof may depend on $\tau$; we neither indicate nor track this dependence.

For two quantities $a_N$ and $b_N$ depending on $N$, the notation $a_N = \OO(b_N)$ means that $|a_N| \le C|b_N|$ for some constant $C>0$, and $a_N=\oo(b_N)$ means that $|a_N| \le c_N |b_N|$ for some positive sequence $c_N\downarrow 0$ as $N\to \infty$. We also use the notations $a_N \lesssim b_N$ if $a_N = \OO(b_N)$, and $a_N \sim b_N$ if $a_N = \OO(b_N)$ and $b_N = \OO(a_N)$. For a matrix $A$, we use $\|A\|:=\|A\|_{l^2 \to l^2}$ to denote the operator norm; 
for a vector $\mathbf v=(v_i)_{i=1}^n$, $\|\mathbf v\|\equiv \|\mathbf v\|_2$ stands for the Euclidean norm, while $|\mathbf v| \equiv \|\mathbf v\|_1$ stands for the $l^1$-norm. In this paper, we often write an identity matrix as $I$ or $1$ without causing any confusions. If two random variables $X$ and $Y$ have the same distribution, we write $X\stackrel{d}{=} Y$.


%
%
%

\vspace{5pt}

\noindent{\bf Acknowledgements.} 
I would like to thank Marc Potters and Xiucai Ding for bringing this problem to my attention and for helpful discussions. I also want to thank my advisor Jun Yin for the guidance and valuable suggestions. 
\end{section}

\section{Definitions and Main Result}\label{main_result}

\subsection{Separable covariance matrices}

We consider a class of separable covariance matrices of the form $\mathcal Q_1:=A^{1/2}XBX^*A^{1/2}$, where $A$ and $B$ are deterministic non-negative definite symmetric (or Hermitian) matrices. Note that $A$ and $B$ are not necessarily diagonal. We assume that $X=(x_{ij})$ is an $n\times N$ random matrix with entries $x_{ij}= N^{-1/2}q_{ij}$, $1 \leq i \leq n$, $1 \leq j \leq N$, where $q_{ij}$ are {\it{i.i.d.}} random variables satisfying
\begin{equation}\label{assm1}
\mathbb{E} q_{11} =0, \ \quad \ \mathbb{E} \vert q_{11} \vert^2  =1.
\end{equation}
For definiteness, in this paper we focus on the real case, i.e.\;the random variable $q_{11}$ is real. However, we remark that our proof can be applied to the complex case after minor modifications if we assume in addition that $\Re\, q_{11}$ and $\Im\, q_{11}$ are independent centered random variables with variance $1/2$.  We will also use the $N \times N$ matrix $\mathcal Q_2:=B^{1/2}X^* A X B^{1/2}$. 
We assume that the aspect ratio $d_N:= n/N$ satisfies $\tau \le d_N \le \tau^{-1}$ for some constant $0<\tau <1$. Without loss of generality, by switching the roles of $\mathcal Q_1$ and $\mathcal Q_2$ if necessary, 
we can assume that 
\begin{equation}
 \tau \le d_N \le 1 \ \ \text{ for all } N. \label{assm2}
 \end{equation}
For simplicity of notations, we will often abbreviate $d_N$ as $d$ in this paper. We denote the eigenvalues of $\mathcal Q_1$ and $\mathcal Q_2$ in descending order by $\lambda_1(\mathcal Q_1)\geq \ldots \geq \lambda_{n}(\mathcal Q_1)$ and $\lambda_1(\mathcal Q_2) \geq \ldots \geq \lambda_N(\mathcal Q_2)$. Since $\mathcal Q_1$ and $\mathcal Q_2$ share the same nonzero eigenvalues, we will for simplicity write $\lambda_j$, $1\le j \le N\wedge n$, to denote the $j$-th eigenvalue of both $\mathcal Q_1$ and $\mathcal Q_2$ without causing any confusion. 

We assume that $A$ and $B$ have eigendecompositions
\be\label{eigen}A= U\Sig U^*, \quad B= V\wt \Sigma V^* ,\quad \Sig=\text{diag}(\si_1, \ldots, \si_n), \quad \wt\Sig=\text{diag}(\wt \si_1, \ldots, \wt \si_N),
\ee
where
$$\si_1 \ge \si_2 \ge \ldots \ge \si_n \ge 0, \quad \wt\si_1 \ge \wt\si_2 \ge \ldots \ge \wt \si_N \ge 0.$$
We denote the empirical spectral densities (ESD) of $A$ and $B$ by
\begin{equation}\label{sigma_ESD}
\pi_A\equiv \pi_A^{(n)} := \frac{1}{n} \sum_{i=1}^n \delta_{\si_i},\quad \pi_B\equiv \pi_B^{(N)} := \frac{1}{N} \sum_{i=1}^N \delta_{\wt\si_i}.
\end{equation}
We assume that there exists a small constant $0<\tau<1$ such that for all $N$ large enough,
\begin{equation}\label{assm3}
\max\{\si_1, \wt \sigma_1\} \le \tau^{-1}, \quad \max\left\{\pi_A^{(n)}([0,\tau]), \pi_B^{(N)}([0,\tau])\right\} \le 1 - \tau .
\end{equation}
The first condition means that the operator norms of $A$ and $B$ are bounded by $\tau^{-1}$, and the second condition means that the spectrums of $A$ and $B$ do not concentrate at zero.

We summarize our basic assumptions here for future reference.
\begin{assumption}\label{assm_big1}
We assume that $X$ is an $n\times N$ random matrix with real $i.i.d.$ entries satisfying (\ref{assm1}), $A$ and $B$ are deterministic non-negative definite symmetric matrices satisfying \eqref{eigen} and (\ref{assm3}), and $d_N$ satisfies \eqref{assm2}.
\end{assumption}

\subsection{Resolvents and limiting law}

In this paper, we will study the eigenvalue statistics of $\mathcal Q_{1}$ and $\mathcal Q_2$ through their {\it{resolvents}} (or  {\it{Green's functions}}). It is equivalent to study the matrices 
\be\label{Qtilde}
\wt{\mathcal Q}_1(X):=\Sig^{1/2} U^{*}XBX^*U\Sig^{1/2}, \quad \wt{\mathcal Q}_2(X):=\wt\Sig^{1/2}V^*X^* A X V\wt \Sig^{1/2}.
\ee
In this paper, we shall denote the upper half complex plane and the right half real line by 
$$\mathbb C_+:=\{z\in \mathbb C: \im z>0\}, \quad \mathbb R_+:=[0,\infty).$$ 

\begin{definition}[Resolvents]\label{resol_not}
For $z = E+ \ii \eta \in \mathbb C_+,$ we define the resolvents for $\wt {\mathcal Q}_{1,2}$ as
\begin{equation}\label{def_green}
\mathcal G_1(X,z):=\left(\wt{\mathcal Q}_1(X) -z\right)^{-1} , \ \ \ \mathcal G_2 (X,z):=\left(\wt{\mathcal Q}_2(X)-z\right)^{-1} .
\end{equation}
 We denote the ESD $\rho^{(n)}$ of $\wt {\mathcal Q}_{1}$ and its Stieltjes transform as
\be\label{defn_m}
\rho\equiv \rho^{(n)} := \frac{1}{n} \sum_{i=1}^n \delta_{\lambda_i(\wt{\mathcal Q}_1)},\quad m(z)\equiv m^{(n)}(z):=\int \frac{1}{x-z}\rho_{1}^{(n)}(\dd x)=\frac{1}{n} \mathrm{Tr} \, \mathcal G_1(z).
\ee
We also introduce the following quantities:
$$m_1(z)\equiv m_1^{(n)}(z):= \frac{1}{N}\sum_{i=1}^n\sigma_i (\mathcal G_1(z) )_{ii},\quad m_2(z)\equiv m_2^{(N)}(x):=\frac{1}{N}\sum_{\mu=1}^N \wt\sigma_\mu (\mathcal G_2(z) )_{\mu\mu}. $$

\end{definition}

It was shown in \cite{Separable} that if $d_N \to d \in (0,\infty)$ and $\pi_A^{(n)}$, $\pi_B^{(N)}$ converge to certain probability distributions, then almost surely $\rho^{(n)}$ converges to a deterministic distributions $ \rho_{\infty}$. We now describe it through the Stieltjes transform
$$m_{\infty}(z):=\int_{\mathbb R} \frac{\rho_{\infty}(\dd x)}{x-z}, \quad z \in \mathbb C_+.$$
For any finite $N$ and $z\in \mathbb C_+$, we define $(m^{(N)}_{1c}(z),m^{(N)}_{2c}(z))\in \mathbb C_+^2$ as the unique solution to the system of self-consistent equations
\begin{equation}\label{separa_m12}
{m^{(n)}_{1c}(z)} = d_N \int\frac{x}{-z\left[1+xm^{(N)}_{2c}(z) \right]} \pi_A^{(n)}(\dd x), \quad  {m^{(N)}_{2c}(z)} =  \int\frac{x}{-z\left[1+xm^{(N)}_{1c}(z) \right]} \pi_B^{(N)}(\dd x).
\end{equation}
Then we define
\begin{equation}\label{def_mc}
m_c(z)\equiv m_c^{(n)}(z):= \int\frac{1}{-z\left[1+xm^{(N)}_{2c}(z) \right]} \pi_A^{(n)}(\dd x).
\end{equation}
It is easy to verify that $m_c^{(n)}(z)\in \mathbb C_+$ for $z\in \mathbb C_+$. Letting $\eta \downarrow 0$, we can obtain a probability measure $\rho_{c}^{(n)}$ with the inverse formula
\begin{equation}\label{ST_inverse}
\rho_{c}^{(n)}(E) = \lim_{\eta\downarrow 0} \frac{1}{\pi}\Im\, m^{(n)}_{c}(E+\ii \eta).
\end{equation}
If $d_N \to d \in (0,\infty)$ and $\pi_A^{(n)}$, $\pi_B^{(N)}$ converge to certain probability distributions, then $m_c^{(n)}$ also converges and we define
$$m_{\infty}(z):=\lim_{N\to \infty} m_c^{(n)}(z), \ \ z \in \mathbb C_+.$$
Letting $\eta \downarrow 0$, we can recover the asymptotic eigenvalue density $ \rho_{\infty}$ with
\begin{equation}\label{ST_inverse}
\rho_{\infty}(E) = \lim_{\eta\downarrow 0} \frac{1}{\pi}\Im\, m_{\infty}(E+\ii \eta).
\end{equation}
It is also easy to see that $\rho_\infty$ is the weak limit of $\rho_{c}^{(n)}$. 

The above definitions of $m_c^{(n)}$, $\rho_c^{(n)}$, $m_\infty$ and $\rho_\infty$ make sense due to the following theorem. Throughout the rest of this paper, we often omit the super-indices $(n)$ and $(N)$ from our notations. 


\begin{theorem} [Existence, uniqueness, and continuous density]
For any $z\in \mathbb C_+$, there exists a unique solution $(m_{1c},m_{2c})\in \mathbb C_+^2$ to the systems of equations in (\ref{separa_m12}). The function $m_c$ in (\ref{def_mc}) is the Stieltjes transform of a probability measure $\mu_c$ supported on $\mathbb R_+$. Moreover, $\mu_c$ has a continuous derivative $\rho_c(x)$ on $(0,\infty)$, which is defined by \eqref{ST_inverse}.
\end{theorem}
\begin{proof}
See {\cite[Theorem 1.2.1]{Zhang_thesis}}, {\cite[Theorem 2.4]{Hachem2007}} and {\cite[Theorem 3.1]{Separable_solution}}.
\end{proof}

We now make a small detour and discuss about another very enlightening way to understand the Stieltjes transforms $m_{1,2c}$ and $m_c$. 
Consider the vector solution $\mathbf v=(v_1,\cdots, v_n)$ to the following self-consistent vector equation \cite{Alt_Gram,AEK_Gram}:
\begin{equation}\label{self_vector}
\frac{1}{\mathbf v(z)} = - z + S\frac{1}{1+ S^T \mathbf v(z)}, \quad z\in \mathbb C_+,
\end{equation}
where $1/\mathbf v$ denotes the entrywise reciprocal, and $S$ is an $n\times N$ matrix with entries
\be\label{defS}
S_{i\mu}=\frac{1}{N}\sigma_i \wt \sigma_\mu, \quad i \in \llbracket 1,n\rrbracket, \quad \mu \in \llbracket 1,N\rrbracket.
\ee
In fact, if one regards $\mathcal X_1:= \llbracket 1,n\rrbracket$ and $\mathcal X_2:= \llbracket 1,N\rrbracket$ as measure spaces equipped with counting measures
$$\pi_1 = \sum_{i=1}^n \delta_i, \quad \pi_2 = \sum_{\mu=1}^N \delta_\mu, $$
then $S$ defines a linear operator $S:l^\infty(\mathcal X_2)\to l^\infty(\mathcal X_1)$ such that
$$(S\mathbf w)_i = \frac{\sigma_i }{N} \sum_{\mu=1}^N \wt \sigma_\mu w_\mu ,\quad \mathbf w\in l^\infty(\mathcal X_2), \quad i\in \mathcal X_1. $$
Now we can regard \eqref{self_vector} as a self-consistent equation of the function $\mathbf v: \mathbb C_+ \to l^\infty (\mathcal X_1)$.
Suppose $\mathbf v$ is a solution to \eqref{self_vector} with $\im \mathbf v(z)>0$, then it is easy to verify that
$$m_{1c}= \frac{1}{N}\sum_{i=1}^n \sigma_i v_i, \quad m_{2c}= \frac{1}{N}\sum_{\mu=1}^N  \frac{\wt \sigma_\mu}{-z(1+ \wt \sigma_\mu m_{1c})}, \quad m_c=\frac1n\sum_{i=1}^n v_i .$$
The structure of the solution $\mathbf v$ was well-studied in \cite{Alt_Gram,AEK_Gram}. In particular, one has the following preliminary result on the existence and uniqueness of the solution.

\begin{theorem}[Proposition 2.1 of \cite{Alt_Gram}]
There is a unique function $\mathbf v:\mathbb C_+\to l^\infty(\mathcal X_1)$ satisfying \eqref{self_vector} and $\im \mathbf v(z)>0$ for all $z\in \mathbb C_+$. Moreover, for each $k\in \mathcal X_1$, there is a unique probability measure $\mu_k$ on $\mathbb R$ such that $v_k$ is the Stieltjes transform of $\mu_k$, i.e.
$$v_k (z)=\int_0^\infty \frac{1}{E-z} \mu_k(\dd E), \quad z\in \mathbb C_+.$$
The measures $\mu_k$, $k\in \mathcal X_1$, all have the same support contained in $[0,\wh C]$, where 
$$\wh C:=4\max\left\{\|S\|_{l^\infty(\cal X_2)\to l^\infty(\cal X_1)},\|S^*\|_{l^\infty(\cal X_1)\to l^\infty(\cal X_2)} \right\}.$$
%
\end{theorem}

Now we go back to study the equations in (\ref{separa_m12}). If we define the function
\begin{equation}\label{separable_MP}
f(z,\al):=- \al + \int\frac{x}{-z+xd_N \int\frac{t}{1+t\al} \pi_A(\dd t)} \pi_B(\dd x) ,
\end{equation}
then $m_{2c}(z)$ can be characterized as the unique solution to the equation $f(z,\al)=0$ of $\al$ with $\Im \, \al> 0$, and $m_{1c}(z)$ is defined using the first equation in \eqref{separa_m12}.
Moreover, $m_{1,2c}(z)$ are the Stieltjes transforms of densities $\rho_{1,2c}$:
$$\rho_{1,2c}(E) = \lim_{\eta\downarrow 0} \frac{1}{\pi}\Im\, m_{1,2c}(E+\ii \eta).$$
Then we have the following result.

\begin{lemma}\label{lambdar}
The densities $\rho_{c}$ and $\rho_{1,2c}$ all have the same support on $(0,\infty)$, which is a union of intervals: 
\begin{equation}\label{support_rho1c}
{\rm{supp}} \, \rho_{c} \cap (0,\infty) ={\rm{supp}} \, \rho_{1,2c} \cap (0,\infty) = \bigcup_{k=1}^p [a_{2k}, a_{2k-1}] \cap (0,\infty),
\end{equation}
where $p\in \mathbb N$ depends only on $\pi_{A,B}$. Moreover, $(x,\al)=(a_k, m_{2c}(a_k))$ are the real solutions to the equations
\begin{equation}
f(x,\al)=0, \ \ \text{and} \ \ \frac{\partial f}{\partial \al}(x,\al) = 0. \label{equationEm2}
\end{equation}
Moreover, we have $m_{1c}(a_1) \in (-\wt \sigma_1^{-1}, 0)$ and $m_{2c}(a_1) \in (-\sigma_1^{-1}, 0)$. 
\end{lemma}
\begin{proof}
See Section 3 of \cite{Separable_solution}.
\end{proof}

 We shall call $a_k$ the spectral edges. In particular, we will focus on the rightmost edge $\lambda_+ := a_1$. 
Now we make the following assumption: there exists a constant $\tau>0$ such that 
\begin{equation}\label{assm_gap}
1 + m_{1c}(\lambda_+) \wt \sigma_1 \ge \tau, \quad 1 + m_{2c}(\lambda_+) \sigma_1\ge \tau. 
\end{equation}
This assumption guarantees a regular square-root behavior of the spectral densities $\rho_{1,2c}$ near $\lambda_+$ as shown by the following lemma.

\begin{lemma} \label{lambdar_sqrt}
Under the assumptions \eqref{assm2}, \eqref{assm3} and \eqref{assm_gap}, there exist constants $a_{1,2}>0$ such that
\be\label{sqroot3}
\rho_{1,2c}(\lambda_+ - x) = a_{1,2} x^{1/2} + \OO(x), \quad x\downarrow 0,
\ee
and
\be\label{sqroot4}
\quad m_{1,2c}(z) = m_{1,2c}(\lambda_+) + \pi a_{1,2}(z-\lambda_+)^{1/2} + \OO(|z-\lambda_+|), \quad z\to \lambda_+ , \ \ \im z\ge 0.
\ee
The estimates \eqref{sqroot3} and \eqref{sqroot4} also hold for $\rho_c$ and $m_c$ with a different constant. 
\end{lemma}
\begin{proof}
Differentiating the equation $f(z,\al)=0$ with respect to $\al$, we can get that $z'(m_r)=0$ and $z''(m_r)=-{\partial_\al^2 f (\lambda_+,m_r)}/{\partial_z f(\lambda_+,m_r)}$, where $m_r:=m_{2c}(\lambda_+)$. 
After a straightforward calculation, we have
$${\partial_z f}(z, \al) = \int\frac{x}{z^2 \left[1  +x g(z,\al)\right]^2} \pi_B(\dd x), \quad g(z,\al): = d_N \int\frac{t}{ -z\left(1+t\al \right) } \pi_A(\dd t),$$
and
$$\partial_\al^2 f(z,\al) = - 2\int\frac{x^3}{z\left[1+x g(z,\al)\right]^3}\left(\partial_\al g(z,\al)\right)^2 \pi_B(\dd x) + \int\frac{x^2}{z\left[1+x g(z,\al)\right]^2}\partial_\al^2 g(z,\al) \pi_B(\dd x),$$
where 
$$\partial_\al g(z,\al)=  d_N \int\frac{t^2}{ z\left(1+t\al \right)^2 } \pi_A(\dd t),\quad  \partial_\al^2 g(z,\al)=  -2d_N \int\frac{t^3}{ z\left(1+t\al \right)^3 } \pi_A(\dd t).$$
Using \eqref{assm3} and \eqref{assm_gap}, it is easy to show that 
$$|\partial_z f (\lambda_+,m_r)| \sim 1, \quad  \left|\partial_\al^2 f (\lambda_+,m_r)\right| \sim 1.$$
Thus we have $|z''(m_r)|\sim 1$, which by Theorem 3.3 of \cite{Separable_solution}, implies \eqref{sqroot3} and \eqref{sqroot4} for $\rho_{2c}$ and $m_{2c}$. 
The estimates for $\rho_{1c}$, $m_{1c}$, $\rho_c$, and $m_c$ then follow from simple applications of \eqref{separa_m12} and \eqref{def_mc}.
\end{proof}


\subsection{Main result}

The main result of this paper is the following theorem. 

\begin{theorem} \label{main_thm}
Let $\mathcal Q_1:=A^{1/2}XBX^*A^{1/2}$ be an $n \times n$ separable covariance matrix with $A$, $B$ and $X$ satisfying Assumption \ref{assm_big1} and \eqref{assm_gap}. 
Let $\lambda_1$ be the largest eigenvalue of $\mathcal Q_1$. 
If the conditions (\ref{tail_cond}) and \eqref{assm_3rdmoment} hold, then we have
\begin{equation}
\nc \lim_{N\to \infty}\left[\mathbb{P}(N^{{2}/{3}}(\lambda_1 - \lambda_+) \leq s)- \mathbb{P}^G(N^{{2}/{3}}(\lambda_1 - \lambda_+) \leq s)\right] = 0  \nc \label{SUFFICIENT}
\end{equation}
for all $s\in \mathbb R$, where $\mathbb P^G$ denotes the law for {\nc $X=(x_{ij})$ with real {\it{i.i.d.}}\;Gaussian entries $N^{1/2}x_{ij}= q_{ij}$ satisfying \eqref{assm1}}.
The condition \eqref{assm_3rdmoment} is not necessary if $A$ or $B$ is diagonal.
\end{theorem}

\begin{remark}\label{sharp}
The moment condition is actually sharp in the following sense. If the condition (\ref{tail_cond}) does not hold for $X$, then one can show that (see e.g. \cite[Section 4]{DY}) for any fixed $a > \lambda_+ $, 
\begin{equation}\nonumber
\limsup_{N \rightarrow \infty} \mathbb{P}\left(\lambda_1(XX^*) \geq a\right)>0, 
\end{equation}
where $\lambda_1(XX^*)$ denotes the largest eigenvalue of $XX^*$. Thus if $\min\{\sigma_n,\wt \sigma_N\} \ge \tau$ for some constant $\tau>0$, we then have 
$$\limsup_{N \rightarrow \infty} \mathbb{P}\left(\lambda_1(\mathcal Q_1) \geq a\right)>0$$
for any fixed $a > \lambda_+ $, and the edge universality \eqref{SUFFICIENT} cannot hold.
 \end{remark}
 
 \nc
 \begin{remark}
It is clear that \eqref{SUFFICIENT} gives the edge universality of the largest eigenvalues of separable covariance matrices. However, to the best of our knowledge, so far there is no explicit formula for the limiting distribution of the largest eigenvalue of $\mathcal Q_1$ when $X$ is Gaussian. In an ongoing work \cite{DY2}, we shall prove that the largest eigenvalue of $\mathcal Q_1$ actually converges weakly to the Tracy-Widom distribution. Here we state the precise result we expect to prove in \cite{DY2}, which may be of interest to some readers. 

Recall the proof of Lemma \ref{lambdar_sqrt}. We define $\gamma_0 \equiv \gamma_0(A,B)$ such that
$${\gamma_0^3} = \frac{\partial_z f(\lambda_+,m_{2c}(\lambda_+))}{-\frac{1}{2}\partial_\al^2 f (\lambda_+,m_{2c}(\lambda_+))}\left[ \int \frac{t}{\lambda_+ (1+tm_{2c}(\lambda_+))^2}\pi_A(\dd t)\right]^{2}  = \frac{I_1^2 J_1}{I_2^2 J_3 + I_3 J_2},$$
where we denote
$$I_1 (A,B):= \int \frac{t}{\lambda_+ (1+tm_{2c}(\lambda_+))^2}\pi_A(\dd t), \quad J_1 (A,B):= \int\frac{x}{\lambda_+^2 \left(1  +x m_{1c}(\lambda_+)\right)^2} \pi_B(\dd x),$$
and for $k=2,3$,
$$I_k (A,B):=d_N \int\frac{t^k}{ \lambda_+ \left(1+tm_{2c}(\lambda_+) \right)^k } \pi_A(\dd t), \quad J_k (A,B):= \int\frac{x^k}{ \lambda_+ \left(1+xm_{1c}(\lambda_+) \right)^k } \pi_B(\dd x).$$
Using \eqref{assm3} and \eqref{assm_gap}, it is easy to see that $\gamma_0\sim 1$. Then we have the following result: if $A$ and $B$ satisfy Assumption \ref{assm_big1} and \eqref{assm_gap}, then we have
\begin{equation}\label{TW_sep}
\begin{split}
 & \lim_{N\to \infty} \mathbb{P}^G(\gamma_0(A,B) N^{{2}/{3}}(\lambda_1(A,B) - \lambda_+(A,B)) \leq s)= F_1(s) 
 \quad \text{for all $s\in \mathbb R$,}
  \end{split}
\end{equation}
where $\lambda_1(A,B)$ denotes the largest eigenvalue of $\mathcal Q_1(A,B)=A^{1/2}XBX^*A^{1/2} $, and $F_1$ is the type-1 Tracy-Widom distribution. 
 (\ref{SUFFICIENT}) and \eqref{TW_sep} together show that the distribution of the rescaled largest eigenvalue of $\mathcal Q_1$ converges to the Tracy-Widom distribution if the conditions (\ref{tail_cond}) and \eqref{assm_3rdmoment} hold. In particular, in the case of sample covariance matrices, the condition \eqref{assm_3rdmoment} is not necessary. 
\end{remark}
 \nc


\begin{remark}\label{finite_correlation}
The universality result (\ref{SUFFICIENT}) can be extended to the joint distribution of the $k$ largest eigenvalues for any fixed $k$:
\begin{equation}\label{SUFFICIENT2}
\begin{split}
\lim_{N\to \infty}\left[\mathbb{P}\left( \left(N^{{2}/{3}}(\lambda_i - \lambda_+) \leq s_i\right)_{1\le i \le k} \right) - \mathbb{P}^G\left(\left(N^{{2}/{3}}(\lambda_i - \lambda_+) \leq s_i\right)_{1\le i \le k} \right) \right]=0, 
\end{split}
\end{equation}
for all $s_1 , s_2, \ldots, s_k \in \mathbb R$. Let $H^{GOE}$ be an $N\times N$ random matrix belonging to the Gaussian orthogonal ensemble. The joint distribution of the $k$ largest eigenvalues of $H^{GOE}$, $\mu^{GOE}_1 \ge \ldots \ge \mu_k^{GOE}$, can be written in terms of the Airy kernel for any fixed $k$ \cite{Forr}. In \cite{DY2}, 
we actually show that
\begin{align*} 
\lim_{N\to \infty} \mathbb{P}^G  \left(\left(\gamma_0(A,B) N^{{2}/{3}}(\lambda_i(A,B) - \lambda_+(A,B)) \leq s_i\right)_{1\le i \le k} \right) = \lim_{N\to \infty} \mathbb{P}\left(\left( N^{{2}/{3}}(\mu_i^{GOE} - 2) \leq s_i\right)_{1\le i \le k} \right), 
\end{align*}
for all $s_1 , s_2, \ldots, s_k \in \mathbb R$. Hence (\ref{SUFFICIENT2}) gives a complete description of the finite-dimensional correlation functions of the largest eigenvalues of $\mathcal Q_1$.
 \end{remark}

\begin{remark} 
A key input for the proof of \eqref{SUFFICIENT} is the anisotropic local law for the resolvents in \eqref{def_green}. Our basic strategy is first to prove the anisotropic local law for $\mathcal G_{1,2}$ when $X$ is Gaussian, and then to obtain the anisotropic local law for the general $X$ case through a comparison with the Gaussian case. Without (\ref{assm_3rdmoment}), the comparison argument cannot give the anisotropic local law up to the optimal scale. However, in the case where $A$ or $B$ is diagonal, the condition (\ref{assm_3rdmoment}) is not needed for the comparison argument in \cite{Anisotropic} to work. 
\nc We refer the reader to the discussion following Theorem \ref{LEM_SMALL}, which explains why and where the condition \eqref{assm_3rdmoment} is needed. \nc We will try to remove the assumption (\ref{assm_3rdmoment}) completely in future works.
\end{remark}

\begin{figure}[htb]
\centering
\subfigure[For $X$ satisfying \eqref{tail_cond}.]{\includegraphics[width=8cm]{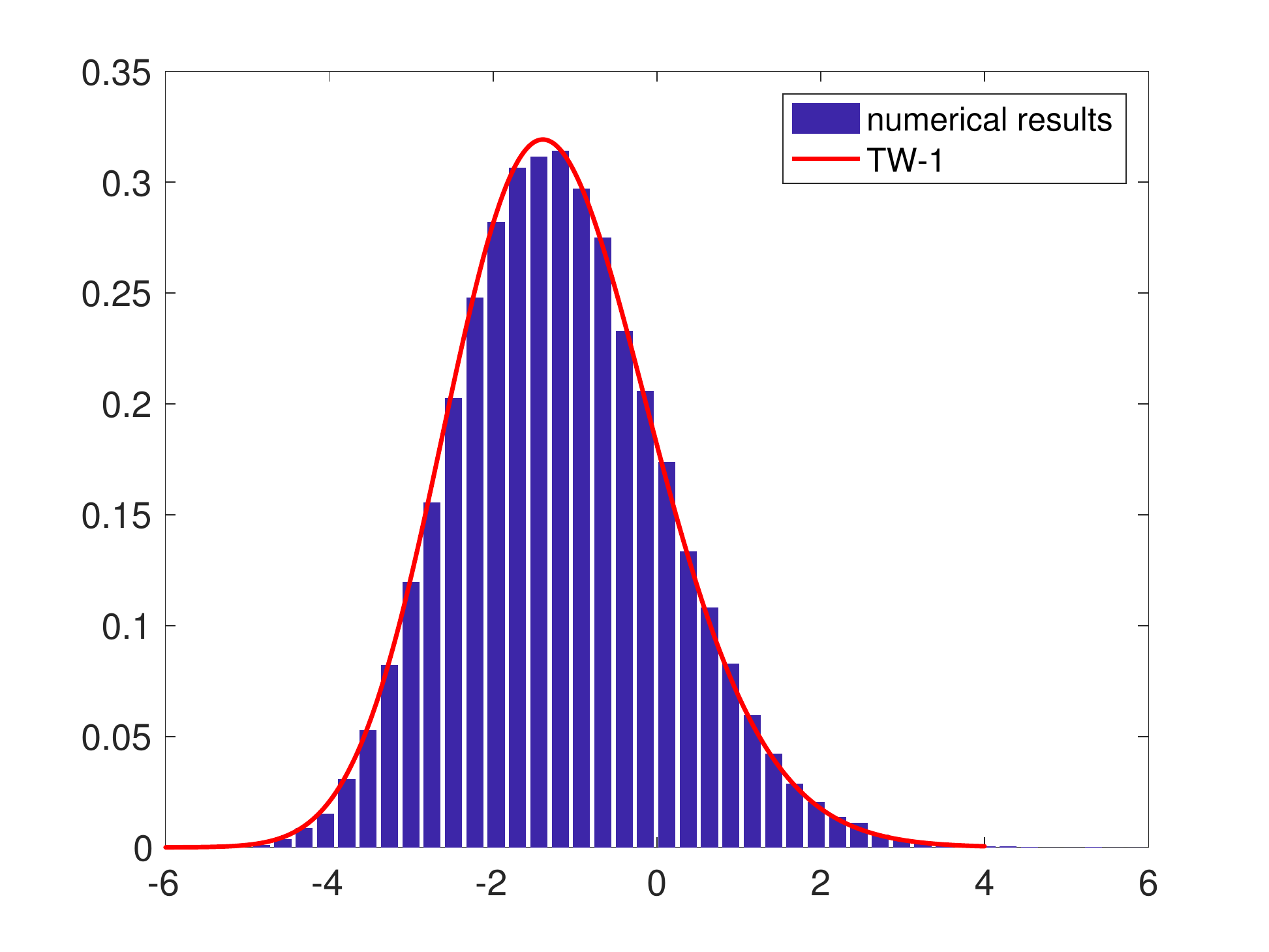}}
\subfigure[For Gaussian $X$.]{\includegraphics[width=8cm]{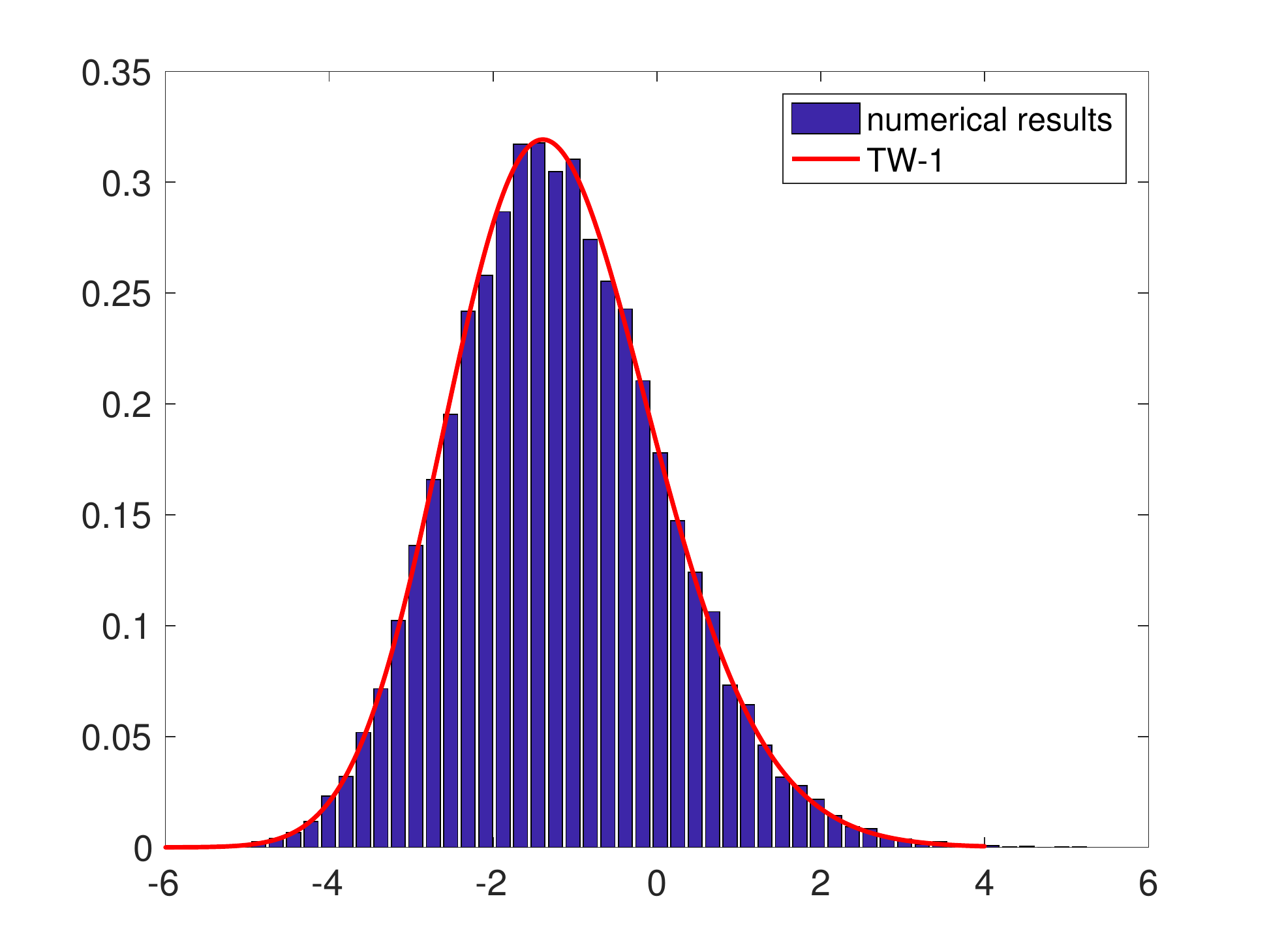}}
\caption{ Histograms for the largest eigenvalues of $20000$ ensembles. }
\label{fig1}
\end{figure}

Finally, we illustrate the edge universality result with some numerical simulations. Consider the following setting: (1) $N= 2n$, i.e. $d_N=0.5$; (2) we take
\begin{equation}\nonumber
\Sigma = {\text{diag}}(\underbrace{1,\cdots, 1}_{n/2},\underbrace{4,\cdots, 4}_{n/2}), \quad \wt\Sigma= {\text{diag}}(\underbrace{1,\cdots, 1}_{N/2},\underbrace{4,\cdots, 4}_{N/2});
\end{equation} 
(3) $U$ and $V$ are orthogonal matrices uniformly chosen from orthogonal groups $\mathbf O(n)$ and $\mathbf O(N)$. Then we take $n=1000$ and calculate the largest eigenvalues for $20000$ independently chosen matrices. The histograms are plotted in Fig.\;\ref{fig1}. In case (a), the entries $\sqrt{N}x_{ij}$ are drawn independently from a distribution with mean zero, variance 1 and satisfying \eqref{tail_cond}; in case (b), the entries $\sqrt{N}x_{ij}$ are $i.i.d.$\;Gaussian with mean zero and variance 1. We translate and rescale the numerical results properly, and one can observe that they fit the type-1 Tracy-Widom distribution very well.

\subsection{Statistical applications}

In this subsection, we briefly discuss some applications of our result to high-dimensional statistics. 

If we take $B=I$, then $\mathcal Q_1$ becomes the normal sample covariance matrix and Theorem \ref{main_thm} indicates that the edge universality of the largest eigenvalue of $\mathcal Q_1$ holds true for correlated data (i.e. non-diagonal $A$) with heavy tails as in (\ref{tail_cond}). So far, this is the strongest edge universality for sample covariance matrices compared with \cite{BPZ1,LS} (assuming high moments and diagonal $A$), \cite{Anisotropic} (assuming high moments) and \cite{DY} (assuming diagonal $A$). 
On the other hand, the separable data model $Y=A^{1/2}XB^{1/2}$ for some nontrivial $B$ is widely used in spatio-temporal data modeling, where $A$ is the spatial covariance matrix and $B$ is the temporal covariance matrix. 
If the entries of $X$ are symmetrically distributed and the singular values of $A,B$ are such that \eqref{assm_gap} holds, then Theorem \ref{main_thm} shows that the largest eigenvalue of $\mathcal Q_1$ satisfies the edge universality as long as (\ref{tail_cond}) holds. We now describe some possible applications of this result. 


Consider the following standard signal plus noise model in classic signal processing \cite{SMK}:
\begin{equation} \label{model_application}
\mathbf{y}=\Gamma \mathbf{s}+A^{1/2}\mathbf{x},
\end{equation}
where $\Gamma$ is an $n \times k$ deterministic matrix, $\mathbf{s}$ is a $k$-dimensional centered signal vector, $A$ is an $n\times n$ deterministic positive definite matrix, and $\mathbf{x}$ is an $n$-dimensional noise vector with $i.i.d.$ mean zero and variance one entries. Moreover, the signal vector and the noise vector are assumed to be independent. In practice, suppose we observe $N$ such samples, where the observations at different times are correlated such that the correlations are independent of the spatial locations. Denoting the temporal covariance matrix by $B$, we then have the spatio-temporal data matrix
$$Y= \Gamma S B^{1/2} + A^{1/2} X B^{1/2},  \quad  S:=(\mathbf s_1, \cdots, \mathbf s_N), \quad X:=(\mathbf x_1, \cdots, \mathbf x_N).$$ 
A fundamental task is to detect the signals via observed samples, and the very first step is to know whether there exists any such signal, i.e., 
\begin{equation}\label{model_null0}
\mathbf{H}_0: \ k=0 \quad \text{vs.} \quad \mathbf{H}_1: \ k \geq 1.
\end{equation}
For the above hypothesis testing problem \eqref{model_null0}, the largest eigenvalue of the observed samples serves as a natural choice for the tests: our result shows that, for {\it heavy-tailed correlated} data satisfying (\ref{tail_cond}), the largest singular value of $Y$ satisfies the Tracy-Widom distribution asymptotically under $\mathbf H_0$.

We can also consider to test whether the space-time data follows a specific separable covariance model with spatial and time covariance matrices $\wt A$ and $\wt B$.
Then we can use the largest singular value of $\wt A^{-1/2}Y\wt B^{-1/2}$ as a test static. Another interesting test static for this hypothesis testing problem is the eigenvector empirical spectral distribution (VESD); see \cite{XYY_VESD,XYZ2013,yang_thesis}. The convergence of VESD for separable covariance matrices has been proved in \cite{yang_thesis} using the anisotropic local law---Theorem \ref{LEM_SMALL} in this paper (which also serves as an important tool for the proof of Theorem \ref{main_thm}). 


Finally, we remark that one can also perform principal component analysis for separable covariance matrices, and study the phase transition phenomena caused by a few large isolated eigenvalues of $A$ or $B$ as in the case of spiked covariance matrices \cite{BBP,  BS_spike, principal, DP_spike}. We expect that our edge universality result will serve as an important input for the study of the  eigenvalues and eigenvectors for the principal components (the outliers) and the bulk components (the non-outliers). For example, in \cite{DY3} we studied the convergence of the outlier eigenvalues and eigenvectors, and the limiting distribution of extremal bulk eigenvalues for the spiked separable covariance model based on our main result, Theorem \ref{main_thm}, and the results given in Section \ref{sec_tools} below.


\section{Basic notations and tools}\label{sec_maintools}

\subsection{Notations}

We will use the following notion of stochastic domination, which was first introduced in \cite{Average_fluc} and subsequently used in many works on random matrix theory, such as \cite{isotropic,principal,local_circular,Delocal,Semicircle,Anisotropic}. It simplifies the presentation of the results and proofs by systematizing statements of the form ``$\xi$ is bounded by $\zeta$ with high probability up to a small power of $N$".

\begin{definition}[Stochastic domination]\label{stoch_domination}
(i) Let
\[\xi=\left(\xi^{(N)}(u):N\in\bbN, u\in U^{(N)}\right),\hskip 10pt \zeta=\left(\zeta^{(N)}(u):N\in\bbN, u\in U^{(N)}\right)\]
be two families of nonnegative random variables, where $U^{(N)}$ is a possibly $N$-dependent parameter set. We say $\xi$ is stochastically dominated by $\zeta$, uniformly in $u$, if for any fixed (small) $\epsilon>0$ and (large) $D>0$, 
\[\sup_{u\in U^{(N)}}\bbP\left[\xi^{(N)}(u)>N^\epsilon\zeta^{(N)}(u)\right]\le N^{-D}\]
for large enough $N\ge N_0(\epsilon, D)$, and we shall use the notation $\xi\prec\zeta$. Throughout this paper, the stochastic domination will always be uniform in all parameters that are not explicitly fixed (such as matrix indices, and $z$ that takes values in some compact set). Note that $N_0(\epsilon, D)$ may depend on quantities that are explicitly constant, such as $\tau$ in Assumption \ref{assm_big1} and \eqref{assm_gap}. If for some complex family $\xi$ we have $|\xi|\prec\zeta$, then we will also write $\xi \prec \zeta$ or $\xi=\OO_\prec(\zeta)$.

(ii) We extend the definition of $\OO_\prec(\cdot)$ to matrices in the weak operator sense as follows. Let $A$ be a family of random matrices and $\zeta$ be a family of nonnegative random variables. Then $A=\OO_\prec(\zeta)$ means that $\left|\left\langle\mathbf v, A\mathbf w\right\rangle\right|\prec\zeta \| \mathbf v\|_2 \|\mathbf w\|_2 $ uniformly in any deterministic vectors $\mathbf v$ and $\mathbf w$. Here and throughout the following, whenever we say ``uniformly in any deterministic vectors", we mean that ``uniformly in any deterministic vectors belonging to certain fixed set of cardinality $N^{\OO(1)}$".

(iii) We say an event $\Xi$ holds with high probability if for any constant $D>0$, $\mathbb P(\Xi)\ge 1- N^{-D}$ for large enough $N$.
\end{definition}

The following lemma collects basic properties of stochastic domination $\prec$, which will be used tacitly in the proof.

\begin{lemma}[Lemma 3.2 in \cite{isotropic}]\label{lem_stodomin}
Let $\xi$ and $\zeta$ be families of nonnegative random variables.

(i) Suppose that $\xi (u,v)\prec \zeta(u,v)$ uniformly in $u\in U$ and $v\in V$. If $|V|\le N^C$ for some constant $C$, then $\sum_{v\in V} \xi(u,v) \prec \sum_{v\in V} \zeta(u,v)$ uniformly in $u$.

(ii) If $\xi_1 (u)\prec \zeta_1(u)$ and $\xi_2 (u)\prec \zeta_2(u)$ uniformly in $u\in U$, then $\xi_1(u)\xi_2(u) \prec \zeta_1(u)\zeta_2(u)$ uniformly in $u$.

(iii) Suppose that $\Psi(u)\ge N^{-C}$ is deterministic and $\xi(u)$ satisfies $\mathbb E\xi(u)^2 \le N^C$ for all $u$. Then if $\xi(u)\prec \Psi(u)$ uniformly in $u$, we have $\mathbb E\xi(u) \prec \Psi(u)$ uniformly in $u$.
\end{lemma}

\begin{definition}[Bounded support condition] \label{defn_support}
We say a random matrix $X=(x_{ij})$ satisfies the {\it{bounded support condition}} with $q$, if
\begin{equation}
\max_{i,j}\vert x_{ij}\vert \prec q. \label{eq_support}
\end{equation}
Here $q\equiv q(N)$ is a deterministic parameter and usually satisfies $ N^{-{1}/{2}} \leq q \leq N^{- \phi} $ for some (small) constant $\phi>0$. Whenever (\ref{eq_support}) holds, we say that $X$ has support $q$. 
\end{definition}


Next we introduce a convenient self-adjoint linearization trick, which has been proved to be useful in studying the local laws of random matrices of the Gram type \cite{Alt_Gram, AEK_Gram, Anisotropic, XYY_circular}. We define the following $(n+N)\times (n+N)$ self-adjoint block matrix, which is a linear function of $X$:
 \begin{equation}\label{linearize_block}
   H \equiv H(X): = \left( {\begin{array}{*{20}c}
   { 0 } & \Sig^{1/2} U^{*}X V\wt \Sig^{1/2}   \\
   {\wt\Sig^{1/2}V^*X^* U\Sig^{1/2} } & {0}  \\
   \end{array}} \right),
 \end{equation}
Then we define its resolvent (Green's function) as
 \begin{equation}\label{eqn_defG}
 G \equiv G (X,z):= \left[H(X)-\left( {\begin{array}{*{20}c}
   { I_{n\times n}} & 0  \\
   0 & { zI_{N\times N}}  \\
\end{array}} \right)\right]^{-1} , \quad z\in \mathbb C_+ .
 \end{equation}
By Schur complement formula, we can verify that (recall \eqref{def_green})
\begin{align} 
G = \left( {\begin{array}{*{20}c}
   { z\mathcal G_1} & \mathcal G_1 \Sig^{1/2} U^{*}X V\wt \Sig^{1/2}  \\
   {\wt\Sig^{1/2}V^*X^* U\Sig^{1/2} \mathcal G_1} & { \mathcal G_2 }  \\
\end{array}} \right) = \left( {\begin{array}{*{20}c}
   { z\mathcal G_1} & \Sig^{1/2} U^{*}X V\wt \Sig^{1/2} \mathcal G_2   \\
   {\mathcal G_2}\wt\Sig^{1/2}V^*X^* U\Sig^{1/2} & { \mathcal G_2 }  \\ 
\end{array}} \right). \label{green2}
\end{align}
Thus a control of $G$ yields directly a control of the resolvents $\mathcal G_{1,2}$. For simplicity of notations, we define the index sets
\[\mathcal I_1:=\{1,...,n\}, \quad \mathcal I_2:=\{n+1,...,n+N\}, \quad \mathcal I:=\mathcal I_1\cup\mathcal I_2.\]
Then we label the indices of the matrices according to 
$$X= (X_{i\mu}:i\in \mathcal I_1, \mu \in \mathcal I_2), \quad A=(A_{ij}: i,j\in \mathcal I_1),\quad B=(B_{\mu\nu}: \mu,\nu\in \mathcal I_2).$$  
In the rest of this paper, 
we will consistently use the latin letters $i,j\in\mathcal I_1$, greek letters $\mu,\nu\in\mathcal I_2$, and $a,b\in\mathcal I$. 

Next we introduce the spectral decomposition of $G$. Let
$$\Sig^{1/2} U^{*}X V\wt \Sig^{1/2}  = \sum_{k = 1}^{n\wedge N} {\sqrt {\lambda_k} \xi_k } \zeta _{k}^* ,$$
be a singular value decomposition of $\Sig^{1/2} U^{*}X V\wt \Sig^{1/2}$, where
$$\lambda_1\ge \lambda_2 \ge \ldots \ge \lambda_{n\wedge N} \ge 0 = \lambda_{n\wedge N+1} = \ldots = \lambda_{n\vee N},$$
$\{\xi_{k}\}_{k=1}^{n}$ are the left-singular vectors, and $\{\zeta_{k}\}_{k=1}^{N}$ are the right-singular vectors.
Then using (\ref{green2}), we can get that for $i,j\in \mathcal I_1$ and $\mu,\nu\in \mathcal I_2$,
\begin{align}
G_{ij} = \sum_{k = 1}^{n} \frac{z\xi_k(i) \xi_k^*(j)}{\lambda_k-z},\ \quad \ &G_{\mu\nu} = \sum_{k = 1}^{N} \frac{\zeta_k(\mu) \zeta_k^*(\nu)}{\lambda_k-z}, \label{spectral1}\\
G_{i\mu} = \sum_{k = 1}^{n\wedge N} \frac{\sqrt{\lambda_k}\xi_k(i) \zeta_k^*(\mu)}{\lambda_k-z}, \ \quad \ &G_{\mu i} = \sum_{k = 1}^{n\wedge N} \frac{\sqrt{\lambda_k}\zeta_k(\mu) \xi_k^*(i)}{\lambda_k-z}.\label{spectral2}
\end{align}

\subsection{Main tools}\label{sec_tools}

For any constants $c_0,C_0>0$ and $\omega \le 1$, we define a domain of the spectral parameter $z$ as
\begin{equation}
S(c_0,C_0,\omega):= \left\{z=E+ \ii \eta: \lambda_+ - c_0 \leq E \leq C_0 \lambda_+, N^{-1+\omega} \leq \eta \leq 1 \right\}. \label{SSET1}
\end{equation}
In particular, we shall denote
\begin{equation}
S(c_0,C_0,-\infty):= \left\{z=E+ \ii \eta: \lambda_+ - c_0 \leq E \leq C_0 \lambda_+, 0 \leq \eta \leq 1 \right\}.
\end{equation}
We define the distance to the rightmost edge as
\begin{equation}
\kappa \equiv \kappa_E := \vert E -\lambda_+\vert  \quad \text{for } z= E+\ii \eta.\label{KAPPA}
\end{equation}
Then we have the following lemma, which summarizes some basic properties of $m_{1,2c}$ and $\rho_{1,2c}$.

\begin{lemma}\label{lem_mbehavior}
Suppose the assumptions \eqref{assm2}, \eqref{assm3} and \eqref{assm_gap} hold. Then
there exists sufficiently small constant $\wt c>0$ such that the following estimates hold:
\begin{itemize}
\item[(1)]
\begin{equation}
\rho_{1,2c}(x) \sim \sqrt{\lambda_+-x}, \quad \ \ \text{ for } x \in \left[\lambda_+ - 2\wt c,\lambda_+ \right];\label{SQUAREROOT}
\end{equation}
\item[(2)] for $z =E+\ii \eta\in S(\wt c,C_0,-\infty)$, 
\begin{equation}\label{Immc}
\vert m_{1,2c}(z) \vert \sim 1,  \quad  \im m_{1,2c}(z) \sim \begin{cases}
    {\eta}/{\sqrt{\kappa+\eta}}, & \text{ if } E\geq \lambda_+ \\
    \sqrt{\kappa+\eta}, & \text{ if } E \le \lambda_+\\
  \end{cases};
\end{equation}
\item[(3)] there exists constant $\tau'>0$ such that
\begin{equation}\label{Piii}
\min_{\mu\in \mathcal I_2} \vert 1 + m_{1c}(z)\wt \sigma_\mu \vert \ge \tau', \quad \min_{i\in \mathcal I_1} \vert 1 + m_{2c}(z)\sigma_i  \vert \ge \tau',
\end{equation}
for any $z \in S(\wt c,C_0,-\infty)$.
\end{itemize}
The estimates \eqref{SQUAREROOT} and \eqref{Immc} also hold for $\rho_c$ and $m_c$. 
\end{lemma}
\begin{proof}
The estimate \eqref{SQUAREROOT} is already given by Lemma \ref{lambdar_sqrt}. The estimate \eqref{Immc} can be proved easily with \eqref{sqroot4}. 
It remains to prove \eqref{Piii}. By assumption \eqref{assm_gap} and the fact $m_{2c}(\lambda_+) \in (-\sigma_1^{-1}, 0)$, we have
$$\left| 1+ m_{2c}(\lambda_+) \sigma_i \right| \ge \tau,  \quad i\in \mathcal I_1.$$
With \eqref{sqroot4}, we see that if $\kappa+\eta \le 2c_0$ for some sufficiently small constant $c_0>0$, then
$$\left| 1+ m_{2c}(z)\sigma_i \right| \ge \tau/2,  \quad i\in \mathcal I_1.$$
Then we consider the case with $E \ge \lambda_+ + c_0$ and $\eta \le c_1$ for some constant $c_1>0$. In fact, for $\eta=0$ and $E> \lambda_+$, $m_{2c}(E)$ is real and it is easy to verify that $m_{2c}'(E)\ge 0$ using 
the Stieltjes transform formula 
\begin{equation}\label{Stj_app}
m_{2c}(z):=\int_{\mathbb R} \frac{\rho_{2c}(\dd x)}{x-z},
\end{equation}
Hence we have
$$ 1+ \sigma_i m_{2c}(E)  \ge 1+ \sigma_i m_{2c}(\lambda_+ ) \ge \tau, \ \ \text{ for }E\ge \lambda_+ + c_0.$$
Using (\ref{Stj_app}) again, we can get that 
$$\left|\frac{\dd m_{2c}(z)}{ \dd z }\right| \le c_0^{-2}, \ \ \text{for } E\ge \lambda_+ + c_0.$$ 
Thus if $c_1$ is sufficiently small, we have
$$\left| 1+ \sigma_i m_{2c}(E+\ii\eta) \right| \ge  \tau/2,  \quad i\in \mathcal I_1,$$
for $E\ge \lambda_+ + c_0$ and $\eta \le c_1$. Finally, it remains to consider the case with $\eta \ge c_1$. In this case, we have $|m_{2c}(z)| \sim \Im \, m_{2c}(z) \sim 1$ by (\ref{Immc}). For $i\in \cal I_1$, if $\sigma_i \le \left|2m_{2c}(z)\right|^{-1}$, then $\left| 1+ \sigma_i m_{2c}(z) \right| \ge 1/2$. Otherwise, we have 
$$\left| 1+ \sigma_i m_{2c}(z) \right| \ge \sigma_i \Im\, m_{2c}(z) \ge \frac{\Im\, m_{2c}(z)}{2 |m_{2c}(z)|}\gtrsim 1 .$$
In sum, we have proved the second estimate in \eqref{Piii}. The first estimate can be proved in a similar way. 
\end{proof}



\begin{definition} [Classical locations of eigenvalues]
The classical location $\gamma_j$ of the $j$-th eigenvalue of $\mathcal Q_1$ is defined as
\begin{equation}\label{gammaj}
\gamma_j:=\sup_{x}\left\{\int_{x}^{+\infty} \rho_{c}(x)dx > \frac{j-1}{n}\right\}.
\end{equation}
In particular, we have $\gamma_1 = \lambda_+$.
\end{definition}

In the rest of this section, we present some results that will be used in the proof of Theorem \ref{main_thm}. Their proofs will be given in subsequent sections. For any matrix $X$ satisfying Assumption \ref{assm_big1} and the tail condition (\ref{tail_cond}), we can construct a matrix $X^s$ that approximates $X$ with probability $1-\oo(1)$, and satisfies Assumption \ref{assm_big1}, the bounded support condition (\ref{eq_support}) with $q\le N^{-\phi}$ for some small constant $\phi>0$, and
\begin{equation}\label{conditionA2}
\mathbb{E}\vert  x^s_{ij} \vert^3 =\OO(N^{-{3}/{2}}), \quad   \mathbb{E} \vert  x^s_{ij} \vert^4  =\OO_\prec (N^{-2});
\end{equation}
see Section \ref{sec_cutoff} for the details. We will need the local laws (Theorem \ref{LEM_SMALL}), eigenvalues rigidity (Theorem \ref{thm_largerigidity}), eigenvector delocalization (Lemma \ref{delocal_rigidity}), and edge universality (Theorem \ref{lem_comparison}) for separable covariance matrices with $X^s$.


We define the deterministic limit $\Pi$ of the resolvent $G$ in (\ref{eqn_defG}) as
\begin{equation}\label{defn_pi}
\Pi (z): = \left( {\begin{array}{*{20}c}
   { -\left(1+m_{2c}(z)\Sigma \right)^{-1} } & 0  \\
   0 & { - z^{-1} (1+m_{1c}(z)\wt \Sigma )^{-1} }  \\
\end{array}} \right) .
\end{equation}
Note that we have
\be\label{mcPi}
\frac1{nz}\sum_{i\in \mathcal I_1} \Pi_{ii} =m_c. 
\ee
Define the control parameters
\begin{equation}\label{eq_defpsi}
\Psi (z):= \sqrt {\frac{\Im \, m_{2c}(z)}{{N\eta }} } + \frac{1}{N\eta}.
\end{equation}
Note that by (\ref{Immc}) and (\ref{Piii}), we have
\begin{equation}\label{psi12}
\|\Pi\|=\OO(1), \quad \Psi \gtrsim N^{-1/2} , \quad \Psi^2 \lesssim (N\eta)^{-1}, \quad \Psi(z) \sim  \sqrt {\frac{\Im \, m_{1c}(z)}{{N\eta }} } + \frac{1}{N\eta},
\end{equation}
for $z\in S(\wt c, C_0,-\infty)$. Now we are ready to state the local laws for $G(X,z)$. For the purpose of proving Theorem \ref{main_thm}, we shall relax the condition \eqref{assm_3rdmoment} a little bit. 


\begin{theorem} [Local laws]\label{LEM_SMALL} 

Suppose Assumption \ref{assm_big1} and \eqref{assm_gap} hold. Suppose $X$ satisfies the bounded support condition (\ref{eq_support}) with $q\le N^{-\phi}$ for some constant $\phi>0$. Furthermore, suppose $X$ satisfies \eqref{conditionA2} and
\be\label{assm_3moment}
\left|\mathbb E x_{ij}^3\right|\le b_N N^{-2}, \quad 1\le i \le n,\ \  1\le j \le N,
\ee
where $b_N$ is an $N$-dependent deterministic parameter satisfying $1 \leq b_N \le N^{1/2}$. Fix $C_0>1$ and let $c_0>0$ be a sufficiently small constant. Given any $\epsilon,\fa>0$, we define the domain
\be \label{tildeS}
\wt S(c_0,C_0,\fa, \e):= S(c_0,C_0,\epsilon) \cap \left\{z = E+ \ii \eta: b_N \left(\Psi^2(z) + \frac{q}{N\eta}\right)\le N^{-\mathfrak a}\right\}.
\ee
Then for any constants $\e>0$ and $\fa>0$, the following estimates hold. 
\begin{itemize}
\item[(1)] {\bf Anisotropic local law}: For any $z\in \wt S(c_0,C_0,\fa,\epsilon)$ and deterministic unit vectors $\mathbf u, \mathbf v \in \mathbb C^{\mathcal I}$,
\begin{equation}\label{aniso_law}
\left| \langle \mathbf u, G(X,z) \mathbf v\rangle - \langle \mathbf u, \Pi (z)\mathbf v\rangle \right| \prec q+ \Psi(z).
\end{equation}

\item[(2)] {\bf Averaged local law}: For any $z \in \wt S(c_0, C_0, \fa,\epsilon)$,  we have
\begin{equation}
 \vert m(z)-m_{c}(z) \vert \prec q^2 + (N \eta)^{-1}. \label{aver_in1} 
\end{equation}
where $m$ is defined in \eqref{defn_m}. Moreover, outside of the spectrum we have the following stronger estimate
\begin{equation}\label{aver_out1}
 | m(z)-m_{c}(z)|\prec q^2 +\nc \frac{N^{-\fa/2}}{N\eta} \nc+ \frac{1}{N(\kappa +\eta)} + \frac{1}{(N\eta)^2\sqrt{\kappa +\eta}},
\end{equation}
uniformly in $z\in \wt S(c_0,C_0,\fa, \epsilon)\cap \{z=E+\ii\eta: E\ge \lambda_+, N\eta\sqrt{\kappa + \eta} \ge N^\epsilon\}$, where $\kappa$ is defined in \eqref{KAPPA}. 
\end{itemize}
The above estimates are uniform in the spectral parameter $z$ and any set of deterministic vectors of cardinality $N^{\OO(1)}$. If $A$ or $B$ is diagonal, then \eqref{aniso_law} and \eqref{aver_in1} hold for $z\in S(c_0,C_0,\epsilon) $, \nc and \eqref{aver_out1} holds for $z\in S(c_0,C_0,\epsilon)\cap \{z=E+\ii\eta: E\ge \lambda_+, N\eta\sqrt{\kappa + \eta} \ge N^\epsilon\}$ without the term ${N^{-\fa/2}}/({N\eta})$. \nc
\end{theorem}

The main difficulty for the proof of Theorem \ref{LEM_SMALL} is due to the fact that the entries of $A^{1/2} XB^{1/2}$ are not independent anymore. However, notice that if $X\equiv X^{Gauss}$ is $i.i.d.$ Gaussian, 
we have 
$$\Sig^{1/2} U^{*}X^{Gauss} V\wt \Sig^{1/2} \stackrel{d}{=}  \Sig^{1/2} X^{Gauss} \wt \Sig^{1/2}.$$
In this case, the problem is reduced to proving the local laws for separable covariance matrices with diagonal spatial and temporal covariance matrices, which can be handled using the standard resolvent methods as in e.g. \cite{isotropic,PY}. To go from the Gaussian case to the general $X$ case, we adopt a continuous self-consistent comparison argument developed in \cite{Anisotropic}. In order for this argument to work, we need to assume 
\eqref{assm_3rdmoment}. \nc The main reason is that we need to match the third moment of $x_{ij}$ with that of the Gaussian random variables in the derivation of equation \eqref{only3rd} below. \nc Under the weaker condition \eqref{assm_3moment}, we cannot prove the local laws up to the optimal scale $\eta \gg N^{-1}$, but only up to the scale $\eta \gg \max\{\frac{qb_N}{N},\frac{\sqrt{b_N}}{N}\}$ near the edge. However, to prove the edge universality, we only need to have a good local law up to the scale $\eta \le N^{-2/3-\e}$, hence $b_N$ can take values up to $b_N \ll N^{1/3}$. (Actually in the proof of Theorem \ref{main_thm} in Section \ref{sec_cutoff}, we will take $b_N=N^{-\e}$ for some small constant $\e>0$; \nc see \eqref{estimate_qs} below for the estimate on $b_N$ that is obtained from \eqref{assm_3rdmoment}\nc.) Finally, if $A$ or $B$ is diagonal, one can prove the local laws up to the optimal scale for all $b_N=\OO( N^{1/2})$ by using an improved comparison argument in \cite{Anisotropic}. 

Following the above discussions, we divide the proof of Theorem \ref{LEM_SMALL} into two steps. In Section \ref{sec_Gauss}, we give the proof for separable covariance matrices of the form $\Sig^{1/2} X \wt \Sig X^* \Sig^{1/2}$, which implies the local laws in the Gaussian $X$ case. In Section \ref{sec_comparison}, we apply the self-consistent comparison argument in \cite{Anisotropic} to extend the result to the general $X$ case. Compared with \cite{Anisotropic}, there are two differences in our setting: (1) the support of $X$ in Theorem \ref{LEM_SMALL} is $q=\OO(N^{-\phi})$ for some constant $0<\phi \le 1/2$, while \cite{Anisotropic} only dealt with $X$ with small support $q=\OO(N^{-1/2})$; (2) one has $B=I$ in \cite{Anisotropic}, which simplifies the proof.

\vspace{5pt}

The second moment of the error $ \langle \mathbf u, (G-\Pi) \mathbf v\rangle $ in fact satisfies a stronger bound.

\begin{lemma}
\label{thm_largebound}
Suppose the assumptions in Theorem \ref{LEM_SMALL} hold. Then for any fixed $\epsilon,\fa>0$ and $z\in \wt S(c_0,C_0,\fa,\epsilon)$,  we have the following bound
\begin{equation}\label{weak_off}
\mathbb{E} \vert \langle \mathbf u , G(X,z)\mathbf v\rangle - \langle \mathbf u , \Pi(z)\mathbf v\rangle  \vert^2 \prec \Psi^2(z), 
\end{equation}
for any deterministic unit vectors $\mathbf u , \mathbf v  \in \mathbb C^{\mathcal I}$.
\end{lemma}



With Theorem \ref{LEM_SMALL} as a key input, we can prove a stronger estimate on $m(z)$ that is independent of $q$. This averaged local law implies the rigidity of eigenvalues for $\mathcal Q_1$. Note that for any fixed $E$, $\Psi^2(E+\ii\eta) + {q}/(N\eta)$ is monotonically decreasing with respect to $\eta$, hence there is a unique $\eta_1(E)$ such that 
$$b_N \left(\Psi^2(E+\ii\eta_1(E)) + \frac{q}{N\eta_1(E)}\right) =1.$$
Then we define $\eta_l(E):=\max_{E\le x \le \lambda_+} \eta_1(x)$ (``$l$" for lower bound) for $E\le \lambda_+$, and $\eta_l(E):= \eta_l(\lambda_+)$ for $E>\lambda_+$. Note that by \eqref{eq_defpsi}, we always have $\eta_l(E)=\OO( b_N/N)$.

\begin{theorem}[Rigidity of eigenvalues] \label{thm_largerigidity}
Suppose the assumptions in Theorem \ref{LEM_SMALL} hold. 
Fix the constants $c_0$ and $C_0$ as given in Theorem \ref{LEM_SMALL}. Then for any fixed $\epsilon,\fa>0$, we have
\begin{equation}
 \vert m(z)-m_{c}(z) \vert \prec (N \eta)^{-1}, \label{aver_in}
\end{equation}
uniformly in $z \in  \wt S(c_0, C_0, \fa,\epsilon)$. Moreover, outside of the spectrum we have the following stronger estimate 
\begin{equation}\label{aver_out0}
 | m(z)-m_{c}(z)|\prec \nc \frac{N^{-\fa/2}}{N\eta} \nc + \frac{1}{N(\kappa +\eta)} + \frac{1}{(N\eta)^2\sqrt{\kappa +\eta}},
\end{equation}
uniformly in $z\in \wt S(c_0,C_0,\fa,\epsilon)\cap \{z=E+\ii\eta: E\ge \lambda_+, N\eta\sqrt{\kappa + \eta} \ge N^\epsilon\}$ for any fixed $\epsilon>0$.  If $A$ or $B$ is diagonal, then \eqref{aver_in} holds for $z\in S(c_0,C_0,\epsilon) $ and \nc\eqref{aver_out0} holds for $z\in S(c_0,C_0,\epsilon)\cap \{z=E+\ii\eta: E\ge \lambda_+, N\eta\sqrt{\kappa + \eta} \ge N^\epsilon\}$ without the term ${N^{-\fa/2}}/({N\eta})$\nc. The bounds (\ref{aver_in}) and \eqref{aver_out0} imply that for any constant $0<c_1<c_0$, the following estimates hold.

\begin{itemize}
\item[(1)] For any $E\ge \lambda_+ - c_1 $, we have
\begin{equation}
 \vert n(E)-n_{c}(E) \vert \prec N^{-1} + (\eta_l(E))^{3/2} + \eta_l(E)\sqrt{\kappa_E} ,  \label{Kdist}
\end{equation}
where $\kappa_E$ is defined in \eqref{KAPPA}, and
\begin{equation}\label{ncE}
n(E):=\frac{1}{N} \# \{ \lambda_j \ge E\}, \ \ n_{c}(E):=\int^{+\infty}_E \rho_{2c}(x)dx.
\end{equation}

\item[(2)] If $b_N\le N^{1/3-c}$ for some constant $c>0$, then for any $j$ such that $\lambda_+ - c_1 \le \gamma_j \le \lambda_+$, we have
\begin{equation}\label{rigidity}
\vert \lambda_j - \gamma_j \vert \prec  j^{-1/3}N^{-2/3} + \eta_0 ,
\end{equation}
where $\eta_0 := \eta_l(\lambda_+ - c_1) =\OO( b_N/N)$. 
\end{itemize}
\end{theorem}


The anisotropic local law (\ref{aniso_law}) implies the following delocalization properties of eigenvectors. 

\begin{lemma}[Isotropic delocalization of eigenvectors] \label{delocal_rigidity}
Suppose \eqref{aniso_law} and \eqref{rigidity} hold. Then for any deterministic unit vectors $\mathbf u \in \mathbb C^{\mathcal I_1}$, $\mathbf v  \in \mathbb C^{\mathcal I_2}$ and constant $0<c_1<c_0$, we have
\begin{equation}
\max_{k: \lambda_+ - c_1 \le \gamma_k \le \lambda_+} \left\{ \left|\langle \mathbf u,\xi_k\rangle \right|^2+\left|\langle \mathbf v ,\zeta_k\rangle \right|^2\right\} \prec \eta_0,\label{delocal}
\end{equation}
where $\eta_0$ is defined below \eqref{rigidity}.
\end{lemma}

\begin{proof}
Choose $z_0=E+\ii N^\e \eta_0\in \wt S(c_0, C_0, \fa,\epsilon)$. 
By (\ref{aniso_law}) and \eqref{psi12}, we have $\im  \langle \mathbf v, G(z_0) \mathbf v\rangle = \OO(1) $ with high probability. Then using the spectral decomposition (\ref{spectral1}), we get
\begin{equation}\label{spectraldecomp}
\sum_{k=1}^N \frac{N^\e \eta_0 \vert \langle \mathbf v, \zeta_k\rangle \vert^2}{(\lambda_k-E)^2+N^{2\e}\eta_0^2}  = \im\, \langle \mathbf v, {G}(z_0)\mathbf v\rangle =  \OO(1)  \quad \text{ with high probability.}
\end{equation}
By (\ref{rigidity}), we have that $\lambda_k + \ii N^\e \eta_0 \in \wt S(c_0, C_0, \fa, \epsilon)$ with high probability for every $k$ such that $\lambda_+ - c_1 \le \gamma_k \le \lambda_+$. Then choosing $E=\lambda_k$ in (\ref{spectraldecomp}) yields that
\begin{equation*}
\vert \langle \mathbf v, \zeta_k\rangle \vert^2 \lesssim N^\e \eta_0 \quad \text{with high probability.}
\end{equation*} 
Since $\epsilon$ is arbitrary, we get $\vert \langle \mathbf v, \zeta_k\rangle \vert^2  \prec \eta_0$. In a similar way, we can prove $\left|\langle \mathbf u,\xi_k\rangle \right|^2 \prec \eta_0$.
\end{proof}

Finally, we have the following edge universality result for separable covariance matrices with support $q\le N^{-\phi}$ and satisfying the condition (\ref{conditionA2}).

\begin{theorem}\label{lem_comparison}
Let $X^{(1)}$ and $X^{(2)}$ be two separable covariance matrices satisfying the assumptions in Theorem \ref{LEM_SMALL}. 
Suppose $b_N \le N^{1/3-c}$ for some constant $c>0$. Then there exist constants $\epsilon,\delta >0$ such that for any $s\in \mathbb R$, 
\be\label{EDDDD} 
\begin{split}
\mathbb{P}^{(1)} \left(N^{{2}/{3}}(\lambda_1-\lambda_{+}) \leq s-N^{-\epsilon}\right)-N^{-\delta} \leq \mathbb{P}^{(2)}\left(N^{{2}/{3}}(\lambda_1-\lambda_{+})\leq s\right) &\\
 \leq \mathbb{P}^{(1)}\left(N^{{2}/{3}}(\lambda_1-\lambda_{+}) \leq s+N^{-\epsilon}\right)+N^{-\delta} &, 
 \end{split}
\ee
where $\mathbb{P}^{(1)}$ and $\mathbb{P}^{(2)}$ denote the laws of $X^{(1)}$ and $X^{(2)}$, respectively. 
\end{theorem}

\begin{remark} \label{rigid_multi}
As in \cite{EKYY,EYY,LY}, Theorem \ref{lem_comparison} can be can be generalized to finite correlation functions of the $k$ largest eigenvalues for any fixed $k$:
\begin{equation}\label{EDDDD_ext}
\begin{split}
 \mathbb{P}^{(1)} \left( \left(N^{{2}/{3}}(\lambda_i-\lambda_{+}) \leq s_i -N^{-\epsilon}\right)_{1\le i \le k}\right)-N^{-\delta} \le \mathbb{P}^{(2)} \left( \left(N^{{2}/{3}}(\lambda_i-\lambda_{+})  \leq s_i \right)_{1\le i \le k}\right) & \\
  \le \mathbb{P}^{(1)} \left( \left(N^{{2}/{3}}(\lambda_i-\lambda_{+}) \leq s_i + N^{-\epsilon}\right)_{1\le i \le k}\right)+ N^{-\delta} &. \end{split}
\end{equation}
The proof of (\ref{EDDDD_ext}) is similar to that of (\ref{EDDDD}) except that it uses a general form of the Green function comparison theorem; see e.g.\;\cite[Theorem 6.4]{EYY}. As a corollary, we can get the stronger edge universality result (\ref{SUFFICIENT2}).
\end{remark}

The proofs for Lemma \ref{thm_largebound}, Theorem \ref{thm_largerigidity} and Theorem \ref{lem_comparison} follow essentially the same path as discussed below. First, for random matrix $\wt X$ with small suppoort $q=\OO(N^{-1/2})$, we have the averaged local laws \eqref{aver_in}-\eqref{aver_out0} and the following anisotropic local law 
$$\left| \langle \mathbf u, G(\wt X,z) \mathbf v\rangle - \langle \mathbf u, \Pi (z)\mathbf v\rangle \right| \prec \Psi(z).$$
With these estimates, one can prove that Lemma \ref{thm_largebound}, Theorem \ref{thm_largerigidity} and Theorem \ref{lem_comparison} hold in the small support case using the methods in e.g. \cite{EKYY,EYY,PY}. Then it suffices to use a comparison argument to show that the large support case is ``sufficiently close" to the small support case. In fact, given any matrix $X$ satisfying the assumptions in Theorem \ref{LEM_SMALL}, we can construct a matrix $\wt X$ having the same first four moments as $X$ but with smaller support $q=\OO(N^{-1/2})$, which is the content of the next lemma.

\begin{lemma} [Lemma 5.1 in \cite{LY}]\label{lem_decrease}
Suppose $X$ satisfies the assumptions in Theorem \ref{LEM_SMALL}. Then there exists another matrix $\wt{X}=(\wt x_{ij})$, such that $\wt{X}$ satisfies the bounded support condition (\ref{eq_support}) with $q=N^{-1/2}$, and the first four moments of the $X$ entries and $\wt{X}$ entries match, i.e.
\begin{equation}\label{match_moments}
\mathbb Ex_{ij}^k =\mathbb E\wt x_{ij}^k, \ \ k=1,2,3,4.
\end{equation}
\end{lemma}

It is known that the Lindeberg replacement strategy combined with the four moment matching usually implies some universality results in random matrix theory, see e.g.\;\cite{TaoVu_edge,TaoVu_4moment,TaoVu_local}. This is actually also true in our case. We shall extend the Green function comparison method developed in \cite{LY} (which is essentially an iterative application of the Lindeberg strategy using the four moment matching), and prove that Lemma \ref{thm_largebound}, Theorem \ref{thm_largerigidity} and Theorem \ref{lem_comparison} also hold for the large support case. The proofs are given in Section \ref{sec_Lindeberg}. 

\section{Proof of of Theorem \ref{main_thm}}\label{sec_cutoff}

In this section, we prove Theorem \ref{main_thm} with the results in Section \ref{sec_tools}. 
Given the matrix $X$ satisfying Assumption \ref{assm_big1} and the tail condition (\ref{tail_cond}), we introduce a cutoff on its matrix entries at the level $N^{-\epsilon}$. For any fixed $\epsilon>0$, define
\begin{equation*}
\alpha_N:=\mathbb{P}\left(\vert q_{11} \vert > N^{{1}/{2}-\epsilon}\right), \ \ \beta_N:=\mathbb{E}\left[\mathbf{1}{\left(|q_{11}|> N^{{1}/{2}-\epsilon}\right)}q_{11} \right].
\end{equation*}
By (\ref{tail_cond}) and integration by parts, we get that for any fixed $\delta>0$ and large enough $N$,
\begin{equation}
\alpha_N \leq \delta N^{-2+4\epsilon}, \ \ \vert \beta_N \vert \leq \delta N^{-{3}/{2}+3 \epsilon} . \label{BBBOUNDS}
\end{equation}
\nc Let $\rho(\dd x)$ be the law of $q_{11}$. \nc Then we define independent random variables $q_{ij}^s$, $q_{ij}^l$, $c_{ij}$, $1\le i \le n$ and $1\le j \le N$, in the following ways.
\begin{itemize}
\item \nc $q_{ij}^s$ has law $\rho_s$, which is defined such that
\begin{equation}
 \rho_s(\cal E)= \frac1{1-\alpha_N}\int \mathbf 1\left( x+\frac{\beta_N}{1-\al_N} \in \cal E \right)\mathbf{1}\left(\left| x \right| \leq N^{{1}/{2}-\epsilon} \right)  \rho(\dd x) \nonumber
\end{equation}
for any event $\cal E$. Note that if $q_{11}$ has density $\rho(x)$, then the density for $q_{11}^s$ is 
\begin{equation}
\rho_s(x)= \mathbf{1}\left(\left| x-\frac{\beta_N}{1-\alpha_N}  \right| \leq N^{{1}/{2}-\epsilon} \right) \frac{\rho\left(x-\frac{\beta_N}{1-\alpha_N}\right)}{1-\alpha_N}. \nonumber
\end{equation}
\item \nc $q_{ij}^l$ has law $\rho_l$, such that 
\begin{equation}
 \rho_l(\cal E)= \frac1{\alpha_N}\int \mathbf 1\left( x+\frac{\beta_N}{1-\al_N} \in \cal E \right)\mathbf{1}\left(\left| x \right| > N^{{1}/{2}-\epsilon} \right)  \rho(\dd x) \nonumber
\end{equation}
for any event $\cal E$.\nc
\item $c_{ij}$ is a Bernoulli 0-1 random variable with $\mathbb{P}(c_{ij}=1)=\alpha_N$ and $\mathbb{P}(c_{ij}=0)=1-\alpha_N$.
\end{itemize}
Let $X^s$, $X^l$ and $X^c$ be random matrices such that $X^s_{ij} = N^{-1/2}q_{ij}^s$, $X^l_{ij} = N^{-1/2}q_{ij}^l$ and $X^c_{ij} = c_{ij}$.
It is easy to check that for independent $X^s$, $X^l$ and $X^c$,
\begin{equation}
X_{ij} \stackrel{d}{=} X^s_{ij}\left(1-X^c_{ij}\right)+X^l_{ij}X^c_{ij} - \frac{1}{\sqrt{N}}\frac{\beta_N}{1-\alpha_N}. \label{T3}
\end{equation}
\nc The purpose of this decomposition (in distribution) is to write $X$ into a well-behaved random matrix $X^s$ with bounded support $q=\OO(N^{-\e})$ plus a perturbation matrix $(X^l- X^s) X^c$. Here the matrix $X^c$ gives the locations of the nonzero entries of the perturbation matrix, and its rank is at most $N^{5\e}$ with high probability; see \eqref{LDP_B} below. The matrix $X^l$ contains the ``abnormal" large entries above the cutoff, but the tail condition \eqref{tail_cond} guarantees that the sizes of these entries are of order $\oo(1)$ in probability; see \eqref{prob_R}. Hence the perturbation $(X^l- X^s) X^c$ is of low rank and has small strengths. Then as in the famous BBP transition \cite{BBP}, we will show that the effect of this perturbation on the largest eigenvalue is negligible.   \nc

If we define the $n\times N$ matrix $Y=(Y_{ij})$ by 
$$Y_{ij}=\frac{1}{\sqrt{N}}\frac{\beta_N}{1-\alpha_N}=\OO(\delta N^{-2+3\epsilon}), \quad 1\le i \le n,\ \ 1\le  j \le N,$$
then we have $ \| Y \| =\OO (N^{-1+3\epsilon}) $. 
In the proof below, one will see that (recall \eqref{Qtilde})
$$ \left\| \Sig^{1/2} U^{*}(X+Y) V\wt \Sig^{1/2} \right\| = \lambda^{1/2}_1\left(\wt {\mathcal Q}_1(X+Y)\right) =\OO(1)$$
with probability $1-\oo(1)$. 
Thus with probability $1-\oo(1)$, we have 
\begin{align}\label{const_err}
\left|\lambda_1\left(\wt {\mathcal Q}_1(X+Y)\right) - \lambda_1\left(\wt {\mathcal Q}_1(X)\right)\right| = \OO\left( N^{-1+3\epsilon} \right).
\end{align}
Hence the deterministic part in (\ref{T3}) is negligible under the scaling $N^{2/3}$.

By (\ref{tail_cond}), \eqref{assm_3rdmoment} and integration by parts, it is easy to check that
\begin{align}\label{estimate_qs}
\mathbb{E} q^{s}_{11} =0, \ \ \mathbb{E}\vert q^s_{11}\vert^2=1-\OO(N^{-1+2 \epsilon}), \ \ \mathbb{E}|q^s_{11}|^3 = \OO(1),  \ \ \mathbb{E}( q^s_{11})^3 = \OO(N^{-1/2+\e}),  \ \ \mathbb{E}\vert q^s_{11} \vert^4=\OO(\log N).
\end{align}
\nc Note that this is the only place where \eqref{assm_3rdmoment} is used in order to get the estimate on $\mathbb{E}( q^s_{11})^3$. For the reason why this estimate is needed, we refer the reader to the discussion below Theorem \ref{LEM_SMALL}. \nc
Thus $X_1:=(\mathbb{E}\vert q^s_{11} \vert^2)^{-{1}/{2}}X^s$ is a matrix that satisfies the assumptions for $X$ in Theorem \ref{LEM_SMALL} with $b_N=\OO(N^\e)$ and $q=\OO(N^{-\e})$. 
Then by Theorem \ref{lem_comparison}, there exist constants $\epsilon',\delta'>0$ such that for any $s\in \mathbb R$,
\be\label{univ_small}
\begin{split}
\mathbb{P}^G \left( N^{{2}/{3}}(\lambda_1-\lambda_{+}) 
 \leq  s-N^{-\epsilon'}\right)-N^{-\delta'} \leq
\mathbb{P}^s\left(N^{{2}/{3}}(\lambda_1-\lambda_{+})\leq s \right) & \\
\leq \mathbb{P}^G\left(N^{{2}/{3}}(\lambda_1-\lambda_{+}) \leq s+N^{-\epsilon'}\right)+N^{-\delta'} &,
\end{split}
\ee
where $\mathbb P^s$ denotes the law for $X^s$ and $\mathbb P^G$ denotes the law for $i.i.d.$ Gaussian matrix. Now we write the first two terms on the right-hand side of (\ref{T3}) as
$$X^s_{ij}(1-X^c_{ij})+X^l_{ij}X^c_{ij} = X^s_{ij} + R_{ij} X^c_{ij}, \quad R_{ij}:=X^l_{ij}-X^s_{ij}.$$
We define the matrix $R^c :=(R_{ij} X^c_{ij})$. It remains to show that the effect of $R^c$ on $\lambda_1$ is negligible. Note that $X^c_{ij}$ is independent of $X^s_{ij}$ and $R_{ij}$. 

We first introduce a cutoff on matrix $X^c$ as $\wt X^c : =\mathbf{1}_{\mathscr A}X^c$, where
\begin{align*}
\mathscr A:=& \left\{\#\{(i,j):X^c_{ij}=1\}\leq N^{5\epsilon}\right\} \cap \left\{X^c_{ij}=X^c_{kl}=1  {\Rightarrow} \{i,j\}=\{k,l\} \ \text{or} \  \{i,j\} \cap \{k,l\}=\emptyset \right\}.
\end{align*}
If we regard the matrix $X^c$ as a sequence $\mathbf X^c$ of $nN$ {\it{i.i.d.}} Bernoulli random variables, it is easy to obtain from the large deviation formula that
\begin{equation}\label{LDP_B}
\mathbb{P}\left(\sum_{i=1}^{nN} \mathbf X^c_i \leq N^{5 \epsilon}\right) \geq 1- \exp(- N^{\epsilon}),
\end{equation}
for sufficiently large $N$. Suppose the number $n_0$ of the nonzero elements in $X^c$ is given with $n_0 \le N^{5\epsilon} $. Then it is easy to check that
\begin{align}\label{LDP_C}
\mathbb P\left(\exists \, i = k, j\ne l \ \text{or} \  i\ne k, j =l \text{ such that } X^c_{ij}=X^c_{kl}=1 \left| \sum_{i=1}^{nN} \mathbf X^c_i = n_0 \right. \right) = \OO(n_0^2N^{-1}). 
\end{align}
Combining the estimates (\ref{LDP_B}) and (\ref{LDP_C}), we get that 
\begin{equation}\label{prob_A}
\mathbb P(\mathscr A) \ge 1 -  \OO(N^{-1+10\epsilon}).
\end{equation}
On the other hand, by condition (\ref{tail_cond}), we have
\begin{equation}\label{prob_R}
\mathbb{P}\left(|R_{ij}| \geq \omega \right) \leq \mathbb{P}\left(|q_{ij}| \geq \frac{\omega}{2}N^{1/2}\right) = \oo(N^{-2}) ,
\end{equation}
for any fixed constant $\omega>0$. Hence if we introduce the matrix 
$$E = \mathbf 1\left(\mathscr A\cap \left\{\max_{i,j} |R_{ij}| \leq \omega \right\}\right) R^c,$$
then we have 
\begin{equation}\label{EneR}
\mathbb P( E = R^c)=1-\oo(1)
\end{equation}
by (\ref{prob_A}) and (\ref{prob_R}). Thus we only need to study the largest eigenvalue of $\wt {\mathcal Q}_1(X^s+E)$, where $\max_{i,j} |E_{ij}| \le \omega$ and $\text{rank}(E) \le N^{5\epsilon}$. In fact, it suffices to prove that
\begin{equation}\label{eq_claim}
\mathbb P\left(\left|\lambda_1^s - \lambda_1^E\right| \le N^{-3/4}\right) = 1-\oo(1), \quad \lambda_1^s:=\lambda_1\left(\wt {\mathcal Q}_1(X^s) \right),\quad  \lambda_1^E:=\lambda_1\left(\wt {\mathcal Q}_1(X^s+E) \right).
\end{equation}
The estimate (\ref{eq_claim}), combined with (\ref{const_err}), (\ref{univ_small}) and (\ref{EneR}), concludes (\ref{SUFFICIENT}).

Now we prove (\ref{eq_claim}). Since $ \wt{X}^c $ is independent of $X^s$, \nc the positions of the nonzero elements of $\wt X^c$ are independent of $X^s$\nc. \nc Without loss of generality, we assume the positions of the $n_0$ nonzero entries of $\wt X^c$ are $(1,1), (2,2),\cdots, (n_0,n_0)$, which correspond to the following entries of $E$: \nc
\begin{equation}
e_{11}, \ e_{22}, \ \cdots,  \ e_{n_0n_0}, \ \ n_0 \le N^{5\epsilon}. \label{STRUCTURE}
\end{equation}
For other choices of the positions of nonzero entries, the proof is exactly the same, but we make this assumption to simplify the notations. By the definition of $E$, we have $|e_{ii}| \le \omega$, $1\le i \le n_0$. We define the matrices
$$  H^s : = \left( {\begin{array}{*{20}c}
   { 0 } & \Sig^{1/2} U^{*}X^s V\wt \Sig^{1/2}  \\
   {(\Sig^{1/2} U^{*}X^s V\wt \Sig^{1/2})^*} & {0}  \\
   \end{array}} \right)$$ 
and $H^{E} : = H^s + P$, where
\begin{align*}
P: &= \left( {\begin{array}{*{20}c}
   { 0 } & \Sig^{1/2} U^{*}E V\wt \Sig^{1/2}   \\
   { (\Sig^{1/2} U^{*}E V\wt \Sig^{1/2} )^*} & {0}  \\
   \end{array}} \right)= \left( {\begin{array}{*{20}c}
   {\Sig^{1/2} U^{*}} & 0  \\
   { 0} & {\wt \Sig^{1/2}V^*}  \\
   \end{array}} \right)\left( {\begin{array}{*{20}c}
   { 0 } & E  \\
   { E^*} & {0}  \\
   \end{array}} \right)\left( {\begin{array}{*{20}c}
   {U\Sig^{1/2} } & 0  \\
   { 0} & {V\wt \Sig^{1/2}}  \\
   \end{array}} \right)\\
  & = \left( {\begin{array}{*{20}c}
   {\Sig^{1/2} U^{*}} & 0  \\
   { 0} & {\wt \Sig^{1/2}V^*}  \\
   \end{array}} \right)W P_DW^* \left( {\begin{array}{*{20}c}
   {U\Sig^{1/2} } & 0  \\
   { 0} & {V\wt \Sig^{1/2}}  \\
   \end{array}} \right),
   \end{align*}
where $P_D$ is a $2n_0 \times 2n_0$ diagonal matrix
$$P_D= \text{diag}\left(e_{11}, \ldots, e_{n_0n_0}, -e_{11},\ldots, -e_{n_0n_0}\right),$$
and $W$ is an $(n+N)\times 2n_0$ matrix such that
$$W_{ab}= \begin{cases}
\delta_{a,i}/\sqrt{2} + \delta_{a,(n+i)}/\sqrt{2}, \ &b = i , \, i\le n_0\\
\delta_{a,i}/\sqrt{2} - \delta_{a,(n+i)}/\sqrt{2}, \ &b = i+n_0 , \, i\le n_0
\end{cases}.$$
\nc Without loss of generality, we assume that $e_{ii} \ne 0$, $1\le i \le n_0$ (otherwise we only need to use a matrix $W$ with smaller rank). \nc
With the identity
$$\det\left( {\begin{array}{*{20}c}
   { - I_{n\times n}} & \Sig^{1/2} U^{*}X V\wt \Sig^{1/2}  \\
   {(\Sig^{1/2} U^{*}X V\wt \Sig^{1/2})^*} & { - zI_{N\times N}}  \\
\end{array}} \right) =(-1)^Nz^{N-n}\det\left(\wt{\mathcal Q}_1(X) - zI_{n\times n}\right), $$
and Lemma 6.1 of \cite{KY}, we find that if $\mu\notin \sigma(\wt{\mathcal Q}_1(X^s))$, then $\mu$ is an eigenvalue of $\wt {\mathcal Q}_1( X^s + \gamma E)$ if and only if
\begin{equation}
\det\left(O^* G^s (\mu)O + (\gamma P_D)^{-1}\right)=0, \quad 0<\gamma<1,
\end{equation}
where
$$G^s(\mu):=\left(H^s-\left( {\begin{array}{*{20}c}
   {  I_{n\times n}} & 0  \\
   {0} & { \mu I_{N\times N}}  \\
\end{array}} \right)\right)^{-1}, \quad O:=\left( {\begin{array}{*{20}c}
   {\Sig^{1/2} U^{*}} & 0  \\
   { 0} & {\wt \Sig^{1/2}V^*}  \\
   \end{array}} \right)W.$$ 
Define $R^\gamma:=O^* G^s O + (\gamma P_D)^{-1}$ for $0<\gamma < 1$, and let $\mu: = \lambda_1^s \pm N^{-3/4}.$ We claim that 
\begin{equation}\label{suff_claim}
{\mathbb P\left(\det R^\gamma (\mu) \ne 0 \text{ for all }0< \gamma \le 1\right) = 1-\oo(1).}
\end{equation}
If (\ref{suff_claim}) holds, then $\mu$ is not an eigenvalue of $\wt{\mathcal Q}_1(X+\gamma E)$ with probability $1-\oo(1)$. Denote the largest eigenvalue of $\wt{\mathcal Q}_1(X+\gamma E)$ by $\lambda^\gamma_1$, $0<\gamma \leq 1$, and define $\lambda^{0}_1:=\lim_{\gamma\downarrow 0}\lambda_1^\gamma$. Then we have $\lambda^0_1= \lambda_1^s $ and $\lambda^1_1= \lambda_1^E$. With the continuity of $\lambda_1^\gamma$ with respect to $\gamma$ and the fact that $\lambda^0_1\in (\lambda_1^s -N^{-3/4},\lambda_1^s +N^{-3/4})$, we find that
$$ \lambda_1^E = \lambda^1_1 \in (\lambda_1^s-N^{-3/4},\lambda_1^s+N^{-3/4}),$$
with probability $1-\oo(1)$, which proves (\ref{eq_claim}).


Finally, we prove (\ref{suff_claim}). Note that $\eta_0=\OO(b_N/N)=\OO(N^{-1+\e})$, hence $z=\lambda_++\ii N^{-{2}/{3}}$ is in $\wt S(c_0, C_0,\delta, \delta)$ for a small constant $\delta>0$. Now we write
\be\label{Rgamma}
R^\gamma(\mu)=O^* \left( G^s(\mu) - G^s(z)\right) O+ O^* \left( G^s(z) - \Pi(z)\right) O + O^*  \Pi(z) O + (\gamma P_D)^{-1} . 
\ee
With \eqref{psi12}, 
we have 
\be\label{Pi part}
\|O^*  \Pi(z) O\| =\OO(1)
\ee
By Lemma \ref{thm_largebound}, we have
$$\mathbb E \left| \left[O^* \left( G^s(z) - \Pi(z)\right) O\right]_{ab} \right|^2 \prec \Psi^2(z) = \OO(N^{-2/3}), \quad 1\le a,b \le 2m,$$
where we used (\ref{Immc}) and \eqref{eq_defpsi} in the second step. Then with Markov's inequality and a union bound, we can get that
\begin{equation}\label{BOUND2}
\max_{1\le a , b \le 2n_0} \left| \left[O^* \left( G^s(z) - \Pi(z)\right) O\right]_{ab} \right| \le N^{-1/6}
\end{equation}
holds with probability $1- \OO(n_0N^{-1/3})$. Thus we have
\be\label{Gz part}
\left\|O^* \left( G^s(z) - \Pi(z)\right) O\right\| = \OO(n_0N^{-1/6})= \OO(1) \quad \text{with probability $1- \OO(n_0N^{-1/3})$}.
\ee
It remains to bound the first term in \eqref{Rgamma}. As pointed out in Remark \ref{rigid_multi}, we can extend (\ref{univ_small}) to the finite correlation functions of the largest eigenvalues. Since the largest eigenvalues in the Gaussian case are separated in the scale $N^{-2/3}$, we conclude that
\begin{equation}\label{repulsion_estimate}
\mathbb P\left( \min_{i}|\lambda_i(\wt Q_1(X^s) ) - \mu| \ge N^{-3/4} \right) = 1-\oo(1).
\end{equation}
On the other hand, the rigidity result (\ref{rigidity}) gives that 
\begin{equation}\label{rigid_estimate}
|\mu - \lambda_+ | \prec N^{-2/3} .
\end{equation}
Using (\ref{delocal}), (\ref{repulsion_estimate}), (\ref{rigid_estimate}) and the rigidity estimate (\ref{rigidity}), we can get that for any set $\Omega$ of deterministic unit vectors of cardinality $N^{\OO(1)}$, 
\begin{equation}
\sup_{\mathbf u,\mathbf v\in \Omega}\left\vert \left\langle \mathbf u, \left(G^{s}(z)-G^{s}(\mu)\right) \mathbf v\right\rangle \right\vert \le N^{-{1}/{4}+3\e} \label{REALCOMPLEX}
\end{equation}
with probability $1-\oo(1)$. For instance, for deterministic unit vectors $\mathbf u,\mathbf v \in \mathbb C^{\mathcal I_2}$ and any constant $0<c\le \e$, we have with probability $1-\oo(1)$ that
\begin{align*}
& \left| \left\langle \mathbf u, \left(G^{s}(z)-G^{s}(\mu)\right) \mathbf v\right\rangle\right|  \le \sum_{k} \left|\langle \mathbf u, \zeta_k \rangle \langle \mathbf v, \zeta_k\rangle \right|\left| \frac{1}{\lambda_k - z} - \frac{1}{\lambda_k - \mu}\right| \\
& \prec \frac{1}{N^{2/3}}\sum_{\gamma_k \le \lambda_+- c_1} \left|\langle \mathbf u, \zeta_k \rangle \langle \mathbf v, \zeta_k\rangle \right| +  \frac{N^\e}{N^{5/3}}\sum_{ \gamma_k > \lambda_+ - c_1} \frac{1}{|\lambda_k - z||\lambda_k - \mu|}  \\
& \le \frac{1}{N^{2/3}} + \frac{N^\e}{N^{5/3}}\sum_{1\le k \le N^c} \frac{1}{|\lambda_k - z||\lambda_k - \mu|} + \frac{N^\e}{N^{5/3}}\sum_{k> N^c, \gamma_k>\lambda_+ - c_1} \frac{1}{|\lambda_k - z||\lambda_k - \mu|} \\
& \prec \frac{1}{N^{2/3}} + \frac{N^{c+\e}}{N^{1/4}} + \frac{N^\e}{N^{2/3}}\left(\frac{1}{N}\sum_{k> N^c, \gamma_k>\lambda_+ - c_1} \frac{1}{|\lambda_k - z||\lambda_k - \mu|}\right) \prec N^{-1/4+c+\e},
\end{align*}
where in the first step we used (\ref{spectral1}), in the second step (\ref{delocal}) (\nc with $\eta_0= \OO(N^{-1+\e})$\nc) and $|\lambda_k - z||\lambda_k - \mu| \gtrsim 1$ for $\gamma_k \le \lambda_+-c_1$ due to (\ref{rigidity}), in the third step the Cauchy-Schwarz inequality, in the fourth step (\ref{repulsion_estimate}), and in the last step $|\lambda_k - z||\lambda_k - \mu| \sim (k/N)^{-4/3}$ for $k>N^c$ by the rigidity estimate (\ref{rigidity}). 
For the other choices of deterministic unit vectors $\mathbf u ,\mathbf v \in \mathbb C^{\mathcal I_{1,2}}$, we can prove \eqref{REALCOMPLEX} in a similar way. Now with (\ref{REALCOMPLEX}), we can get that 
\be\label{Gmu part}
\left\|O^* \left( G^s(\mu) - G^s(z)\right) O\right\| = \OO(n_0N^{-1/4+3\e}) \quad \text{with probability $1- \oo(1)$}.
\ee
With \eqref{Pi part}, \eqref{Gz part} and \eqref{Gmu part}, we see that as long as $\omega$ is chosen to be sufficiently small, we have
$$\left\| O^* \left( G^s(\mu) - G^s(z)\right) O+ O^* \left( G^s(z) - \Pi(z)\right) O + O^*  \Pi(z) O\right\| < (\gamma \omega)^{-1}$$
for all $0<\gamma\le 1$ with probability $1-\oo(1)$. This proves the claim (\ref{suff_claim}), which further gives (\ref{eq_claim}) and completes the proof.

\section{Proof of Theorem \ref{LEM_SMALL}: Gaussian $X$}\label{sec_Gauss}


As discussed below Theorem \ref{LEM_SMALL}, in this section we prove Theorem \ref{LEM_SMALL} for separable covariance matrices of the form $\Sig^{1/2} X \wt \Sig X^* \Sig^{1/2}$, which will imply the local laws in the Gaussian $X$ case. Thus in this section, we use the following resolvent:
\begin{equation}\label{eqn_comparison1}
G(X,z) {=}  \left[\left( {\begin{array}{*{20}c}
   { 0 } & \Sig^{1/2} X \wt \Sig^{1/2}   \\
   {\wt\Sig^{1/2}X^*\Sig^{1/2} } & {0}  \\
   \end{array}} \right)-\left( {\begin{array}{*{20}c}
   { I_{n\times n}} & 0  \\
   0 & { zI_{N\times N}}  \\
\end{array}} \right)\right]^{-1},
\end{equation}
with $X$ satisfying \eqref{eq_support} with $q=N^{-1/2}$. More precisely, we will prove the following result. 

\begin{proposition}\label{prop_diagonal}
Suppose Assumption \ref{assm_big1} and \eqref{assm_gap} hold. Suppose $X$ satisfies the bounded support condition (\ref{eq_support}) with $q= N^{-1/2}$. Suppose $A$ and $B$ are diagonal, i.e. $U=I_{n\times n}$ and $V=I_{N\times N}$. Fix $C_0>1$ and let $c_0>0$ be a sufficiently small constant. Then for any fixed $\epsilon>0$, the following estimates hold. 
\begin{itemize}
\item[(1)] {\bf Anisotropic local law}:  For any $z\in S(c_0,C_0,\epsilon)$ and deterministic unit vectors $\mathbf u, \mathbf v \in \mathbb C^{\mathcal I}$,
\begin{equation}\label{aniso_diagonal}
\left| \langle \mathbf u, G(X,z) \mathbf v\rangle - \langle \mathbf u, \Pi (z)\mathbf v\rangle \right| \prec \Psi(z).
\end{equation}

\item[(2)] {\bf Averaged local law}: We have 
\begin{equation}\label{aver_diagonal}
 | m(z)-m_{c}(z)|\prec ({N\eta})^{-1}
\end{equation}
for any $z\in S(c_0,C_0,\epsilon)$, and 
\begin{equation}\label{aver_out}
 | m(z)-m_{c}(z)|\prec \frac{1}{N(\kappa +\eta)} + \frac{1}{(N\eta)^2\sqrt{\kappa +\eta}},
\end{equation}
for any $z\in S(c_0,C_0,\epsilon)\cap \{z=E+\ii\eta: E\ge \lambda_+, N\eta\sqrt{\kappa + \eta} \ge N^\epsilon\}$. 
\end{itemize}
Both of the above estimates are uniform in the spectral parameter $z$ and the deterministic vectors $\mathbf u, \mathbf v$.

\end{proposition}

Under a different set of assumptions, the local law as in Proposition \ref{prop_diagonal} has been proved in \cite{Alt_Gram}. \nc However, in order to satisfy their assumptions in our setting, we need to assume that the eigenvalues of $A$ and $B$ are both upper and lower bounded by some constants $\tau \le \sigma_i,  \wt\sigma_i \le \tau^{-1}$, which rules out the possibility of zero or very small (that is, ${\rm o}(1)$) eigenvalues of $A$ and $B$. On the other hand, our assumptions in \eqref{assm3} and \eqref{assm_gap} are slightly more general, and allow for a large portion of small or zero eigenvalues of $A$ and $B$. \nc For reader's convenience, we shall give the proof of Proposition \ref{prop_diagonal} in our setting. This proof is similar to the previous proof of the local laws, such as \cite{isotropic, DY, Anisotropic, XYY_circular}. Thus instead of giving all the details, we only describe briefly the proof. In particular, we shall focus on the key self-consistent equation argument, which is (almost) the only part that departs significantly from the previous proof in e.g. \cite{isotropic}. In the proof, we always denote the spectral parameter by $z=E+\ii\eta$. 

\subsection{Basic tools}
In this subsection, we collect some basic tools that will be used. For simplicity, we denote $Y:=\Sig^{1/2} X \wt \Sig^{1/2}$. 

\begin{definition}[Minors]
For any $ (n+N)\times (n+N)$ matrix $\cal A$ and $\mathbb T \subseteq \mathcal I$, we define the minor $\cal A^{(\mathbb T)}:=(\cal A_{ab}:a,b \in \mathcal I\setminus \mathbb T)$ as the $ (n+N-|\mathbb T|)\times (n+N-|\mathbb T|)$ matrix obtained by removing all rows and columns indexed by $\mathbb T$. Note that we keep the names of indices when defining $\cal A^{(\mathbb T)}$, i.e. $(\cal A^{(\mathbb{T})})_{ab}= \cal A_{ab}$ for $a,b \notin \mathbb{{T}}$. Correspondingly, we define the resolvent minor as
\begin{align*}
G^{(\mathbb T)}:&=\left[\left(H - \left( {\begin{array}{*{20}c}
   { I_{n\times n}} & 0  \\
   0 & { zI_{N\times N}}  \\
\end{array}} \right)\right)^{(\mathbb T)}\right]^{-1} = \left( {\begin{array}{*{20}c}
   { z\mathcal G_1^{(\mathbb T)}} & \mathcal G_1^{(\mathbb T)} Y^{(\mathbb T)}  \\
   {\left(Y^{(\mathbb T)}\right)^*\mathcal G_1^{(\mathbb T)}} & { \mathcal G_2^{(\mathbb T)} }  \\
\end{array}} \right)  = \left( {\begin{array}{*{20}c}
   { z\mathcal G_1^{(\mathbb T)}} & Y^{(\mathbb T)}\mathcal G_2^{(\mathbb T)}   \\
   {\mathcal G_2^{(\mathbb T)}}\left(Y^{(\mathbb T)}\right)^* & { \mathcal G_2^{(\mathbb T)} }  \\
\end{array}} \right),
\end{align*}
and the partial traces
$$m_1^{(\mathbb T)}:=\frac{1}{Nz}\sum_{i\notin \mathbb T}\sigma_i G_{ii}^{(\mathbb T)},\ \ m_2^{(\mathbb T)}:= \frac{1}{N}\sum_{\mu \notin \mathbb T}\wt \sigma_\mu G_{\mu\mu}^{(\mathbb T)}.$$ 
For convenience, we will adopt the convention that for any minor $\cal A^{(T)}$ defined as above, $\cal A^{(T)}_{ab} = 0$ if $a \in \mathbb T$ or $b \in \mathbb T$. We will abbreviate $(\{a\})\equiv (a)$, $(\{a, b\})\equiv (ab)$, and $\sum_{a}^{(\mathbb T)} := \sum_{a\notin \mathbb T} .$
\end{definition}


\begin{lemma}{(Resolvent identities).}

\begin{itemize}
\item[(i)]
For $i\in \mathcal I_1$ and $\mu\in \mathcal I_2$, we have
\begin{equation}
\frac{1}{{G_{ii} }} =  - 1 - \left( {YG^{\left( i \right)} Y^*} \right)_{ii} ,\ \ \frac{1}{{G_{\mu \mu } }} =  - z  - \left( {Y^*  G^{\left( \mu  \right)} Y} \right)_{\mu \mu }.\label{resolvent2}
\end{equation}

 \item[(ii)]
 For $i\ne j \in \mathcal I_1$ and $\mu \ne \nu \in \mathcal I_2$, we have
\begin{equation}
G_{ij}   = G_{ii} G_{jj}^{\left( i \right)} \left( {YG^{\left( {ij} \right)} Y^* } \right)_{ij},\ \ G_{\mu \nu }  = G_{\mu \mu } G_{\nu \nu }^{\left( \mu  \right)} \left( {Y^*  G^{\left( {\mu \nu } \right)} Y} \right)_{\mu \nu }. \label{resolvent3}
\end{equation}
For $i\in \mathcal I_1$ and $\mu\in \mathcal I_2$, we have
\begin{equation}\label{resolvent6}
\begin{split}
& G_{i\mu } = G_{ii} G_{\mu \mu }^{\left( i \right)} \left( { - Y_{i\mu }  +  {\left( {YG^{\left( {i\mu } \right)} Y} \right)_{i\mu } } } \right), \ \  G_{\mu i}  = G_{\mu \mu } G_{ii}^{\left( \mu  \right)} \left( { - Y_{\mu i}^*  + \left( {Y^*  G^{\left( {\mu i} \right)} Y^*  } \right)_{\mu i} } \right).
\end{split}
\end{equation}

 \item[(iii)]
 For $a \in \mathcal I$ and $b, c \in \mathcal I \setminus \{a\}$,
\begin{equation}
G_{bc}  = G_{bc}^{\left( a \right)} + \frac{G_{ba} G_{ac}}{G_{aa}}, \ \ \frac{1}{{G_{bb} }} = \frac{1}{{G_{bb}^{(a)} }} - \frac{{G_{ba} G_{ab} }}{{G_{bb} G_{bb}^{(a)} G_{aa} }}. \label{resolvent8}
\end{equation}

 \item[(iv)]
All of the above identities hold for $G^{(\mathbb T)}$ instead of $G$ for $\mathbb T \subset \mathcal I$, and in the case where $A$ and $B$ are not diagonal.
\end{itemize}
\label{lemm_resolvent}
\end{lemma}
\begin{proof}
All these identities can be proved using Schur's complement formula. The reader can refer to, for example, \cite[Lemma 4.4]{Anisotropic}.
\end{proof}

\begin{lemma}\label{Ward_id}
Fix constants $c_0,C_0>0$. The following estimates hold uniformly for all $z\in S(c_0,C_0,a)$ for any $a\in \mathbb R$:
\begin{equation}
\left\| G \right\| \le C\eta ^{ - 1} ,\ \ \left\| {\partial _z G} \right\| \le C\eta ^{ - 2}. \label{eq_gbound}
\end{equation}
Furthermore, we have the following identities:
\begin{align}
& \sum_{i \in \mathcal I_1 }  \left| {G_{j i} } \right|^2 = \sum_{i \in \mathcal I_1 }  \left| {G_{ij} } \right|^2  = \frac{|z|^2}{\eta}\Im\left(\frac{G_{jj}}{z}\right) ,  \label{eq_gsq2} \\
& \sum_{\mu  \in \mathcal I_2 } {\left| {G_{\nu \mu } } \right|^2 } = \sum_{\mu  \in \mathcal I_2 } {\left| {G_{\mu \nu} } \right|^2 }  = \frac{{\Im \, G_{\nu\nu} }}{\eta}, \label{eq_gsq1}\\ 
& \sum_{i \in \mathcal I_1 } {\left| {G_{\mu i} } \right|^2 } = \sum_{i \in \mathcal I_1 } {\left| {G_{i\mu} } \right|^2 } = {G}_{\mu \mu}  + \frac{\bar z}{\eta} \Im \, G_{\mu\mu} , \label{eq_gsq3} \\ 
&\sum_{\mu \in \mathcal I_2 } {\left| {G_{i \mu} } \right|^2 } = \sum_{\mu \in \mathcal I_2 } {\left| {G_{\mu i} } \right|^2 } =  \frac{{G}_{ii}}{z}  + \frac{\bar z}{\eta} \Im\left(\frac{{G_{ii} }}{z}\right) . \label{eq_gsq4} 
 \end{align}
All of the above estimates remain true for $G^{(\mathbb T)}$ instead of $G$ for any $\mathbb T \subseteq \mathcal I$, and in the case where $A$ and $B$ are not diagonal.
\label{lemma_Im}
\end{lemma}
\begin{proof}
These estimates and identities can be proved through simple calculations with (\ref{green2}), (\ref{spectral1}) and (\ref{spectral2}). We refer the reader to \cite[Lemma 4.6]{Anisotropic} and \cite[Lemma 3.5]{XYY_circular}.
\end{proof}

\begin{lemma}
Fix constants $c_0,C_0>0$. For any $\mathbb T \subseteq \mathcal I$ and $a\in \mathbb R$, the following bounds hold uniformly in $z\in S(c_0,C_0,a)$:
\begin{equation}\label{m_T}
\big| {m_1  - m_1^{\left( \mathbb T \right)} } \big| + \big| {m_2  - m_2^{\left( \mathbb T \right)} } \big| \le \frac{{C\left| \mathbb T \right|}}{{N\eta }}, 
\end{equation}
where $C>0$ is a constant depending only on $\tau$.
\end{lemma}
\begin{proof}
For $\mu\in\mathcal I_2$, we have
\begin{align*}
\left|m_2-m_2^{(\mu)}\right|& =\frac{1}{N}\left|\sum_{\nu\in\mathcal I_2}  \wt \sigma_\nu\frac{G_{\nu\mu}G_{\mu\nu}}{G_{\mu\mu}}\right| \le \frac{C}{N|G_{\mu\mu}|} \sum_{\nu\in\mathcal I_2} |G_{\nu\mu}|^2 = \frac{C\Im\, G_{\mu\mu}}{N\eta |G_{\mu\mu}|} \le \frac{C}{N\eta}, 
\end{align*}
where in the first step we used (\ref{resolvent8}), and in the second and third steps we used (\ref{eq_gsq1}). Similarly, using (\ref{resolvent8}) and (\ref{eq_gsq3}) we get
\begin{align*}
\left|m_2 -m_2^{(i)}\right| & = \frac{1}{N}\left|\sum_{\nu \in\mathcal I_2}\wt \sigma_\nu\frac{G_{\nu i}G_{i\nu}}{G_{ii}}\right| \le \frac{C}{N|G_{ii}|} \left( \frac{{G}_{ii}}{z}  + \frac{\bar z}{\eta} \Im\left(\frac{{G_{ii} }}{z}\right)\right)   \le \frac{C}{N\eta}.
\end{align*}
Similarly, we can prove the same bounds for $m_1$. Then (\ref{m_T}) can be proved by induction on the indices in $\mathbb T$. 
\end{proof}

The following lemma gives large deviation bounds for bounded supported random variables. 

\begin{lemma}[Lemma 3.8 of \cite{EKYY1}]\label{largederivation}
Let $(x_i)$, $(y_j)$ be independent families of centered and independent random variables, and $(A_i)$, $(B_{ij})$ be families of deterministic complex numbers. Suppose the entries $x_i$, $y_j$ have variance at most $N^{-1}$ and satisfy the bounded support condition (\ref{eq_support}) with $q\le N^{-\epsilon}$ for some constant $\epsilon>0$. Then we have the following bound:
\begin{align}
\Big\vert \sum_i A_i x_i \Big\vert \prec  q \max_{i} \vert A_i \vert+ \frac{1}{\sqrt{N}}\Big(\sum_i |A_i|^2 \Big)^{1/2} , \quad  & \Big\vert \sum_{i,j} x_i B_{ij} y_j \Big\vert \prec q^2 B_d  + qB_o + \frac{1}{N}\Big(\sum_{i\ne j} |B_{ij}|^2\Big)^{{1}/{2}} , \\
 \Big\vert \sum_{i} \bar x_i B_{ii} x_i - \sum_{i} (\mathbb E|x_i|^2) B_{ii}  \Big\vert  \prec q B_d  ,\quad & \Big\vert \sum_{i\ne j} \bar x_i B_{ij} x_j \Big\vert  \prec qB_o + \frac{1}{N}\left(\sum_{i\ne j} |B_{ij}|^2\right)^{{1}/{2}} ,
\end{align}
where $B_d:=\max_{i} |B_{ii} |$ and $B_o:= \max_{i\ne j} |B_{ij}|.$
\end{lemma}  

For the proof of Proposition \ref{prop_diagonal}, it is convenient to introduce the following random control parameters.

\begin{definition}[Control parameters]
We define the random errors
\begin{equation}\label{eqn_randomerror}
\Lambda : = \mathop {\max }\limits_{a,b \in \mathcal I} \left| {\left( {G - \Pi } \right)_{ab} } \right|,\ \ \Lambda _o : = \mathop {\max }\limits_{a \ne b \in \mathcal I} \left| {G_{ab} } \right|, \ \ \theta:= |m_1-m_{1c}| +  |m_2-m_{2c}| ,
\end{equation}
and the random control parameter (recall $\Psi$ defined in \eqref{eq_defpsi})
\begin{equation}\label{eq_defpsitheta}
\Psi _\theta  : = \sqrt {\frac{{\Im \, m_{2c}  + \theta }}{{N\eta }}} + \frac{1}{N\eta}.
\end{equation}
\end{definition}

\subsection{Entrywise local law}

The main goal of this subsection is to prove the following entrywise local law. The anisotropic local law \eqref{aniso_diagonal} then follows from the entrywise local law combined with a polynomialization method as we will explain in next subsection.

\begin{proposition}\label{prop_entry}
Suppose the assumptions in Proposition \ref{prop_diagonal} hold. Fix $C_0>0$ and let $c_0>0$ be a sufficiently small constant. Then for any fixed $\epsilon>0$, the following estimate holds uniformly for $z\in S(c_0,C_0,\epsilon)$: 
\begin{equation}\label{entry_diagonal}
\max_{a,b}\left| G_{ab}(X,z)  - \Pi_{ab} (z) \right| \prec \Psi(z).
\end{equation} 
\end{proposition}

In analogy to \cite[Section 3]{EKYY1} and \cite[Section 5]{Anisotropic}, we introduce the $Z$ variables
\begin{equation*}
  Z_{a}^{(\mathbb T)}:=(1-\mathbb E_{a})\big(G_{aa}^{(\mathbb T)}\big)^{-1}, \ \ a\notin \mathbb T,
\end{equation*}
where $\mathbb E_{a}[\cdot]:=\mathbb E[\cdot\mid H^{(a)}],$ i.e. it is the partial expectation over the randomness of the $a$-th row and column of $H$. By (\ref{resolvent2}), we have
\begin{equation}
Z_i = (\mathbb E_{i} - 1) \left( {YG^{\left( i \right)} Y^*} \right)_{ii} = \sigma_i \sum_{\mu ,\nu\in \mathcal I_2} \sqrt{\wt \sigma_\mu \wt \sigma_\nu}G^{(i)}_{\mu\nu} \left(\frac{1}{N}\delta_{\mu\nu} - X_{i\mu}X_{i\nu}\right),\label{Zi}
\end{equation}
and
\begin{equation}
\begin{split}
Z_\mu &= (\mathbb E_{\mu} - 1) \left( {Y^*  G^{\left( \mu  \right)} Y} \right)_{\mu \mu } = \wt \sigma_\mu \sum_{i,j \in \mathcal I_1} \sqrt{\sigma_i \sigma_j}G^{(\mu)}_{ij} \left(\frac{1}{N} \delta_{ij} - X_{i\mu}X_{j\mu}\right).\label{Zmu}
\end{split}
\ee
The following lemma plays a key role in the proof of local laws.

\begin{lemma}\label{Z_lemma}
Suppose the assumptions in Proposition \ref{prop_diagonal} hold. Let $c_0>0$ be a sufficiently small constant and fix $C_0, \epsilon >0$. Define the $z$-dependent event $\Xi(z):=\{\Lambda(z) \le (\log N)^{-1}\}$. Then there exists constant $C>0$ such that the following estimates hold uniformly for all $a\in \mathcal I$ and $z\in S(c_0,C_0,\epsilon)$:
\begin{align}
{\mathbf 1}(\Xi)\left(\Lambda_o + |Z_{a}|\right) \prec \Psi_\theta, \label{Zestimate1}
\end{align}
and 
\begin{align}
{\mathbf 1}\left(\eta \ge 1 \right)\left(\Lambda_o + |Z_{a}|\right)\prec \Psi_\theta. \label{Zestimate2}
\end{align}
\end{lemma}
\begin{proof}
 Applying Lemma \ref{largederivation} to $Z_{i}$ in (\ref{Zi}), we get that on $\Xi$,
\begin{equation}\label{estimate_Zi}
\begin{split}
\left| Z_{i}\right| & \prec  q+\frac{1}{N} \left( \sum_{\mu, \nu} \wt \sigma_\mu {\left| G_{\mu\nu}^{(i)}  \right|^2 }  \right)^{1/2} = q+ \frac{1}{N}\left( {\sum_\mu \frac{\wt \sigma_\mu  \im G_{\mu\mu}^{(i)} }{\eta} } \right)^{1/2}= q + \sqrt { \frac{ \Im\, m_2^{(i)}  } {N\eta} },
\end{split}
\end{equation}
where we used (\ref{assm3}), (\ref{eq_gsq1}) and the fact that $\max_{a,b}|G_{ab}|= \OO(1)$ on event $\Xi$. Now by (\ref{eqn_randomerror}), (\ref{eq_defpsitheta}) and the bound (\ref{m_T}), we have that
\begin{align}\label{m2psi}
\sqrt{\frac{{\Im\, m_2^{(i)} }}{N\eta} } = \sqrt {\frac{{\Im\,m_{2c}  + \Im ( {m_2^{(i)}  - m_2 }) + \Im ( {m_2  - m_{2c} } )}}{{N\eta }}}  \le C \Psi _\theta .
\end{align}
Together with the fact that $q= N^{-1/2} \lesssim \Psi_\theta$ by \eqref{psi12}, we get \eqref{Zestimate1} for ${\mathbf 1}(\Xi) |Z_{i}|$. Similarly, we can prove the same estimate for ${\mathbf 1}(\Xi) |Z_{\mu}| $, where in the proof we need to use  (\ref{eq_gsq2}) and \eqref{psi12}. If $\eta\ge 1$, we also have $\max_{a,b}|G_{ab}|= \OO(1)$ by (\ref{eq_gbound}). Then repeating the above proof, we obtain \eqref{Zestimate2} for ${\mathbf 1}(\eta\ge 1) |Z_{a}|$.
Similarly, using (\ref{resolvent3}) and Lemmas \ref{Ward_id}-\ref{largederivation}, we can prove that 
\begin{equation}\label{2blocks}
{\mathbf 1}(\Xi) \left( |G_{ij}| + |G_{\mu\nu}|\right) + {\mathbf 1}(\eta\ge 1) \left( |G_{ij}| + |G_{\mu\nu}|\right) \prec \Psi_\theta.
\end{equation}

It remains to prove the bounds for $G_{i\mu}$ and $G_{\mu i}$ entries. Using (\ref{resolvent6}), \eqref{eq_support}, 
the bound $\max_{a,b}|G_{ab}|= \OO(1)$ on $\Xi$, Lemma \ref{Ward_id} and Lemma \ref{largederivation}, we get that 
\begin{equation*}
\begin{split}
\left|G_{i\mu }\right| & \prec q+  \frac{1}{N}\left( {\sum^{(i\mu)}_{j,\nu } \wt \sigma_\nu {\left| {G_{\nu j}^{(i\mu)} } \right|^2 } } \right)^{1/2}= q+  \frac{1}{N}\left( \sum_{\nu}^{(\mu)}\wt \sigma_\nu \left({G}^{(i\mu)}_{\nu \nu}  + \frac{\bar z}{\eta} \Im \, G_{\nu\nu}^{(i\mu)}\right)\right)^{1/2} \lesssim q+\sqrt{ \frac{|m_2^{(i\mu)}|}{N} } + \sqrt{\frac{ \Im \, m_2^{(i\mu)}}{N\eta}} .
\end{split}
\end{equation*}
As in (\ref{m2psi}), we can show that
\begin{equation}\label{estimatel1} \sqrt{\frac{ \Im \, m_2^{(i\mu)}}{N\eta}}= \OO(\Psi_\theta).\end{equation}
For the other term, we have
\begin{equation}\label{estimatel2}
\begin{split}
  \sqrt{ \frac{|m_2^{(i\mu)}|}{N} } & \le  \sqrt {\frac{|m_{2c}|  + |m_2^{(i\mu)}  - m_2| + |m_2  - m_{2c} |}{{N}}} \lesssim \frac{1}{N\sqrt{\eta}}+\sqrt { \frac{\theta }{{N}}}  + \sqrt {\frac{{\left| {m_{2c} } \right|}}{N}}   \lesssim  \Psi _\theta , 
 \end{split}
\end{equation}
where we used (\ref{m_T}) and ${{\left| {m_{2c} } \right|}}{{N}^{-1}} = O(\Psi^2) $ by \eqref{psi12}. With (\ref{estimatel1}) and (\ref{estimatel2}), we obtain that ${\mathbf 1}(\Xi) |G_{i\mu}| \prec \Psi_\theta $. Together with (\ref{2blocks}), we get the estimate (\ref{Zestimate1}) for $\mathbf 1(\Xi)\Lambda_o$. Finally, the estimate (\ref{Zestimate2}) for ${\mathbf 1}\left(\eta \ge 1 \right)\Lambda_o$ can be proved in a similar way with the bound $\mathbf{1}(\eta\ge 1)\max_{a,b}|G_{ab}|= \OO(1)$.
\end{proof}

A key component of the proof for Proposition \ref{prop_entry} is an analysis of the self-consistent equation. Recall the equations in \eqref{separa_m12} and the function $f(z,\al)$ in \eqref{separable_MP}.


\begin{lemma}\label{lemm_selfcons_weak}
Let $c_0>0$ be a sufficiently small constant and fix $C_0,\epsilon>0$. Then the following estimates hold uniformly in $z \in S(c_0, C_0,\epsilon)$: 
\begin{equation}
{\mathbf 1}(\eta \ge 1)\left| f(z, m_2) \right|\prec N^{-1/2}, \quad {\mathbf 1}(\eta\ge 1)\left|  m_1(z) - d_N \int\frac{x}{-z\left[1+xm_{2}(z) \right]} \pi_A^{(n)}(\dd x)\right| \prec N^{-1/2}, \label{selfcons_lemm2}
\end{equation}
and
\begin{equation}
{\mathbf 1}(\Xi)\left| f(z, m_2) \right| \prec \Psi_\theta, \quad {\mathbf 1}(\Xi)\left|  m_1(z) - d_N \int\frac{x}{-z\left[1+xm_{2}(z) \right]} \pi_A^{(n)}(\dd x)\right| \prec \Psi_\theta, \label{selfcons_lemm}
\end{equation}
where $\Xi$ is as given in Lemma \ref{Z_lemma}. Moreover, we have the finer estimates
\begin{equation}
{\mathbf 1}(\Xi)\left|f(z, m_2)\right| \prec  {\mathbf 1}(\Xi)\left(\left|[Z]_1\right| + \left|[Z]_2\right|\right) + \Psi^2_\theta, \label{selfcons_improved}
\end{equation}
and
\be\label{selfcons_improved2}
{\mathbf 1}(\Xi)\left|  m_1(z) - d_N \int\frac{x}{-z\left[1+xm_{2}(z) \right]} \pi_A^{(n)}(\dd x)\right| \prec {\mathbf 1}(\Xi)\left|[Z]_1\right| + \Psi^2_\theta,
\ee
where
\begin{equation}\label{def_Zaver}
[Z]_1:=\frac{1}{N}\sum_{i\in \mathcal I_1} \frac{\sigma_i}{(1+\sigma_i m_2)^2} Z_i, \ \ [Z]_2:=\frac{1}{N}\sum_{\mu \in \mathcal I_2} \frac{\wt \sigma_\mu }{\left(1 + \wt \sigma_\mu m_1\right)^2}Z_\mu.
\end{equation}
\end{lemma}

\begin{proof}
We first prove (\ref{selfcons_improved}) and \eqref{selfcons_improved2}, from which (\ref{selfcons_lemm}) follows due to (\ref{Zestimate1}) and (\ref{Piii}). By (\ref{resolvent2}), (\ref{Zi}) and (\ref{Zmu}), we have
\begin{equation}\label{self_Gii}
\frac{1}{{G_{ii} }}=  - 1 - \frac{\sigma_i}{N} \sum_{\mu\in \mathcal I_2}\wt \sigma_\mu G^{\left( i \right)}_{\mu\mu} + Z_i =  - 1 - \sigma_i m_2 + \epsilon_i,
\end{equation}
and
\begin{equation}\label{self_Gmu}
\frac{1}{{G_{\mu\mu} }}=  - z - \frac{\wt \sigma_\mu}{N} \sum_{i\in \mathcal I_1}\sigma_i G^{\left( \mu\right)}_{ii}+ Z_{\mu} =  - z - z \wt \sigma_\mu m_1 + \epsilon_\mu,
\end{equation}
where 
$$\epsilon_i := Z_i + \sigma_i\left(m_2 - m_2^{(i)}\right) \ \ \text{and} \ \ \epsilon_\mu := Z_{\mu} + z \wt \sigma_\mu\left(m_1 -m_1^{(\mu)}\right).$$
By (\ref{m_T}) and (\ref{Zestimate1}), we have for all $i$ and $\mu$,
\begin{equation}\label{erri}
\mathbf 1(\Xi)\left(|\epsilon_i | + |\epsilon_\mu| \right)\prec \Psi_\theta. 
\end{equation}
Moreover, by (\ref{resolvent8}) we have
\begin{equation}\label{high_err}
\mathbf 1(\Xi)\left( |m_2 - m_2^{(i)}| + |m_1 -m_1^{(\mu)}| \right)\le \mathbf 1(\Xi)\frac{1}{N}\left(\sum_{\nu\in \mathcal I_2} \wt \sigma_\nu \left|\frac{G_{\nu i} G_{i\nu}}{G_{ii}}\right| + \sum_{j\in \mathcal I_1}\sigma_j \left|\frac{G_{j\mu} G_{\mu j}}{G_{\mu\mu}}\right| \right)\prec \Psi_\theta^2,
\end{equation}
where we used (\ref{Zestimate1}) and $|G_{ii}| \sim |G_{\mu\mu}| \sim 1$ on $\Xi$ in the second step.
Now using \eqref{self_Gii}, \eqref{erri}, \eqref{high_err}, \eqref{Zestimate1}, \eqref{Piii} and the definition of $\Xi$, we can obtain that
\be\label{Gii0}
\mathbf 1(\Xi)G_{ii}=\mathbf 1(\Xi)\left[\frac{1}{-(1 + \sigma_i m_2)} - \frac{Z_i}{\left(1 +  \sigma_i m_2\right)^2}  +O_\prec\left(\Psi_\theta^2\right)\right].
\ee
Taking average $\frac{1}{Nz}\sum_i \sigma_i$, we get
\be\label{Gii}
\mathbf 1(\Xi)m_1 =\mathbf 1(\Xi)\left[\frac1N\sum_i\frac{\sigma_i}{-z(1 + \sigma_i m_2)} - z^{-1}[Z]_1  +O_\prec\left(\Psi_\theta^2\right)\right],
\ee
which proves \eqref{selfcons_improved2}. On the other hand, using (\ref{self_Gmu}), \eqref{erri}, \eqref{high_err}, \eqref{Zestimate1}, \eqref{Piii} and the definition of $\Xi$, we obtain that
\be\label{Gmumu0}
\mathbf 1(\Xi)G_{\mu\mu}=\mathbf 1(\Xi)\left[\frac{1}{- z(1 + \wt \sigma_\mu m_1)} - \frac{Z_\mu}{z^2\left(1 + \wt \sigma_\mu m_1\right)^2}  +O_\prec\left(\Psi_\theta^2\right)\right].
\ee
Taking average $N^{-1}\sum_\mu \wt \sigma_\mu$, we get
\be\label{Gmumu}
\mathbf 1(\Xi)m_2=\mathbf 1(\Xi)\left[\frac1N\sum_{\mu}\frac{\wt \sigma_\mu}{- z(1 + \wt \sigma_\mu m_1)} -z^{-2}[Z]_2 +O_\prec\left(\Psi_\theta^2\right)\right].
\ee
Plugging \eqref{Gii} into \eqref{Gmumu}, and using \eqref{Piii} and the definition of $\Xi$, we can obtain that
\be\label{end_rep}
\mathbf 1(\Xi)m_2=\mathbf 1(\Xi)\left[\frac1N\sum_{\mu}\frac{\wt \sigma_\mu}{- z + \frac{\wt \sigma_\mu }N\sum_i\frac{\sigma_i}{1 + \sigma_i m_2} } + O_\prec\left(|[Z]_1|+|[Z]_2|+\Psi_\theta^2\right)\right].
\ee
Comparing with \eqref{separable_MP}, we have proved \eqref{selfcons_improved}.

Then we prove (\ref{selfcons_lemm2}). Using the bound $\mathbf{1}(\eta\ge 1)\max_{a,b}|G_{ab}|=\OO(1)$, we trivially have $|m_1|+ |m_2| + \theta=\OO(1)$. Thus we have 
$\mathbf 1(\eta\ge 1) \Psi_\theta =\OO(N^{-1/2})$. Then \eqref{m_T} and (\ref{Zestimate2}) together give that
\begin{equation}\label{epsilonL}
\mathbf 1(\eta\ge 1) (|\epsilon_i|+|\epsilon_\mu|) \prec N^{-1/2}.
\end{equation}
First we claim that in the case $\eta \ge 1$, with high probability,
\begin{equation}\label{estimate_m2L}
|m_1| \ge \Im\, m_1 \ge c , \quad |m_2| \ge \Im\, m_2 \ge c , 
\end{equation} 
for some constant $c>0$. By the spectral decomposition (\ref{spectral1}), we have
$$\im G_{ii} = \Im \sum_{k = 1}^{n} \frac{z|\xi_k(i)|^2}{\lambda_k-z}= \sum_{k = 1}^{n} |\xi_k(i)|^2 \Im \left( -1 + \frac{\lambda_k}{\lambda_k-z}\right)\ge 0.$$
Then applying it to (\ref{self_Gmu}), $G_{\mu\mu}^{-1}$ is of order $\OO(1)$ and has imaginary part $\le - \eta + O_\prec\left( N^{-1/2}\right)$. This implies $ \Im\, G_{\mu\mu} \gtrsim \eta$ with high probability, which gives the second estimate of (\ref{estimate_m2L}) by \eqref{assm3}. Moreover, with \eqref{assm3} we also get that $\im ( 1 + \sigma_i m_2)\gtrsim 1$ for $i \le \tau n.$ 
Then with \eqref{self_Gii} and a similar argument as above, we obtain the first estimate of (\ref{estimate_m2L}). Next, we claim that in the case $\eta\ge 1$, with high probability,
\begin{equation}\label{estimate_m23}
| 1+ \wt \sigma_\mu m_1| \ge c', \quad | 1+ \sigma_i m_2| \ge c' , 
\end{equation} 
for some constant $c'>0$. In fact, if $\sigma_i \le |2m_2|^{-1}$, we trivially have $| 1+ \sigma_i m_2|\ge 1/2$. Otherwise, we have 
$$|1+ \sigma_i m_2| \ge \frac{ \Im\, m_2}{2|m_2|} \ge c'$$
by \eqref{estimate_m2L}. The first estimate in \eqref{estimate_m23} can be proved in the same way. Finally, with (\ref{epsilonL}), (\ref{estimate_m2L}) and (\ref{estimate_m23}), we can repeat the previous arguments between \eqref{self_Gii} and \eqref{end_rep} to get (\ref{selfcons_lemm2}).
\end{proof}

The following lemma gives the stability of the equation $ f(z,\al)=0$. Roughly speaking, it states that if $f(z, m_{2}(z))$ is small and $m_2(\wt z)-m_{2c}(\wt z)$ is small for $\Im\, \wt z \ge \Im\, z$, then $m_{2}(z)-m_{2c}(z)$ is small. For an arbitrary $z\in S(c_0,C_0, \e)$, we define the discrete set
\begin{align*}
L(z):=\{z\}\cup \{z'\in S(c_0,C_0, \e): \text{Re}\, z' = \text{Re}\, z, \text{Im}\, z'\in [\text{Im}\, z, 1]\cap (N^{-10}\mathbb N)\} .
\end{align*}
Thus, if $\text{Im}\, z \ge 1$, then $L(z)=\{z\}$; if $\text{Im}\, z<1$, then $L(z)$ is a 1-dimensional lattice with spacing $N^{-10}$ plus the point $z$. Obviously, we have $|L(z)|\le N^{10}$. 

\begin{lemma}\label{stability}
Let $c_0>0$ be a sufficiently small constant and fix $C_0,\epsilon>0$. The self-consistent equation $f(z,\al)=0$ is stable on $S(c_0,C_0, \epsilon)$ in the following sense. Suppose the $z$-dependent function $\delta$ satisfies $N^{-2} \le \delta(z) \le (\log N)^{-1}$ for $z\in S(c_0,C_0, \epsilon)$ and that $\delta$ is Lipschitz continuous with Lipschitz constant $\le N^2$. Suppose moreover that for each fixed $E$, the function $\eta \mapsto \delta(E+\ii\eta)$ is non-increasing for $\eta>0$. Suppose that $u_2: S(c_0,C_0,\epsilon)\to \mathbb C$ is the Stieltjes transform of a probability measure. Let $z\in S(c_0,C_0,\epsilon)$ and suppose that for all $z'\in L(z)$ we have 
\begin{equation}\label{Stability0}
\left| f(z', u_2)\right| \le \delta(z').
\end{equation}
Then we have
\begin{equation}
\left|u_2(z)-m_{2c}(z)\right|\le \frac{C\delta}{\sqrt{\kappa+\eta+\delta}},\label{Stability1}
\end{equation}
for some constant $C>0$ independent of $z$ and $N$, where $\kappa$ is defined in (\ref{KAPPA}). 
\end{lemma}
\begin{proof}
This lemma can proved with the same method as in e.g. \cite[Lemma 4.5]{isotropic} and \cite[Appendix A.2]{Anisotropic}. The only input is Lemma \ref{lambdar_sqrt}. 
\end{proof}

Note that by Lemma \ref{stability} and (\ref{selfcons_lemm2}), we immediately get that
\begin{equation}\label{average_L}
\mathbf 1(\eta\ge 1)\theta(z) \prec N^{-1/2}.
\end{equation}
Then from (\ref{Zestimate2}), we obtain the off-diagonal estimate
\begin{equation}\label{offD_L}
\mathbf 1(\eta\ge 1)\Lambda_o(z) \prec N^{-1/2}.
\end{equation}
Using (\ref{self_Gii}), \eqref{self_Gmu} and (\ref{average_L}), we get that 
\begin{equation}\label{diag_L}
\mathbf 1(\eta\ge 1)\left(\left|G_{ii} - \Pi_{ii}\right| + |G_{\mu\mu}-\Pi_{\mu\mu}|\right) \prec N^{-1/2},
\end{equation}
which gives the diagonal estimate. These bounds can be easily generalized to the case $\eta \ge c$ for any fixed $c>0$. Compared with (\ref{entry_diagonal}), one can see that the bounds (\ref{offD_L}) and (\ref{diag_L}) are optimal for the $\eta\ge c$ case. Now it remains to deal with the small $\eta$ case (in particular, the local case with $\eta\ll 1$). We first prove the following weak bound.

\begin{lemma}[Weak entrywise local law]\label{alem_weak} 
Let $c_0>0$ be a sufficiently small constant and fix $C_0,\epsilon>0$. Then we have 
\begin{equation} \label{localweakm}
\Lambda(z) \prec (N\eta)^{-1/4},
\end{equation}
uniformly in $z \in S(c_0,C_0,\epsilon)$.
\end{lemma}
\begin{proof}
One can prove this lemma using a continuity argument as in e.g. \cite[Section 4.1]{isotropic}, \cite[Section 5.3]{Semicircle} or \cite[Section 3.6]{EKYY1}. The key inputs are Lemmas \ref{Z_lemma}-\ref{stability}, and the estimates (\ref{average_L})-(\ref{diag_L}) in the $\eta \ge 1$ case. All the other parts of the proof are essentially the same. 
\end{proof}

To get the strong entrywise local law as in \eqref{entry_diagonal}, we need stronger bounds on $[Z]_1$ and $[Z]_2$ in (\ref{selfcons_improved}) and (\ref{selfcons_improved2}). They follow from the following {\it{fluctuation averaging lemma}}. 

\begin{lemma}[Fluctuation averaging] \label{abstractdecoupling}
Suppose $\Phi$ and $\Phi_o$ are positive, $N$-dependent deterministic functions on $S(c_0,C_0,\epsilon)$ satisfying $N^{-1/2} \le \Phi, \Phi_o \le N^{-c}$ for some constant $c>0$. Suppose moreover that $\Lambda \prec \Phi$ and $\Lambda_o \prec \Phi_o$. Then for all $z \in S(c_0,C_0,\epsilon)$ we have
\begin{equation}\label{flucaver_ZZ}
\left|[Z]_1 \right| + \left|[Z]_2 \right| \prec   {\Phi _o^2}.
\end{equation}
\end{lemma}
\begin{proof}
We suppose that the event $\Xi$ holds. The bound \eqref{flucaver_ZZ} can be proved in a similar way as \cite[Lemma 4.9]{isotropic} and \cite[Theorem 4.7]{Semicircle}. Take $[Z]_1$ as an example. The only complication of the proof is that the coefficients ${\sigma_i}/{(1+\sigma_i m_2)^2}$ are random and depend on $i$. This can be dealt with by writing, for any $i\in \mathcal I_1$,
$$m_2 = m_2^{(i)} + \frac{1}{N}\sum_{\mu\in\mathcal I_2} \wt \sigma_\mu \frac{G_{\mu i} G_{i\mu}}{G_{ii}} = m_2^{(i)} + \OO(\Lambda_o^2).$$
Then we write
\begin{align}
[Z]_1 & =\frac{1}{N}\sum_{i\in \mathcal I_1} \frac{\sigma_i}{\big(1+m_2^{(i)}\sigma_i \big)^2} Z_i + \OO(\Lambda_o^2)  = \frac{1}{N}\sum_{i\in \mathcal I_1} (1-\mathbb E_i)\Bigg[\frac{\sigma_i}{\big(1+m_2^{(i)}\sigma_i\big)^2}G_{ii}^{-1}\Bigg]+ \OO(\Lambda_o^2) \nonumber\\
& = \frac{1}{N}\sum_{i\in \mathcal I_1} (1-\mathbb E_i)\left[\frac{\sigma_i}{\left(1+m_2 \sigma_i\right)^2}G_{ii}^{-1}\right] + \OO(\Lambda_o^2).\label{Z1_aver}
\end{align}
Now the method to bound the first term in the line (\ref{Z1_aver}) is only a slight modification of the one in \cite{isotropic} or \cite{Semicircle}. For the proof of an even more complicated fluctuation averaging lemma, one can also refer to \cite[Lemma 4.9]{XYY_circular}. Finally, we use that $\Xi$ holds with high probability by Lemma \ref{alem_weak} to conclude the proof. 
\end{proof}

Now we give the proof of Proposition \ref{prop_entry}.

\begin{proof}[Proof of Proposition \ref{prop_entry}]
By Lemma \ref{alem_weak}, the event $\Xi$ holds with high probability. Then by Lemma \ref{alem_weak} and Lemma \ref{Z_lemma}, we can take
\be\label{initial_phio}
\Phi_o = \sqrt{\frac{\im m_{2c} + (N\eta)^{-1/4}}{N\eta}} + \frac{1}{N\eta},\quad \Phi= \frac{1}{(N\eta)^{1/4}},
\ee
 in Lemma \ref{abstractdecoupling}. 
Then (\ref{selfcons_improved}) gives
$$|f(z,m_2)| \prec\frac{ \im m_{2c} + (N\eta)^{-1/4}}{N\eta}.$$
Using Lemma \ref{stability}, we get
\be\label{m2}
|m_2-m_{2c}|\prec\frac{\im m_{2c}}{N\eta\sqrt{\kappa+\eta}}+\frac{1}{(N\eta)^{5/8}} \prec \frac{1}{(N\eta)^{5/8}} ,
\ee
where we used $\im m_{2c}=\OO(\sqrt{\kappa+\eta})$ by \eqref{Immc} in the second step. With (\ref{selfcons_improved2}) and \eqref{m2}, we get the same bound for $m_1$, which gives
\be\label{m1}
\theta \prec {(N\eta)^{-5/8}} ,
\ee
Then using Lemma \ref{Z_lemma} and (\ref{m1}), we obtain that
\begin{align}\label{1iteration}
\Lambda_o \prec  \sqrt{\frac{\im m_{2c} + (N\eta)^{-5/8}}{N\eta}} + \frac{1}{N\eta}
\end{align}
uniformly in $z\in S(c_0,C_0,\epsilon)$, which is a better bound than the one in (\ref{initial_phio}). Taking the RHS of \eqref{1iteration} as the new $\Phi_o$, we can obtain an even better bound for $\Lambda_o$. Iterating the above arguments, we get the bound
$$\theta \prec \left({N\eta}\right)^{-\sum_{k=1}^l 2^{-k} - 2^{-l-2} }$$
after $l$ iterations. This implies 
\be\label{aver_proof}
\theta\prec(N\eta)^{-1}
\ee
since $l$ can be arbitrarily large. Now with \eqref{aver_proof}, Lemma \ref{Z_lemma}, \eqref{Gii0} and \eqref{Gmumu0}, we can obtain \eqref{entry_diagonal}. 
\end{proof}

\subsection{Proof of Proposition \ref{prop_diagonal}}

We now can finish the proof of Proposition \ref{prop_diagonal} using Proposition \ref{prop_entry}. By \eqref{Gii0} and \eqref{aver_proof}, we have
\be\label{addm}m=\frac1n\sum_i \frac{1}{-z(1 + \sigma_i m_2)} - \frac1n\sum_i \frac{Z_i}{z\left(1 +  \sigma_i m_2\right)^2}  +O_\prec\left(\Psi^2\right).\ee
Using the same method as in Lemma \ref{abstractdecoupling}, we can obtain that
$$\left|\frac1n\sum_i \frac{Z_i}{\left(1 +  \sigma_i m_2\right)^2}\right| \prec \Psi^2.$$
Together with \eqref{def_mc}, \eqref{Piii} and \eqref{aver_proof}, we get that
$$|m-m_c| \prec (N\eta)^{-1} + \Psi^2 \prec (N\eta)^{-1},$$
where we used \eqref{psi12} in the second step. This proves \eqref{aver_diagonal}. 

For $z\in S_{out}(c_0,C_0,\epsilon): =S(c_0,C_0,\epsilon)\cap \{z=E+\ii\eta: E\ge \lambda_+, N\eta\sqrt{\kappa + \eta} \ge N^\epsilon\}$, we have
$$\Psi^2 \le 2\left[{\frac{\Im \, m_{2c}(z)}{{N\eta }} } + \frac{1}{(N\eta)^2}\right] \lesssim \frac{1}{N\sqrt{\kappa+\eta}} + \frac{1}{(N\eta)^2} ,$$
where we used \eqref{Immc} in the second step. Thus by \eqref{addm}, to prove \eqref{aver_out}, it suffices to show that 
\be\label{aver_proof2}
|m_2-m_{2c}|\prec \frac{1}{N(\kappa +\eta)} + \frac{1}{(N\eta)^2\sqrt{\kappa +\eta}} ,\quad z\in S_{out}(c_0,C_0,\epsilon).
\ee
In fact, taking $\Phi_o=\Phi=\Psi$ in Lemma \ref{abstractdecoupling} and then using Lemma \ref{stability}, we get that
$$|m_2-m_{2c}|\prec \frac{\Psi^2}{\sqrt{\kappa + \eta}} \lesssim \frac{1}{N(\kappa +\eta)} + \frac{1}{(N\eta)^2\sqrt{\kappa +\eta}}.$$
This finishes the proof of \eqref{aver_proof2}, and hence \eqref{aver_out}.

Finally, with (\ref{entry_diagonal}), one can repeat the polynomialization method in \cite[Section 5]{isotropic} to get the anisotropic local law (\ref{aniso_diagonal}). The only difference is that one need to use the first bound in \eqref{assm3}.

\section{Proof of Theorem \ref{LEM_SMALL}: self-consistent comparison}\label{sec_comparison}

In this section, we finish the proof of Theorem \ref{LEM_SMALL} for a general $X$ satisfying \eqref{conditionA2}, \eqref{assm_3moment} and the bounded support condition (\ref{eq_support}) with $q\le N^{-\phi}$ for some constant $\phi>0$. 
Proposition \ref{prop_diagonal} implies that \eqref{aniso_law} holds for Gaussian $X^{Gauss}$ as discussed below Theorem \ref{LEM_SMALL}. Thus the basic idea of this section is to prove that for $X$ satisfying the assumptions in Theorem \ref{LEM_SMALL}, 
\begin{equation*}
\left\langle \mathbf u, \left( G(X,z) -  G(X^{Gauss},z)\right) \mathbf v\right\rangle \prec q+\Psi(z)
\end{equation*}
uniformly for deterministic unit vectors $\mathbf u,\mathbf v\in{\mathbb C}^{\mathcal I}$ and $z\in \wt S(c_0,C_0,\fa,\e)$. 
 
For simplicity of notations, we introduce the following notion of generalized entries. For $\mathbf v,\mathbf w \in \mathbb C^{\mathcal I}$ and $a\in \mathcal I$, we shall denote
\begin{equation}
G_{\mathbf{vw}}:=\langle \mathbf v,G\mathbf w\rangle, \quad G_{\mathbf{v}a}:=\langle \mathbf v,G\mathbf e_a\rangle, \quad G_{a\mathbf{w}}:=\langle \mathbf e_a,G\mathbf w\rangle,
\end{equation}
where $\mathbf e_a$ is the standard unit vector along $a$-th axis. Given vectors $\mathbf x\in \mathbb C^{\mathcal I_1}$ and $\mathbf y\in \mathbb C^{\mathcal I_2}$, we always identify them with their natural embeddings $\left( {\begin{array}{*{20}c}
   {\mathbf x}  \\
   0 \\
\end{array}} \right)$ and $\left( {\begin{array}{*{20}c}
   0  \\
   \mathbf y \\
\end{array}} \right)$ in $\mathbb C^{\mathcal I}$.
The exact meanings will be clear from the context. Now similar to Lemma \ref{lemma_Im}, we can prove the following estimates for $\mathcal G$.

\begin{lemma}\label{lem_comp_gbound}
For $i\in \mathcal I_1$ and $\mu\in \mathcal I_2$, we define $\mathbf u_i=U^* \mathbf e_i  \in \mathbb C^{\mathcal I_1}$ and $\mathbf v_\mu=V^* \mathbf e_\mu  \in \mathbb C^{\mathcal I_2}$, i.e. $\mathbf u_i$ is the $i$-th row vector of $U$ and $\mathbf v_\mu$ is the $\mu$-th row vector of $V$. Let $\mathbf x \in \mathbb C^{\mathcal I_1}$ and $\mathbf y \in \mathbb C^{\mathcal I_2}$. Then we have 
  \begin{align}
 & \sum_{i \in \mathcal I_1 }  \left| {G_{\mathbf x \mathbf u_i} } \right|^2  =\sum_{i \in \mathcal I_1 }  \left| {G_{ \mathbf u_i \mathbf x} } \right|^2  = \frac{|z|^2}{\eta}\im\left(\frac{ G_{\mathbf x\mathbf x}}{z}\right) , \label{eq_sgsq2} \\
& \sum_{\mu  \in \mathcal I_2 } {\left| {G_{\mathbf y \mathbf v_\mu } } \right|^2 }=\sum_{\mu  \in \mathcal I_2 } {\left| {G_{\mathbf v_\mu \mathbf y } } \right|^2 }  = \frac{{\im G_{\mathbf y\mathbf y} }}{\eta }, \label{eq_sgsq1}\\ 
& \sum_{i \in \mathcal I_1 } {\left| {G_{\mathbf y \mathbf u_i} } \right|^2 } =\sum_{i \in \mathcal I_1 } {\left| {G_{ \mathbf u_i \mathbf y} } \right|^2 } = {G}_{\mathbf y\mathbf y}  +\frac{\bar z}{\eta} \im G_{\mathbf y\mathbf y}  , \label{eq_sgsq3} \\
& \sum_{\mu \in \mathcal I_2 } {\left| {G_{\mathbf x \mathbf v_\mu} } \right|^2 }= \sum_{\mu \in \mathcal I_2 } {\left| {G_{\mathbf v_\mu \mathbf x } } \right|^2 }= \frac{G_{\mathbf x\mathbf x}}{z}  + \frac{\bar z}{\eta} \im \left(\frac{G_{\mathbf x\mathbf x}}{z}\right) .\label{eq_sgsq4}
 \end{align}
 All of the above estimates remain true for $G^{(\mathbb T)}$ instead of $G$ for any $\mathbb T \subseteq \mathcal I$. 
\end{lemma}
\begin{proof}
We only prove \eqref{eq_sgsq1} and \eqref{eq_sgsq3}. The proof for \eqref{eq_sgsq2} and \eqref{eq_sgsq4} is very similar. With  \eqref{spectral1}, we get that
\begin{align}\label{middle}
\sum_{\mu  \in \mathcal I_2 } {\left| {G_{\mathbf y \mathbf v_\mu } } \right|^2 } =& \sum_{\mu  \in \mathcal I_2 } \left\langle \mathbf y,G {\mathbf v_\mu  } \right\rangle \left\langle {\mathbf v_\mu}, G^\dag \mathbf y \right\rangle  = \sum_{k = 1}^N {\frac{{\left| {\left\langle {\mathbf y,\zeta _k } \right\rangle } \right|^2  }}{{\left( {\lambda _k  - E} \right)^2  + \eta ^2 }} }   =\frac{{\im  G_{\mathbf y\mathbf y} }}{\eta }.
\end{align}
For simplicity, we denote $Y:=\Sig^{1/2} U^{*}X V\wt \Sig^{1/2}$. Then with \eqref{green2} and \eqref{spectral2}, we get that
\begin{align*}
 \sum_{i \in \mathcal I_1 } {\left| {G_{\mathbf y\mathbf u_i} } \right|^2 } =  \left( {{\mathcal G_2} Y^\dag Y \mathcal G_2^\dag  } \right)_{\mathbf y\mathbf y}=  \left( {{\mathcal G_2} \left(Y^\dag Y-\bar z\right) \mathcal G_2^\dag  } \right)_{\mathbf y\mathbf y} + \bar z \left( {{\mathcal G_2} \mathcal G_2^\dag  } \right)_{\mathbf y\mathbf y} =  {G}_{\mathbf y\mathbf y}  +\frac{\bar z}{\eta} \im G_{\mathbf y\mathbf y}  ,
 \end{align*}
 where we used $\mathcal G_2^\dag= \left(Y^\dag Y-\bar z\right)^{-1}$ and \eqref{middle} in the last step.
\end{proof}

Our proof basically follows the arguments in \cite[Section 7]{Anisotropic} with some modifications. Thus we will not give all the details. We first focus on proving the anisotropic local law \eqref{aniso_law}, and the proof of \eqref{aver_in1}-\eqref{aver_out1} will be given at the end of this section. By polarization, to prove \eqref{aniso_law} it suffices to prove that 
 \begin{equation}\label{goal_ani2}
\left\langle \mathbf v, \left(G(X,z)- \Pi(z)\right) \mathbf v \right\rangle \prec q+\Psi(z)
\end{equation}
uniformly in $z\in \wt S(c_0,C_0,\fa,\e)$ and any deterministic unit vector $ \mathbf v\in{\mathbb C}^{\mathcal I}$. In fact, we can obtain the more general bound \eqref{aniso_law}
by applying (\ref{goal_ani2}) to the vectors $\mathbf u + \mathbf v$ and $\mathbf u + i\mathbf v$, respectively.


The proof consists of a bootstrap argument from larger scales to smaller scales in multiplicative increments of $N^{-\delta}$, where
\begin{equation}
 \delta \in\left(0,\frac{\min\{\epsilon,\fa,\phi\}}{2C_a}\right). \label{assm_comp_delta}
\end{equation}
Here $\e,\fa>0$ are the constants in $\wt S(c_0,C_0,\fa,\e)$, $\phi>0$ is a constant such that $q\le N^{-\phi}$, $C_a> 0$ is an absolute constant that will be chosen large enough in the proof. For any $\eta\ge N^{-1+\e}$, we define
\begin{equation}\label{eq_comp_eta}
\eta_l:=\eta N^{\delta l} \text{ for } \ l=0,...,L-1,\ \ \ \eta_L:=1.
\end{equation}
where $L\equiv L(\eta):=\max\left\{l\in\mathbb N|\ \eta N^{\delta(l-1)}<1\right\}.$ Note that $L\le \delta^{-1}$.

By (\ref{eq_gbound}), the function $z\mapsto G(z)- \Pi(z)$ is Lipschitz continuous in $\wt S(c_0,C_0,\fa,\e)$ with Lipschitz constant bounded by $N^2$. Thus to prove (\ref{goal_ani2}) for all $z\in \wt S(c_0,C_0,\fa,\e)$, it suffices to show that (\ref{goal_ani2}) holds for all $z$ in some discrete but sufficiently dense subset ${\mathbf S} \subset \wt S(c_0,C_0,\fa,e)$. We will use the following discretized domain $\bS$.
\begin{definition}
Let $\mathbf S$ be an $N^{-10}$-net of $\wt S(c_0,C_0,\fa,\e)$ such that $ |\mathbf S |\le N^{20}$ and
\[E+\ii\eta\in\mathbf S\Rightarrow E+\ii\eta_l\in\mathbf S\text{ for }l=1,...,L(\eta).\]
\end{definition}

The bootstrapping is formulated in terms of two scale-dependent properties ($\bA_m$) and ($\bC_m$) defined on the subsets
\[\mathbf S_m:=\left\{z\in\mathbf S\mid\text{Im} \, z\ge N^{-\delta m}\right\}.\]
${(\bA_m)}$ For all $z\in\mathbf S_m$, all deterministic unit vectors $\mathbf x \in \mathbb C^{\mathcal I_1}$ and $\mathbf y \in \mathbb C^{\mathcal I_2}$, and all $X$ satisfying the assumptions in Theorem \ref{LEM_SMALL}, we have
\begin{equation}\label{eq_comp_Am}
 \im \left(\frac{G_{\mathbf x\mathbf x}(z)}{z}\right) + \im G_{\mathbf y\mathbf y}(z)\prec \im m_{2c}(z) +N^{C_a\delta}(q+\Psi(z)).
\end{equation}
${(\bC_m)}$ For all $z\in\mathbf S_m$, all deterministic unit vector $\mathbf v\in \mathbb C^{\mathcal I}$, and all $X$ satisfying the assumptions in Theorem \ref{LEM_SMALL}, 
we have
\begin{equation}\label{eq_comp_Cm}
 \left|G_{\mathbf v\mathbf v}(z)-\Pi_{\mathbf v\mathbf v}(z)\right|\prec N^{C_a\delta}(q+\Psi(z)).
\end{equation}
It is trivial to see that ${(\mathbf A_0)}$ holds by \eqref{eq_gbound} and \eqref{Immc}. Moreover, it is easy to observe the following result.

\begin{lemma}\label{lemm_boot2}
For any $m$, property ${(\mathbf C_m)}$ implies property $(\mathbf A_m)$.
\end{lemma}
\begin{proof}
By \eqref{Immc}, \eqref{Piii} and the definition of $\Pi$ in \eqref{defn_pi}, it is easy to get that 
$$\im \left(\frac{\Pi_{\mathbf x\mathbf x}(z)}{z}\right) + \im \Pi_{\mathbf y\mathbf y}(z)\lesssim \im m_{2c}(z) ,$$
which finishes the proof.
\end{proof}

The key step is the following induction result.
\begin{lemma}\label{lemm_boot}
For any $1\le m\le \delta^{-1}$, property $(\mathbf A_{m-1})$ implies property $(\mathbf C_m)$.
\end{lemma}

Combining Lemmas \ref{lemm_boot2} and \ref{lemm_boot}, we conclude that (\ref{eq_comp_Cm}) holds for all $w\in\mathbf S$. Since $\delta$ can be chosen arbitrarily small under the condition (\ref{assm_comp_delta}), we conclude that (\ref{goal_ani2}) holds for all $w\in\mathbf S$, and \eqref{aniso_law} follows for all $z\in \wt S(c_0,C_0,\fa,\e)$. What remains now is the proof of Lemma \ref{lemm_boot}. Denote
\begin{equation}\label{eq_comp_F(X)}
 F_{\mathbf v}(X,z):=\left|G_{\mathbf{vv}}(X,z)-\Pi_{\mathbf {vv}}(z)\right|.
\end{equation}
By Markov's inequality, it suffices to prove the following lemma.
\begin{lemma}\label{lemm_comp_0}
Fix $p\in \mathbb N$ and $m\le \delta^{-1}$. Suppose that the assumptions of Theorem \ref{LEM_SMALL} and property $(\mathbf A_{m-1})$ hold. Then we have
 \begin{equation}
  \mathbb EF_{\mathbf v}^p(X,z)\le\left[ N^{C_a\delta}\left(q+\Psi(z)\right)\right]^p
 \end{equation}
 for all $z\in{\mathbf S}_m$ and any deterministic unit vector $\mathbf v$.
\end{lemma}
In the rest of this section, we focus on proving Lemma \ref{lemm_comp_0}. 
First, in order to make use of the assumption $(\mathbf A_{m-1})$, which has spectral parameters in $\mathbf S_{m-1}$, to get some estimates for $G$ with spectral parameters in $\mathbf S_{m}$, we shall use the following rough bounds for $ G_{\mathbf{xy}}$.

\begin{lemma}\label{lemm_comp_1}
For any $z=E+\ii\eta\in\mathbf S$ and unit vectors $\mathbf x,\mathbf y\in \mathbb C^{\mathcal I}$,  we have 
\begin{align*}
\left|G_{\mathbf x\mathbf y}(z)-\Pi_{\mathbf x\mathbf y}(z)\right|\prec & N^{2\delta}\sum_{l=1}^{L(\eta)} \left[\im \left(\frac{G_{\mathbf x_1\mathbf x_1}(E+\ii\eta_l)}{E+\ii\eta_l}\right)+\im G_{\mathbf x_2\mathbf x_2}(E+\ii\eta_l) \right.\\
& \left. +\im \left(\frac{G_{\mathbf y_1\mathbf y_1}(E+\ii\eta_l)}{E+\ii\eta_l}\right)+\im G_{\mathbf y_2\mathbf y_2}(E+\ii\eta_l)\right]+1,
\end{align*}
where $\mathbf x=\left( {\begin{array}{*{20}c}
   {\mathbf x}_1   \\
   {\mathbf x}_2 \\
   \end{array}} \right)$ and $\mathbf y=\left( {\begin{array}{*{20}c}
   {\mathbf y}_1   \\
   {\mathbf y}_2 \\
   \end{array}} \right)$ for ${\mathbf x}_1,{\mathbf y}_1\in\mathbb C^{\mathcal I_1}$ and ${\mathbf x}_2,{\mathbf y}_2\in\mathbb C^{\mathcal I_2}$, and $\eta_l$ is defined in (\ref{eq_comp_eta}).
\end{lemma}
\begin{proof} The proof is the same as the one for \cite[Lemma 7.12]{Anisotropic}.\end{proof}

Recall that for a given family of random matrices $\cal M$, we use $\cal M=\OO_\prec(\zeta)$ to mean $\left|\left\langle\mathbf v, \cal M\mathbf w\right\rangle\right|\prec\zeta \| \mathbf v\|_2 \|\mathbf w\|_2 $ uniformly in any deterministic vectors $\mathbf v$ and $\mathbf w$ (see Definition \ref{stoch_domination} (ii)).

\begin{lemma}\label{lemm_comp_2}
Suppose $(\mathbf A_{m-1})$ holds, then
 \begin{equation}\label{eq_comp_apbound}
  G(z)-\Pi(z)=\OO_{\prec}(N^{2\delta}),
 \end{equation}
 and
\begin{equation}\label{eq_comp_apbound2}
\im \left(\frac{G_{\mathbf x\mathbf x}(z)}{z}\right) + \im G_{\mathbf y\mathbf y}(z) \prec N^{2\delta}\left[ \im m_{2c}(z)+N^{C_a\delta}(q+\Psi(z))\right],
\end{equation}
 for all $z \in \mathbf S_m$ and any deterministic unit vectors $\mathbf x \in \mathbb C^{\mathcal I_1}$ and $\mathbf y \in \mathbb C^{\mathcal I_2}$.
\end{lemma}
\begin{proof} The proof is the same as the one for \cite[Lemma 7.13]{Anisotropic}.\end{proof}

Now we are ready to perform the self-consistent comparison. 
We divide the proof into three subsections. In Sections \ref{subsec_interp}-\ref{section_words}, we prove Lemma \ref{lemm_comp_0} under the condition 
\be\label{3moment}
\mathbb E x_{ij}^3=0, \quad 1\le i \le n,\ \  1\le j \le N,
\ee
for $z\in S(c_0,C_0,\e)$. Then in Section \ref{subsec_3moment}, we show how to relax \eqref{3moment} to \eqref{assm_3moment} for $z\in \wt S(c_0,C_0,\fa,\e)$.

\begin{subsection}{Interpolation and expansion} \label{subsec_interp}
\begin{definition}[Interpolating matrices]
Introduce the notations $X^0:=X^{Gauss}$ and $X^1:=X$. Let $\rho_{i\mu}^0$ and $\rho_{i\mu}^1$ be the laws of $X_{i\mu}^0$ and $X_{i\mu}^1$, respectively. For $\theta\in [0,1]$, we define the interpolated law
$$\rho_{i\mu}^\theta := (1-\theta)\rho_{i\mu}^0+\theta\rho_{i\mu}^1.$$
\nc Let $\{X^\theta: \theta\in (0,1) \}$ be a collection of random matrices such that the following properties hold. For any fixed $\theta\in (0,1)$, $(X^0,X^\theta, X^1)$ is a triple of independent $\mathcal I_1\times \mathcal I_2$ random matrices, \nc and the matrix $X^\theta=(X_{i\mu}^\theta)$ has law
\begin{equation}\label{law_interpol}
\prod_{i\in \mathcal I_1}\prod_{\mu\in \mathcal I_2} \rho_{i\mu}^\theta(\dd X_{i\mu}^\theta).
\end{equation}
\nc Note that we do not require $X^{\theta_1}$ to be independent of $X^{\theta_2}$ for $\theta_1\ne \theta_2 \in (0,1)$. \nc
For $\lambda \in \mathbb R$, $i\in \mathcal I_1$ and $\mu\in \mathcal I_2$, we define the matrix $X_{(i\mu)}^{\theta,\lambda}$ through
\[\left(X_{(i\mu)}^{\theta,\lambda}\right)_{j\nu}:=\begin{cases}X_{i\mu}^{\theta}, &\text{ if }(j,\nu)\ne (i,\mu)\\ \lambda, &\text{ if }(j,\nu)=(i,\mu)\end{cases}.\]
We also introduce the matrices \[G^{\theta}(z):=G\left(X^{\theta},z\right),\ \ \ G^{\theta, \lambda}_{(i\mu)}(z):=G\left(X_{(i\mu)}^{\theta,\lambda},z\right).\]
\end{definition}

We shall prove Lemma \ref{lemm_comp_0} through interpolation matrices $X^\theta$ between $X^0$ and $X^1$. It holds for $X^0$ by Proposition \ref{prop_diagonal}.
\begin{lemma}\label{Gaussian_case}
Lemma \ref{lemm_comp_0} holds if $X=X^0$.
\end{lemma}

Using (\ref{law_interpol}) and fundamental calculus, we get the following basic interpolation formula.
\begin{lemma}\label{lemm_comp_3}
 For $F:\mathbb R^{\mathcal I_1 \times\mathcal I_2}\rightarrow \mathbb C$ we have
\begin{equation}\label{basic_interp}
\frac{\dd}{\dd\theta}\mathbb E F(X^\theta)=\sum_{i\in\mathcal I_1}\sum_{\mu\in\mathcal I_2}\left[\mathbb E F\left(X^{\theta,X_{i\mu}^1}_{(i\mu)}\right)-\mathbb E F\left(X^{\theta,X_{i\mu}^0}_{(i\mu)}\right)\right]
\end{equation}
 provided all the expectations exist.
\end{lemma}

We shall apply Lemma \ref{lemm_comp_3} to $F(X)=F_{\mathbf v}^p(X,z)$ with $F_{\mathbf v}(X,z)$ defined in (\ref{eq_comp_F(X)}). The main work is devoted to proving the following self-consistent estimate for the right-hand side of (\ref{basic_interp}).

\begin{lemma}\label{lemm_comp_4}
 Fix $p\in 2\mathbb N$ and $m\le \delta^{-1}$. Suppose (\ref{3moment}) and $\mathbf{(A_{m-1})}$ hold, then we have
 \begin{equation}
  \sum_{i\in\mathcal I_1}\sum_{\mu\in\mathcal I_2}\left[\mathbb EF_{\mathbf v}^p\left(X^{\theta,X_{i\mu}^1}_{(i\mu)},z\right)-\mathbb EF_{\mathbf v}^p\left(X^{\theta,X_{i\mu}^0}_{(i\mu)},z\right)\right]=
  \OO\left(\left[N^{C_a\delta} (q+\Psi(z))\right]^p+\mathbb EF_{\mathbf v}^p(X^\theta,z)\right)
 \end{equation}
 for all $\theta\in[0,1]$, $z\in\mathbf S_m$ and any deterministic unit vector $\mathbf v$.
\end{lemma}
Combining Lemmas \ref{Gaussian_case}-\ref{lemm_comp_4} with a Gr\"onwall's argument, we can conclude Lemma \ref{lemm_comp_0} and hence \eqref{goal_ani2} by Markov's inequality. 
In order to prove Lemma \ref{lemm_comp_4}, we compare $X^{\theta,X_{i\mu}^0}_{(i\mu)}$ and $X^{\theta,X_{i\mu}^1}_{(i\mu)}$ via a common $X^{\theta,0}_{(i\mu)}$, i.e. we will prove that
\begin{equation}\label{lemm_comp_5}
\sum_{i\in\mathcal I_1}\sum_{\mu\in\mathcal I_2}\left[\mathbb EF_{\mathbf v}^p\left(X^{\theta,X_{i\mu}^u}_{(i\mu)},z\right)-\mathbb EF_{\mathbf v}^p\left(X^{\theta,0}_{(i\mu)},z\right)\right]= \OO\left(\left[N^{C_a\delta} (q+\Psi(z))\right]^p+\mathbb EF_{\mathbf v}^p(X^\theta,z)\right)
 \end{equation}
for all $u\in \{0,1\}$, $\theta\in[0,1]$, $z\in \mathbf{S}_m$, and any deterministic unit vector $\mathbf v$.

Underlying the proof of (\ref{lemm_comp_5}) is an expansion approach which we will describe below. 
During the proof, we always assume that $(\mathbf A_{m-1})$ holds. Also the rest of the proof is performed at a fixed $z\in \mathbf S_m$. We define the $\mathcal I \times \mathcal I$ matrix $\Delta_{(i\mu)}^\lambda$ as
\begin{equation}\label{deltaimu}
\Delta_{(i\mu)}^{\lambda} :=\lambda \left( {\begin{array}{*{20}c}
   { 0 } & \Sig^{1/2} \mathbf u_i \mathbf v_\mu^* \wt \Sig^{1/2}   \\
   { \wt \Sig^{1/2}\mathbf v_\mu \mathbf u_i^* \Sig^{1/2} } & {0}  \\
   \end{array}} \right),
\end{equation}
where we recall the definitions of $\mathbf u_i$ and $\mathbf v_\mu$ in Lemma \ref{lem_comp_gbound}. Then we have for any $\lambda,\lambda'\in \mathbb R$ and $K\in \mathbb N$,
\begin{equation}\label{eq_comp_expansion}
\G_{(i \mu)}^{\theta,\lambda'} = G_{(i\mu)}^{\theta,\lambda}+\sum_{k=1}^{K}  G_{(i\mu)}^{\theta,\lambda}\left( \Delta_{(i\mu)}^{\lambda-\lambda'} G_{(i\mu)}^{\theta,\lambda}\right)^k+ G_{(i\mu)}^{\theta,\lambda'}\left(\Delta_{(i\mu)}^{\lambda-\lambda'} G_{(i\mu)}^{\theta,\lambda}\right)^{K+1}.
\end{equation}
The following result provides a priori bounds for the entries of $G_{(i\mu)}^{\theta,\lambda}$.
\begin{lemma}\label{lemm_comp_6}
 Suppose that $y$ is a random variable satisfying $|y|\prec q$. Then
 \begin{equation}\label{comp_eq_apriori}
   G_{(i\mu)}^{\theta,y}-\Pi=\OO_{\prec}(N^{2\delta}),\quad i\in\sI_1, \ \mu\in\sI_2.
 \end{equation}
\end{lemma}
\begin{proof} The proof is the same as the one for \cite[Lemma 7.14]{Anisotropic}. \end{proof}

In the following proof, for simplicity of notations, we introduce $f_{(i\mu)}(\lambda):=F_{\mathbf v}^p(X_{(i\mu)}^{\theta, \lambda})$. We use $f_{(i\mu)}^{(r)}$ to denote the $r$-th derivative of $f_{(i\mu)}$. With Lemma \ref{lemm_comp_6} and (\ref{eq_comp_expansion}), it is easy to prove the following result.
\begin{lemma}
Suppose that $y$ is a random variable satisfying $|y|\prec q$. Then for fixed $r\in\bbN$,
  \begin{equation}
  \left|f_{(i\mu)}^{(r)}(y)\right|\prec N^{2\delta(r+p)}.
 \end{equation}
\end{lemma}
By this lemma, the Taylor expansion of $f_{(i\mu)}$ gives
\begin{equation}\label{eq_comp_taylor}
f_{(i\mu)}(y)=\sum_{r=0}^{4p+4}\frac{y^r}{r!}f^{(r)}_{(i\mu)}(0)+\OO_\prec\left( q^{p+4}\right),
\end{equation}
provided $C_a$ is chosen large enough in (\ref{assm_comp_delta}). Therefore we have for $u\in\{0,1\}$,
\begin{align}
&\mathbb EF_{\mathbf v}^p\left(X^{\theta,X_{i\mu}^u}_{(i\mu)}\right)-\mathbb EF_{\mathbf v}^p\left(X^{\theta,0}_{(i\mu)}\right)=\bbE\left[f_{(i\mu)}\left(X_{i\mu}^u\right)-f_{(i\mu)}(0)\right] \nonumber\\
& =\bbE f_{(i\mu)}(0)+\frac{1}{2N}\bbE f_{(i\mu)}^{(2)}(0)+\sum_{r=4}^{4p+4}\frac{1}{r!}\bbE f^{(r)}_{(i\mu)}(0)\bbE\left(X_{i\mu}^u\right)^r+\OO_\prec(q^{p+4}),\label{only3rd}
\end{align}
where we used that $X_{i\mu}^u$ has vanishing first and third moments and its variance is $1/N$. (Note that this is the only place where we need the condition \eqref{3moment}.) By \eqref{conditionA2} and the bounded support condition, we have
\be\label{moment-4}
\left|\bbE\left(X_{i\mu}^u\right)^r\right| \prec N^{-2}q^{r-4} , \quad r \ge 4.
\ee
Thus to show (\ref{lemm_comp_5}), we only need to prove for $r=4,5,...,4p+4$,
\begin{equation}\label{eq_comp_est}
N^{-2}q^{r-4}\sum_{i\in\mathcal I_1}\sum_{\mu\in\mathcal I_2}\left|\bbE f^{(r)}_{(i\mu)}(0)\right|=\OO\left(\left[N^{C_a\delta} (q+\Psi)\right]^p+\mathbb EF_{\mathbf v}^p(X^\theta,z)\right).\end{equation}
In order to get a self-consistent estimate in terms of the matrix $X^\theta$ on the right-hand side of (\ref{eq_comp_est}), we want to replace $X^{\theta,0}_{(i\mu)}$ in $f_{(i\mu)}(0)=F_{\mathbf v}^p(X_{(i\mu)}^{\theta, 0})$ with $X^\theta = X_{(i\mu)}^{\theta, X_{i\mu}^\theta}$. 
\begin{lemma}
Suppose that
\begin{equation}\label{eq_comp_selfest}
N^{-2}q^{r-4}\sum_{i\in\mathcal I_1}\sum_{\mu\in\mathcal I_2}\left|\bbE f^{(r)}_{(i\mu)}(X_{i\mu}^\theta)\right|=\OO\left(\left[N^{C_a\delta} (q+\Psi)\right]^p+\mathbb EF_{\mathbf v}^p(X^\theta,z)\right)
\end{equation}
holds for $r=4,...,4p+4$. Then (\ref{eq_comp_est}) holds for $r=4,...,4p+4$.
\end{lemma}
\begin{proof}
We abbreviate $ f_{(i\mu)}\equiv f$ and $X_{i\mu}^\theta \equiv \xi$. Then with (\ref{eq_comp_taylor}) we can get
\begin{equation}\label{eq_comp_taylor2}
\E f^{(l)}(0)=\E f^{(l)}(\xi)-\sum_{k=1}^{4p+4-l}\E f^{(l+k)}(0)\frac{\E \xi^k}{k!}+\OO_\prec(q^{p+4-l}).
\end{equation}
The estimate \eqref{eq_comp_est} then follows from a repeated application of (\ref{eq_comp_taylor2}).  Fix $r=4,...,4p+4$. Using (\ref{eq_comp_taylor2}), we get
\begin{align*}
\mathbb E f^{(r)}(0)&=\mathbb E f^{(r)}(\xi) - \sum_{k_1\ge 1} \mathbf 1(r+k_1 \le 4p +4)\mathbb E f^{(r+k_1)}(0) \frac{\mathbb E\xi^{k_1}}{k_1!}+\OO_\prec(q^{p+4-r}) \\
&=\mathbb E f^{(r)}(\xi) - \sum_{k_1\ge 1} \mathbf 1(r+k_1 \le 4p +4)\mathbb E f^{(r+k_1)}(\xi) \frac{\mathbb E\xi^{k_1}}{k_1!} \\
&+\sum_{k_1,k_2\ge 1} \mathbf 1(r+k_1+k_2 \le 4p +4)\mathbb E f^{(r+k_1+k_2)}(0) \frac{\mathbb E\xi^{k_1}}{k_1!} \frac{\mathbb E\xi^{k_2}}{k_2!} + \OO_\prec(q^{p+4-r}) \\
&=\cdots=\sum_{t=0}^{4p+4-r}(-1)^t \sum_{k_1,\cdots, k_t \ge 1}\mathbf 1\left(r+\sum_{j=1}^t k_j \le 4p +4\right)\mathbb E f^{(r+\sum_{j=1}^t k_j)}(\xi)\prod_{j=1}^t \frac{\mathbb E\xi^{k_j}}{k_j!} + \OO_\prec(q^{p+4-r}).
\end{align*}
The lemma now follows easily by using \eqref{moment-4}.
\end{proof}
\end{subsection}

\begin{subsection}{Conclusion of the proof with words}\label{section_words}

What remains now is to prove (\ref{eq_comp_selfest}). For simplicity, we abbreviate $X^\theta \equiv X$. 
In order to exploit the detailed structure of the derivatives on the left-hand side of (\ref{eq_comp_selfest}), we introduce the following algebraic objects.

\begin{definition}[Words]\label{def_comp_words}
Given $i\in \mathcal I_1$ and $\mu\in \mathcal I_2$, let $\sW$ be the set of words of even length in two letters $\{\mathbf i, \bm{\mu}\}$. We denote the length of a word $w\in\sW$ by $2{\bm l}(w)$ with ${\bm l}(w)\in \mathbb N$. We use bold symbols to denote the letters of words. For instance, $w=\mathbf t_1\mathbf s_2\mathbf t_2\mathbf s_3\cdots\mathbf t_r\mathbf s_{r+1}$ denotes a word of length $2r$.
Define $\sW_r:=\{w\in \mathcal W: {\bm l}(w)=r\}$ to be the set of words of length $2r$, and such that
each word $w\in \sW_r$ satisfies that $\mathbf t_l\mathbf s_{l+1}\in\{\mathbf i\bm{\mu},\bm{\mu}\mathbf i\}$ for all $1\le l\le r$.

Next we assign to each letter a value $[\cdot]$ through $[\mathbf i]:=\Sigma^{1/2} \bu_i$, $[\bm {\mu}]:=\wt \Sigma^{1/2} \mathbf v_\mu,$ where $\mathbf u_i$ and $\bv_\mu$ are defined in Lemma \ref{lem_comp_gbound} and are regarded as summation indices. Note that it is important to distinguish the abstract letter from its value, which is a summation index. Finally, to each word $w$ we assign a random variable $A_{\mathbf v, i, \mu}(w)$ as follows. If ${\bm l}(w)=0$ we define
 $$A_{\mathbf v, i, \mu}(w):=G_{\mathbf v\mathbf v}-\Pi_{\mathbf v\mathbf v}.$$
 If ${\bm l}(w)\ge 1$, say $w=\mathbf t_1\mathbf s_2\mathbf t_2\mathbf s_3\cdots\mathbf t_r\mathbf s_{r+1}$, we define
 \begin{equation}\label{eq_comp_A(W)}
 A_{\mathbf v, i, \mu}(w):=G_{\bv[\mathbf t_1]} G_{[\mathbf s_2][\mathbf t_2]}\cdots G_{[\mathbf s_r][\mathbf t_r]} G_{[\mathbf s_{r+1}]\bv}.
 \end{equation}
\end{definition}

Notice the words are constructed such that, by \eqref{deltaimu} and (\ref{eq_comp_expansion}) ,
\[\left(\frac{\partial}{\partial X_{i\mu}}\right)^r \left( G_{\mathbf v\mathbf v}-\Pi_{\mathbf v\mathbf v}\right)=(-1)^r r!\sum_{w\in \mathcal W_r} A_{\mathbf v, i, \mu}(w),\quad r\in \mathbb N,\]
with which we get that
\begin{align*}
\left(\frac{\partial}{\partial X_{i\mu}}\right)^r F_{\bv}^p(X)=(-1)^r & \sum_{l_1+\cdots+l_p=r}\prod_{t=1}^{p/2}\left(l_t! l_{t+p/2}!\right) \left(\sum_{w_t\in\sW_{l_t}}\sum_{w_{t+p/2}\in\sW_{l_{t+p/2}}}A_{\mathbf v, i, \mu}(w_t)\overline{A_{\mathbf v, i, \mu}(w_{t+p/2})}\right).
\end{align*}
Then to prove (\ref{eq_comp_selfest}), it suffices to show that
\begin{equation}
N^{-2}q^{r-4}\sum_{i\in\mathcal I_1}\sum_{\mu\in\mathcal I_2}\left|\bbE\prod_{t=1}^{p/2}A_{\mathbf v, i, \mu}(w_t)\overline{A_{\mathbf v, i, \mu}(w_{t+p/2})}\right|=\OO\left(\left[N^{C_a\delta} (q+\Psi)\right]^p+\mathbb EF_{\mathbf v}^p(X,z)\right)\label{eq_comp_goal1}
\end{equation}
for  $4\le r\le 4p+4$ and all words $w_1,...,w_p\in \sW$ satisfying ${\bm l}(w_1)+\cdots+{\bm l}(w_p)=r$.
To avoid the unimportant notational complications associated with the complex conjugates, we will actually prove that
\begin{equation}\label{eq_comp_goal2}
N^{-2}q^{r-4}\sum_{i\in\mathcal I_1}\sum_{\mu\in\mathcal I_2}\left|\bbE\prod_{t=1}^{p}A_{\mathbf v, i, \mu}(w_t)\right|=\OO\left(\left[N^{C_a\delta} (q+\Psi)\right]^p+\mathbb EF_{\mathbf v}^p(X,z)\right).
\end{equation}
The proof of $(\ref{eq_comp_goal1})$ is essentially the same but with slightly heavier notations. Treating empty words separately, we find it suffices to prove
\begin{equation}
\label{eq_comp_goal3}N^{-2}q^{r-4}\sum_{i\in\mathcal I_1}\sum_{\mu\in\mathcal I_2}\bbE\left|A^{p-l}_{\mathbf v, i, \mu}(w_0)\prod_{t=1}^{l}A_{\mathbf v, i, \mu}(w_t)\right|=\OO\left(\left[N^{C_a\delta} (q+\Psi)\right]^p+\mathbb EF_{\mathbf v}^p(X,z)\right)
\end{equation}
for  $4\le r\le 4p+4$, $1\le l \le p$, and words such that ${\bm l}(w_0)=0$, $\sum_t {\bm l}(w_t)=r$ and ${\bm l}(w_t)\ge 1$ for $t\ge 1$.

To estimate (\ref{eq_comp_goal3}) we introduce the quantity
\begin{equation}\label{eq_comp_Rs}
\mathcal R_a:=|G_{\mathbf v \mathbf w_a}|+|G_{\mathbf w_a \mathbf v}|
\end{equation}
for $a\in \sI$, where $\mathbf w_i:= \Sigma^{1/2} \bu_i$ for $i\in\sI_1$ and $\mathbf w_\mu:=\wt \Sigma^{1/2} \bv_\mu$ for $\mu\in\sI_2$.

\begin{lemma}\label{lem_comp_A}
  For $w\in\sW$, we have the rough bound
  \begin{equation}
  |A_{\mathbf v, i, \mu}(w)|\prec N^{2\delta({\bm l}(w)+1)}.\label{eq_comp_A1}
  \end{equation}
  Furthermore, for ${\bm l}(w)\ge 1$ we have
  \begin{equation}
  |A_{\mathbf v, i, \mu}(w)|\prec(\mathcal R_i^2+\mathcal R_\mu^2)N^{2\delta({\bm l}(w)-1)}.\label{eq_comp_A2}
  \end{equation}
  For ${\bm l}(w)=1$, we have the better bound
  \begin{equation}
  |A_{\mathbf v, i, \mu}(w)|\prec \mathcal R_i\mathcal R_\mu.\label{eq_comp_A3}
  \end{equation}
\end{lemma}
\begin{proof}
The estimates (\ref{eq_comp_A1}) and (\ref{eq_comp_A2}) follow immediately from the rough bound (\ref{eq_comp_apbound}) and the definition (\ref{eq_comp_A(W)}).  
The estimate (\ref{eq_comp_A3}) follows from the constraint $\mathbf t_1\ne\mathbf s_2$ in the definition (\ref{eq_comp_A(W)}).
\end{proof}
By pigeonhole principle, if $r\le 2l-2$, then there exist at least two words $w_t$ with ${\bm l}(w_t)=1$. Therefore by Lemma \ref{lem_comp_A} we have
\begin{equation}\label{eq_comp_r1}
 \left|A^{p-l}_{\mathbf v, i, \mu}(w_0)\prod_{t=1}^{l}A_{\mathbf v, i, \mu}(w_t)\right|\prec N^{2\delta(r+l)}F_{\bv}^{p-l}(X)\left(\one(r\ge 2l-1)(\mathcal R_i^2+\mathcal R_\mu^2)+\one(r\le 2l-2)\mathcal R_i^2\mathcal R_\mu^2\right).
\end{equation}
Let $\mathbf v=\left( {\begin{array}{*{20}c}
   {\mathbf v}_1   \\
   {\mathbf v}_2 \\
   \end{array}} \right)$ for ${\mathbf v}_1 \in\mathbb C^{\mathcal I_1}$ and ${\mathbf v}_2\in\mathbb C^{\mathcal I_2}$. Then using Lemma \ref{lem_comp_gbound}, we get
\begin{align}
 \frac{1}{N}\sum_{i\in\sI_1}\mathcal R_i^2+ \frac{1}{N}\sum_{\mu\in\sI_2}\mathcal R_{\mu}^2 & \prec\frac{\im \left(z^{-1}G_{\mathbf v_1\mathbf v_1}\right) + \im \left(G_{\mathbf v_2\mathbf v_2}\right) +\eta \left|G_{\mathbf v_1\mathbf v_1}\right|+\eta \left|G_{\mathbf v_2\mathbf v_2}\right|}{N\eta} \nonumber\\
& \prec N^{2\delta}\frac{\im m_{2c} + N^{C_a\delta}(q+\Psi(z)) }{N\eta}\prec N^{(C_a+2)\delta} \left( \Psi^2(z) + \frac{q}{N\eta}\right), \label{eq_comp_r2}
\end{align}
where in the second step we used the two bounds in Lemma \ref{lemm_comp_2} and $\eta =\OO(\im m_{2c})$ by \eqref{Immc}, and in the last step the definition of $\Psi$ in \eqref{eq_defpsi}. Using the same method we can get
\begin{equation}\label{eq_comp_r3}
\frac{1}{N^2}\sum_{i\in\sI_1}\sum_{\mu\in\sI_2}\mathcal R_i^2\mathcal R_\mu^2\prec \left[N^{(C_a+2)\delta}\left(\Psi^2(z) + \frac{q}{N\eta}\right)\right]^2.
\end{equation}
Plugging (\ref{eq_comp_r2}) and (\ref{eq_comp_r3}) into (\ref{eq_comp_r1}), we get that the left-hand side of
(\ref{eq_comp_goal3}) is bounded by
\begin{align*}
& q^{r-4}N^{2\delta(r+l+2)}\bbE F_{\bv}^{p-l}(X)\left[\one(r\ge 2l-1)\left(N^{C_a\delta/2}(q+\Psi)\right)^2+\one(r\le 2l-2)\left(N^{C_a\delta/2}(q+\Psi)\right)^4\right]\\
& \le N^{2\delta(r+l+2)} \bbE F_{\bv}^{p-l}(X)\left[\one(r\ge 2l-1)\left(N^{C_a\delta/2}(q+\Psi)\right)^{r-2}+\one(r\le 2l-2)\left(N^{C_a\delta/2}(q+\Psi)\right)^r\right]\\
 &\le \bbE F_{\bv}^{p-l}(X)\left[\one(r\ge 2l-1)\left(N^{C_a\delta/2+12\delta}(q+\Psi)\right)^{r-2}+\one(r\le 2l-2)\left(N^{C_a\delta/2+12\delta}(q+\Psi)\right)^r\right],
\end{align*}
where we used that $l\le r$ and $r\ge 4$ in the last step. If we choose $C_a\ge 25$, then by (\ref{assm_comp_delta}) we have $N^{C_a\delta/2+12\delta} \ll \min \{ N^{\phi/2},N^{\e/2}\}$, and hence $N^{C_a\delta/2+12\delta}(q+\Psi) \ll 1$. Moreover, if $r\ge 4$ and $r\ge 2l-1$, then $r\ge l+2$. Therefore we conclude that the left-hand side of $(\ref{eq_comp_goal3})$ is bounded by
\begin{equation}
\bbE F_{\bv}^{p-l}(X)\left[N^{C_a\delta}(q+\Psi)\right]^l.
\end{equation}
Now (\ref{eq_comp_goal3}) follows from H\"older's inequality. This concludes the proof of (\ref{eq_comp_selfest}), and hence of (\ref{lemm_comp_5}), and hence of Lemma \ref{lemm_boot}. This proves \eqref{goal_ani2}, and hence \eqref{aniso_law} under the condition \eqref{3moment}.

\end{subsection}

\subsection{Non-vanishing third moment}\label{subsec_3moment}

In this subsection, we prove Lemma \ref{lemm_comp_0} under \eqref{assm_3moment} for $z\in \wt S(c_0, C_0,\fa,\e)$. Following the arguments in Section \ref{subsec_interp} and Section \ref{section_words}, we see that it suffices to prove the estimate ($\ref{eq_comp_selfest}$) in the $r=3$ case. In other words, we need to prove the following lemma. 
\begin{lemma}\label{lemm_comparison_big}
Fix $p\in 2\mathbb N$ and $m \le \delta^{-1}$. Let $z\in {\mathbf S}_m $ and suppose $(\mathbf A_{m-1})$ holds. Then 
\begin{equation}\label{eq_comp_selfest_generalX}
b_N N^{-2}\sum_{i\in\mathcal I_1}\sum_{\mu\in\mathcal I_2}\left|\bbE f^{(3)}_{(i\mu)}(X_{i\mu}^\theta)\right|=\OO\left(\left[N^{C_a\delta} (q+\Psi)\right]^p+\mathbb EF_{\mathbf v}^p(X^\theta,z)\right).
\end{equation}
\end{lemma}
\begin{proof}
The main new ingredient of the proof is a further iteration step at a fixed $z$. Suppose
\begin{equation}\label{comp_geX_iteration}
G-\Pi=\OO_\prec(\Phi)
\end{equation}
for some deterministic parameter $\Phi\equiv \Phi_N$. By the a priori bound (\ref{eq_comp_apbound}), we can take $\Phi\le N^{2\delta}$. Assuming (\ref{comp_geX_iteration}), we shall prove a self-improving bound of the form
\begin{equation}\label{comp_geX_self-improving-bound}
b_N N^{-2}\sum_{i\in\mathcal I_1}\sum_{\mu\in\mathcal I_2}\left|\bbE f^{(3)}_{(i\mu)}(X_{i\mu}^\theta)\right|=\OO\left(\left[N^{C_a\delta} (q+\Psi)\right]^p+(N^{-\fa/2}\Phi)^p+\mathbb EF_{\mathbf v}^p(X^\theta,w)\right).
\end{equation}
Once (\ref{comp_geX_self-improving-bound}) is proved, we can use it iteratively to get an increasingly accurate bound 
for $\left|G_{\mathbf{vv}}(X,z)-\Pi_{\mathbf {vv}}(z)\right|$. After each step, we obtain a better bound (\ref{comp_geX_iteration}) with $\Phi$ reduced by $N^{-\fa/2}$. Hence after $\OO(\fa^{-1})$ many iterations we obtain (\ref{eq_comp_selfest_generalX}).

As in Section \ref{section_words}, to prove (\ref{comp_geX_self-improving-bound}) it suffices to show 
\begin{equation}\label{comp_geX_words}
b_N N^{-2}\left|\sum_{i\in\mathcal I_1}\sum_{\mu\in\mathcal I_2}A^{p-l}_{\mathbf v, i, \mu}(w_0)\prod_{t=1}^{l}A_{\mathbf v, i, \mu}(w_t)\right|\prec F_{\bv}^{p-l}(X)\left[N^{(C_0-1)\delta}(q+\Psi) + N^{-\fa/2}\Phi\right]^l,
\end{equation}
which follows from the bound
\begin{equation}\label{comp_geX_words2}
b_N N^{-2}\left|\sum_{i\in\mathcal I_1}\sum_{\mu\in\mathcal I_2}\prod_{t=1}^{l}A_{\mathbf v, i, \mu}(w_t)\right|\prec \left[N^{(C_0-1)\delta}(q+\Psi) + N^{-\fa/2}\Phi\right]^l.
\end{equation}
We now list all the three cases with $l=1,\, 2,\, 3$, and discuss each case separately.  

When $l = 1$, the single factor $A_{\mathbf v, i, \mu}(w_1)$ is of the form
\[ G_{\mathbf v[\mathbf t_1]} G_{[\mathbf s_2][\mathbf t_2]} G_{[\mathbf s_3][\mathbf t_3]} G_{[\mathbf s_4]\mathbf v}.\]
Then we split it as
\begin{align}
G_{\mathbf v[\mathbf t_1]} G_{[\mathbf s_2][\mathbf t_2]} G_{[\mathbf s_3][\mathbf t_3]} G_{[\mathbf s_4]\mathbf v}
=& G_{\mathbf v[\mathbf t_1]} \Pi_{[\mathbf s_2][\mathbf t_2]} \Pi_{[\mathbf s_3][\mathbf t_3]} G_{[\mathbf s_4]\mathbf v} + G_{\mathbf v[\mathbf t_1]}\wt G_{[\mathbf s_2][\mathbf t_2]} \Pi_{[\mathbf s_3][\mathbf t_3]}G_{[\mathbf s_4]\mathbf v}\nonumber\\
 + & G_{\mathbf v[\mathbf t_1]} \Pi_{[\mathbf s_2][\mathbf t_2]} \wt G_{[\mathbf s_3][\mathbf t_3]}  G_{[\mathbf s_4]\mathbf v}+ G_{\mathbf v[\mathbf t_1]}\wt G_{[\mathbf s_2][\mathbf t_2]}
\wt G_{[\mathbf s_3][\mathbf t_3]} G_{[\mathbf s_4]\mathbf v},\label{comp_geX_expG}
\end{align}
where we abbreviate $\wt G: = G - \Pi$. For the second term, we have
\begin{align}\label{term11}
b_N N^{-2}\sum_{i\in\mathcal I_1}\sum_{\mu\in\mathcal I_2}\left|G_{\mathbf v[\mathbf t_1]}\wt G_{[\mathbf s_2][\mathbf t_2]} \Pi_{[\mathbf s_3][\mathbf t_3]}G_{[\mathbf s_4]\mathbf v}\right|\prec b_N \Phi \cdot N^{(C_a+2)\delta}\left(\Psi^2 + \frac{q}{N\eta}\right)\prec N^{-\fa/2}\Phi
\end{align}
provided $\delta$ is small enough, where we used (\ref{eq_comp_r2}), (\ref{comp_geX_iteration}) and the definition \eqref{tildeS}. The third and fourth terms of (\ref{comp_geX_expG}) can be dealt with in a similar way. For the first term, we consider the following two cases. 

\vspace{5pt}
\noindent  {\bf Case 1:} $[\mathbf t_1]=\mathbf w_i$ and $[\mathbf s_4]=\bw_\mu$. Then we have
\begin{align*}
& \Big|\sum_{i\in\mathcal I_1}\sum_{\mu\in\mathcal I_2} G_{\mathbf v \mathbf w_i} \Pi_{[\mathbf s_2][\mathbf t_2]}\Pi_{[\mathbf s_3][\mathbf t_3]} G_{\mathbf w_\mu \mathbf v}\Big| \prec N^{1+2\delta}\left(\sum_{\mu\in\mathcal I_2}| G_{\mathbf w_\mu\mathbf v}|^2\right)^{1/2}\prec N^{3/2+(C_a/2+3)\delta}(q+\Psi),
\end{align*}
\nc where in the first step we used 
\be\label{add2nd}\Big|\sum_{i\in\mathcal I_1} G_{\mathbf v \mathbf w_i} \Pi_{[\mathbf s_2][\mathbf t_2]}\Pi_{[\mathbf s_3][\mathbf t_3]} \Big| \prec N^{1/2+2\delta}, \ee 
and in the second step we used (\ref{eq_comp_r2}). To get \eqref{add2nd}, we used the a priori bound (\ref{comp_geX_iteration}) with $\Phi\le N^{2\delta}$, which gives that for any deterministic unit vectors $\bv$ and $\bw$ (recall Definition \ref{stoch_domination} (ii)),
$$ |\langle  \bv, G\bw\rangle| \prec N^{2\delta}.$$ 
Applying this estimate with deterministic vectors $\bv$ and $\bw:=\sum_{i\in \mathcal I_1} \Pi_{[\mathbf s_2][\mathbf t_2]}\Pi_{[\mathbf s_3][\mathbf t_3]} \mathbf w_i $, we get 
$$ \Big|\sum_{i\in\mathcal I_1} G_{\mathbf v \mathbf w_i} \Pi_{[\mathbf s_2][\mathbf t_2]}\Pi_{[\mathbf s_3][\mathbf t_3]} \Big| \prec N^{2\delta}\|\mathbf v\| \|\mathbf w\|=\OO(N^{1/2+2\delta}),$$
using $\|\mathbf w\|=\OO(N^{1/2})$. This explains \eqref{add2nd}.
 If $[\mathbf t_1]=\mathbf w_\mu$ and $[\mathbf s_4]=\mathbf v_i$, the proof is similar.  \nc 

\vspace{5pt}
\noindent {\bf Case 2:} If $[\mathbf t_1]=[\mathbf s_4]$, then at least one of the terms $\Pi_{[\mathbf s_2][\mathbf t_2]}$ and $\Pi_{[\mathbf s_3][\mathbf t_3]}$ must be of the form $\Pi_{\mathbf w_i\mathbf w_\mu}=0$ or $\Pi_{\mathbf w_\mu\mathbf w_i}=0$, and \nc hence we have
$$\sum_i|\Pi_{[\mathbf s_2][\mathbf t_2]}\Pi_{[\mathbf s_3][\mathbf t_3]}|=0 \quad\text{ or }\quad \sum_\mu | \Pi_{[\mathbf s_2][\mathbf t_2]} \Pi_{[\mathbf s_3][\mathbf t_3]}|=0.$$
\nc In sum, we obtain that
$$b_NN^{-2}\Big|\sum_{i\in\mathcal I_1}\sum_{\mu\in\mathcal I_2} G_{\mathbf v[\mathbf t_1]} \Pi_{[\mathbf s_2][\mathbf t_2]}\Pi_{[\mathbf s_3][\mathbf t_3]} G_{[\mathbf s_4]\mathbf v}\Big|\prec N^{(C_a-1)\delta}(q+\Psi)$$
provided that $C_a\ge 8$. Together with \eqref{term11}, this proves  (\ref{comp_geX_words2}) for $l=1$.

When $l=2$, $\prod_{t=1}^2 A_{\mathbf v, i, \mu}(w_t)$ is of the form
\begin{align}
 & G_{\mathbf v\mathbf w_i} G_{\mathbf w_\mu \mathbf v} G_{\mathbf v \mathbf w_i} G_{\mathbf w_\mu\mathbf w_\mu} G_{\mathbf w_i\mathbf v}, \quad   G_{\mathbf v\mathbf w_i} G_{\mathbf w_\mu \mathbf v} G_{\mathbf v \mathbf w_\mu} G_{\mathbf w_i \mathbf w_i} G_{\mathbf w_\mu\mathbf v}, \label{eqn_q21}  \\
 &  G_{\mathbf v\mathbf w_i} G_{\mathbf w_\mu \mathbf v} G_{\mathbf v \mathbf w_i} G_{\mathbf w_\mu \mathbf w_i} G_{\mathbf w_\mu\mathbf v}, \quad   G_{\mathbf v\mathbf w_i} G_{\mathbf w_\mu \mathbf v} G_{\mathbf v \mathbf w_\mu} G_{\mathbf w_i \mathbf w_\mu} G_{\mathbf w_i\mathbf v},\label{eqn_q22}
\end{align}
or an expression obtained from one of these four by exchanging $\mathbf w_i$ and $\mathbf w_\mu$. The first expression in (\ref{eqn_q21}) can be estimated using  (\ref{eq_comp_r2}) and (\ref{comp_geX_iteration}):
\begin{equation}\label{q=2_2}
\Big|\sum_i G_{\mathbf v\mathbf w_i} G_{\mathbf v \mathbf w_i} G_{\mathbf w_i\mathbf v}\Big|\prec N^{1+(C_a+4)\delta} \left(\Psi^2 + \frac{q}{N\eta}\right),
\end{equation}
\nc and
\begin{equation}
\label{q=2_1}
\begin{split}
\sum_\mu G_{\mathbf w_\mu \mathbf v} G_{\mathbf w_\mu\mathbf w_\mu}&=\sum_\mu G_{\mathbf w_\mu \mathbf v}\wt G_{\mathbf w_\mu\mathbf w_\mu}+\sum_\mu  G_{\mathbf w_\mu \mathbf v}\Pi_{\mathbf w_\mu\mathbf w_\mu} \\
&=\OO_\prec\left[N^{1+(C_a/2+1)\delta}\Phi\left(\Psi^2 + \frac{q}{N\eta}\right)^{1/2}+ N^{1/2+2\delta}\right],
\end{split}
\end{equation}
where in the second step we applied the same argument to $\sum_\mu  G_{\mathbf w_\mu \mathbf v}\Pi_{\mathbf w_\mu\mathbf w_\mu}$ as the one for \eqref{add2nd}. \nc Combining \eqref{tildeS}, (\ref{q=2_2}) and (\ref{q=2_1}), we get that 
\[b_N N^{-2}\Big|\sum_i\sum_\mu G_{\mathbf v\mathbf w_i} G_{\mathbf w_\mu \mathbf v} G_{\mathbf v \mathbf w_i} G_{\mathbf w_\mu\mathbf w_\mu} G_{\mathbf w_i\mathbf v}\Big| \prec \left(N^{(C_a-1)\delta}(q+\Psi) + N^{-\fa/2}\Phi\right)^2,\]
provided that $\delta$ is small enough. The second expression in (\ref{eqn_q21}) can be estimated similarly. The first expression of (\ref{eqn_q22}) can be estimated using \eqref{tildeS}, (\ref{eq_comp_r2}) and (\ref{comp_geX_iteration}) as
\begin{equation*}
\begin{split}
b_N N^{-2}\left|\sum_i\sum_\mu  G_{\mathbf v\mathbf w_i} G_{\mathbf w_\mu \mathbf v} G_{\mathbf v \mathbf w_i} G_{\mathbf w_\mu \mathbf w_i} G_{\mathbf w_\mu\mathbf v}\right|& \prec b_N N^{-2+2\delta}\sum_i\sum_\mu\left| G_{\mathbf v\mathbf w_i}\right|^2\left| G_{\mathbf w_\mu\mathbf v}\right|^2 \\
& \prec b_N N^{(2C_a+6)\delta}\left( \Psi^2 +\frac{q}{N\eta}\right)^2 \le (q +\Psi)^2
\end{split}
\end{equation*}
for small enough $\delta$. The second expression in (\ref{eqn_q22}) is estimated similarly.  This proves (\ref{comp_geX_words2}) for $l=2$.

When $l = 3$, $\prod_{t=1}^3 A_{\mathbf v, i, \mu}(w_t)$ is of the form 
$( G_{\mathbf v\mathbf w_i} G_{\mathbf w_\mu \mathbf v})^3$ or an expression obtained by exchanging $\mathbf w_i$ and $\mathbf w_\mu$ in some of the three factors. We use (\ref{eq_comp_r2}) and $\sum_i|\Pi_{\mathbf v\mathbf w_i}|^2 = \OO(1)$ to get that
\[\left|\sum_i( G_{\mathbf v\mathbf w_i})^3\right|\prec \sum_i|\wt G_{\mathbf v\mathbf w_i}|^3+\sum_i|\Pi_{\mathbf v\mathbf w_i}|^3\prec \Phi\sum_i \left(| G_{\mathbf v\mathbf w_i}|^2+|\Pi_{\mathbf v\mathbf w_i}|^2 \right)+1\prec N^{1+(C_a+2)\delta}\left(\Psi^2 +\frac{q}{N\eta}\right)\Phi+\Phi+1.\]
Now we conclude (\ref{comp_geX_words2}) for $l=3$ using \eqref{tildeS} and $N^{-1/2}=\OO( q+\Psi)$.
\end{proof}

If $A$ or $B$ is diagonal, then we can still prove \eqref{aniso_law} for all $z\in S(c_0,C_0,\e)$ without using \eqref{3moment}. This follows from an improved self-consistent comparison argument for sample covariance matrices (i.e. separable covariance matrices with $B=I$) in \cite[Section 8]{Anisotropic}. The argument for separable covariance matrices with diagonal $A$ or $B$ is almost the same except for some notational differences, so we omit the details.

\subsection{Weak averaged local law}\label{section_averageTX}

In this section, we prove the weak averaged local laws in \eqref{aver_in1} and \eqref{aver_out1}. 
The proof is similar to the one for \eqref{aniso_law} in previous subsections, and we only explain the differences. Note that the bootstrapping argument is not necessary, since we already have a good a priori bound by \eqref{aniso_law}.
In analogy to (\ref{eq_comp_F(X)}), we define
\begin{align*}
\wt F(X,z) : &= |m(z)-m_{c}(z)| =\left|\frac{1}{nz}\sum_{i\in\sI_1} \left(G_{ii}(X,z)- \Pi_{ii}(z)\right)\right|,
\end{align*}
where we used \eqref{mcPi}. 
Moreover, by Proposition \ref{prop_diagonal}, we know that \eqref{aver_in1} and \eqref{aver_out1} hold for Gaussian $X$ (without the $q^2$ term). For now, we assume \eqref{3moment} and prove the following stronger estimates:
\begin{equation}
 \vert m(z)-m_{c}(z) \vert \prec (N \eta)^{-1} \label{aver_ins} 
\end{equation}
for $z\in S(c_0,C_0,\epsilon)$, and 
\begin{equation}\label{aver_outs}
 | m(z)-m_{c}(z)|\prec \frac{q}{N\eta}  + \frac{1}{N(\kappa +\eta)} + \frac{1}{(N\eta)^2\sqrt{\kappa +\eta}},
\end{equation}
for $z\in S(c_0,C_0,\epsilon)\cap \{z=E+\ii\eta: E\ge \lambda_+, N\eta\sqrt{\kappa + \eta} \ge N^\epsilon\}$. At the end of this section, we will show how to relax \eqref{3moment} to \eqref{assm_3moment} for $z\in \wt S(c_0,C_0,\fa,\e)$.

Note that
\be\label{psi2}
\Psi^2(z) \lesssim \frac{1}{N\eta}, \quad \text{and} \quad \Psi^2(z) \lesssim \frac{1}{N(\kappa +\eta)} + \frac{1}{(N\eta)^2\sqrt{\kappa +\eta}} \ \text{ outside of the spectrum}.
\ee
Then following the argument in Section \ref{subsec_interp}, analogous to (\ref{eq_comp_selfest}), we only need to prove that
\begin{equation}\label{eq_comp_selfestAvg}
N^{-2}q^{r-4}\sum_{i\in\mathcal I_1}\sum_{\mu\in\mathcal I_2}\left|\bbE \left(\frac{\partial}{\partial X_{i\mu}}\right)^r\wt F^p(X)\right|=\OO\left(\left[N^{\delta}\left(\Psi^2+\frac{q}{N\eta}\right)\right]^p+\mathbb E\wt F^p(X)\right)
\end{equation}
for all $r=4,...,4p+4$, where $\delta>0$ is any positive constant. Analogous to (\ref{eq_comp_goal2}), it suffices to prove that for $r=4,...,4p+4$,
\begin{equation}\label{eq_comp_goalAvg}
N^{-2}q^{r-4}\sum_{i\in\mathcal I_1}\sum_{\mu\in\mathcal I_2}\left|\bbE\prod_{t=1}^{p}\left(\frac{1}{n}\sum_{j\in\sI_1}A_{ \mathbf e_j, i, \mu}(w_t)\right)\right|=\OO\left(\left[N^{\delta}\left(\Psi^2+\frac{q}{N\eta}\right)\right]^p+\mathbb E\wt F^p(X)\right)
\end{equation}
for $\sum_t {\bm l}(w_t)=r$. 
Similar to (\ref{eq_comp_Rs}) we define
\begin{equation}\nonumber
\mathcal R_{j, a}:=| G_{j \mathbf w_a}|+| G_{\mathbf w_a j}|.
\end{equation}
Using \eqref{aniso_law} and Lemma \ref{lem_comp_gbound}, similarly to \eqref{eq_comp_r2}, we get that
\begin{equation}\label{eq_comp_r22}
\begin{split}
 \frac{1}{n}\sum_{j\in\sI_1}\mathcal R_{j,a}^2 & \prec \frac{ \im \left(z^{-1}G_{\mathbf w_i\mathbf w_i}\right) + \im G_{\mathbf w_\mu\mathbf w_\mu} + \eta\left(\left| G_{\mathbf w_i\mathbf w_i} \right|+ \left| G_{\mathbf w_\mu \mathbf w_\mu} \right|\right)}{N\eta} \prec \Psi^2+\frac{q}{N\eta}.
 \end{split}
\end{equation}
Since $G=\OO_\prec(1)$ by \eqref{aniso_law}, we have 
\begin{equation}\label{average_bound}
\left|\frac{1}{n}\sum_{j\in\sI_1}A_{ \mathbf e_j, i, \mu}(w)\right|\prec \frac{1}{n}\sum_{j\in\sI_1}\left(\mathcal R_{j,i}^2+\mathcal R_{j,\mu}^2\right)\prec \Psi^2 +\frac{q}{N\eta} \quad \text{ for any $w$ such that }{\bm l}(w)\ge 1.
\end{equation}
With (\ref{average_bound}), for any $r\ge 4$, the left-hand side of (\ref{eq_comp_goalAvg}) is bounded by
\[\bbE\wt F^{p-l}(X)\left(\Psi^2+\frac{q}{N\eta}\right)^{l}.\]
Applying H\"older's inequality, we get \eqref{eq_comp_selfestAvg}, which completes the proof of \eqref{aver_ins} and \eqref{aver_outs} under \eqref{3moment}. 

\vspace{5pt}

\nc Then we prove the averaged local law \eqref{aver_in1} for $z\in \wt S(c_0,C_0,\fa,\e)$ and \eqref{aver_out1} for $z\in \wt S(c_0,C_0,\fa, \epsilon)\cap \{z=E+\ii\eta: E\ge \lambda_+, N\eta\sqrt{\kappa + \eta} \ge N^\epsilon\}$ under \eqref{assm_3moment}. \nc By \eqref{psi2}, it suffices to prove 
\begin{equation}\label{comp_avg_geX_self-improving-bound}
b_N N^{-2}\left|\sum_{i\in\mathcal I_1}\sum_{\mu\in\mathcal I_2}\bbE \left(\frac{\partial}{\partial X_{i\mu}}\right)^3\wt F^p(X)\right|=\OO\left(\left[N^\delta (q^2 +\Psi^2)\right]^p + \left( \frac{N^{-\fa/2 + \delta}}{N\eta}\right)^p+\mathbb E\wt F^p(X)\right),
\end{equation}
for any small constant $\delta>0$. Analogous to the arguments in Section \ref{subsec_3moment}, it reduces to showing that
\begin{equation}\label{eq_comp_goalAvg_genX}
b_N N^{-2}\left|\sum_{i\in\mathcal I_1}\sum_{\mu\in\mathcal I_2} \prod_{t=1}^{l}\left(\frac{1}{n}\sum_{j\in\sI_1}A_{ \mathbf e_j, i, \mu}(w_t)\right)\right|=\OO_\prec\left(\left(q^2+\Psi^2\right)^{l} + \left( \frac{N^{-\fa/2}}{N\eta}\right)^l\right),
\end{equation}
where $l\in \{1,2,3\}$ is the number of words with nonzero length. Then we can discuss these three cases using a similar argument as in Section \ref{subsec_3moment}, with the only difference being that we now can use the anisotropic local law \eqref{aniso_law} instead of the a priori bounds \eqref{comp_eq_apriori}  and (\ref{comp_geX_iteration}). 

%

In the $l=1$ case, we first consider the expression $A_{ \mathbf e_j, i, \mu}(w_1) =  G_{j\mathbf w_i} G_{\mathbf w_\mu \mathbf w_\mu} G_{\mathbf w_i\mathbf w_i} G_{\mathbf w_\mu j}$. We have 
\begin{equation}\nonumber
\left|\sum_i G_{j\mathbf w_i} G_{\mathbf w_i\mathbf w_i}\right| \prec \left|\sum_i G_{j\mathbf w_i} \Pi_{\mathbf w_i\mathbf w_i}\right| + \sum_i (q+\Psi)\left| G_{j\mathbf w_i} \right|\prec \sqrt N + N (q+\Psi) \left(\Psi^2 +\frac{q}{N\eta}\right)^{1/2},
\end{equation}
where we used \eqref{aniso_law} and \eqref{eq_comp_r2}.
Similarly, we also have
\begin{equation}\nonumber
 \left|\sum_\mu G_{\mathbf w_\mu\mathbf w_\mu} G_{\mathbf w_\mu j}\right|\prec \left|\sum_\mu {\Pi}_{\mathbf w_\mu \mathbf w_\mu}  G_{\mathbf w_\mu j}\right|  + \sum_\mu (q+\Psi)\left|G_{\mathbf w_\mu j}\right|  \prec \sqrt N(q+\Psi)+N (q+\Psi) \left(\Psi^2 +\frac{q}{N\eta}\right)^{1/2},
\end{equation} 
where we also used $\Pi_{\mathbf w_\mu j}=0$ for any $\mu$ in the second step. Then with \eqref{tildeS}, we can see that the LHS of (\ref{eq_comp_goalAvg_genX}) is bounded by $\OO_\prec(q^2 + \Psi^2)$ in this case.
For the case $A_{ \mathbf e_j, i, \mu}(w_1) =  G_{j\mathbf w_i} G_{\mathbf w_\mu\mathbf w_\mu} G_{\mathbf w_i \mathbf w_\mu} G_{\mathbf w_i j}$, we can estimate that
$$\left|\sum_\mu  G_{\mathbf w_\mu\mathbf w_\mu} G_{\mathbf w_i \mathbf w_\mu} \right| \prec \left|\sum_\mu  \Pi_{\mathbf w_\mu\mathbf w_\mu} G_{\mathbf w_i \mathbf w_\mu} \right| + \sum_\mu (q+\Psi)\left|G_{\mathbf w_i \mathbf w_\mu} \right| \prec \sqrt N + N (q+\Psi) \left(\Psi^2 +\frac{q}{N\eta}\right)^{1/2},$$
and
$$ \sum_i \left|G_{j\mathbf w_i} G_{\mathbf w_i j}\right|\prec N\left(\Psi^2 +\frac{q}{N\eta}\right).$$
Thus in this case the LHS of (\ref{eq_comp_goalAvg_genX}) is also bounded by $\OO_\prec(q^2 + \Psi^2)$. The case $A_{ \mathbf e_j, i, \mu}(w_1) =  G_{j\mathbf w_i} G_{\mathbf w_\mu\mathbf w_i} G_{\mathbf w_\mu \mathbf w_\mu} G_{\mathbf w_i j}$ can be handled similarly. Finally in the case $A_{ \mathbf e_j, i, \mu}(w_1) =  G_{j\mathbf w_i} G_{\mathbf w_\mu\mathbf w_i} G_{\mathbf w_\mu \mathbf w_i} G_{\mathbf w_\mu j}$, we can estimate that
$$ \left|\sum_{i,\mu}  G_{j\mathbf w_i} G_{\mathbf w_\mu\mathbf w_i} G_{\mathbf w_\mu \mathbf w_i} G_{\mathbf w_\mu j} \right| \prec  \sum_{i,\mu} \left(\left| G_{j\mathbf w_i}\right|^2 +\left| G_{\mathbf w_\mu j} \right|^2 \right) | G_{\mathbf w_\mu \mathbf w_i}|^2 \prec N^2 \left(\Psi^2 +\frac{q}{N\eta}\right)^2.$$
Again in this case the LHS of (\ref{eq_comp_goalAvg_genX}) is bounded by $\OO_\prec(q^2 + \Psi^2)$. All the other expressions are obtained from these four by exchanging $\mathbf w_i$ and $\mathbf w_\mu$.

In the $l=2$ case, $\prod_{t=1}^{2}\left(\frac{1}{n}\sum_{j\in\sI_1}A_{ \mathbf e_j, i, \mu}(w_t)\right)$ is of the form (up to some constant coefficients)
\[\frac{1}{N^2}\sum_{j_1,j_2} G_{j_1\mathbf w_i} G_{\mathbf w_\mu j_1} G_{j_2 \mathbf w_i} G_{\mathbf w_\mu\mathbf w_\mu} G_{\mathbf w_i j_2}\quad \text{ or }\quad \frac{1}{N^2}\sum_{j_1,j_2} G_{j_1\mathbf w_i} G_{\mathbf w_\mu j_1} G_{j_2\mathbf w_i} G_{\mathbf w_\mu\mathbf w_i} G_{\mathbf w_\mu j_2},\]
or an expression obtained from one of these terms by exchanging $\mathbf w_i$ and $\mathbf w_\mu$. These two expressions can be written as 
\be\label{2terms}
N^{-2}( G^{\times 2} )_{\mathbf w_\mu\mathbf w_i}(G^{\times 2})_{\mathbf w_i\mathbf w_i} G_{\mathbf w_\mu\mathbf w_\mu}, \quad N^{-2}( G^{\times 2})^2_{\mathbf w_\mu\mathbf w_i} G_{\mathbf w_\mu\mathbf w_i}, \quad G^{\times 2}:= G \begin{pmatrix}I_{\mathcal I_1 \times \mathcal I_1} & 0\\ 0 & 0\end{pmatrix} G.
\ee
For the second term, using \eqref{green2}, \eqref{spectral1} and recalling that $Y=\Sig^{1/2} U^{*}X V\wt \Sig^{1/2}$, we can get that
\begin{align}
& \left|\frac{1}{N^2}\sum_{i,\mu} ( G^{\times 2})^2_{\mathbf w_\mu\mathbf w_i} G_{\mathbf w_\mu\mathbf w_i}\right| \le \frac{1}{N^2}\sum_{i,\mu} \left|( G^{\times 2})_{\mathbf w_\mu\mathbf w_i} \right|^2 \lesssim \frac{1}{N^2}\text{Tr}\left[(\mathcal G_1^{*})^2 YY^* (\mathcal G_1)^2\right] \nonumber\\
& =  \frac{1}{N^2}\text{Tr}\left[\mathcal G_1^{*} (\mathcal G_1)^2\right]  +  \frac{\bar z }{N^2}\text{Tr}\left[(\mathcal G_1^{*})^2 (\mathcal G_1)^2\right]  \lesssim \frac{1}{N^2}\sum_k \frac{1}{\left[(\lambda_k-E)^2 +\eta^2\right]^{3/2}}+ \frac{1}{N^2}\sum_k \frac{1}{\left[(\lambda_k-E)^2 +\eta^2\right]^{2}} \nonumber\\
& \lesssim \frac{1}{N\eta^3}\left(\frac1n \sum_k \frac{\eta}{(\lambda_k-E)^2 +\eta^2} \right) =\frac{\im m}{N\eta^3}\prec  \frac{\im m_c + q+\Psi}{N\eta^3} \lesssim \eta^{-2}\left(\Psi^2 +\frac{q}{N\eta}\right).\label{3term}
\end{align}
Using \eqref{aniso_law} and \eqref{eq_comp_r2}, it is easy to show that
\be \label{3.5term}
\left|\sum_{\mu}( G^{\times 2} )_{\mathbf w_\mu\mathbf w_i} \Pi_{\mathbf w_\mu\mathbf w_\mu}\right| \prec N^{3/2}\left( \Psi^2 + \frac{q}{N\eta}\right),\quad \text{ and } \quad \left|(G^{\times 2})_{\mathbf x\mathbf y} \right| \prec N\left( \Psi^2 + \frac{q}{N\eta}\right), \ee
for any deterministic unit vectors $\mathbf x$, $\mathbf y$. \nc (To get the first estimate in \eqref{3.5term}, we write  
$$\sum_{\mu}( G^{\times 2} )_{\mathbf w_\mu\mathbf w_i} \Pi_{\mathbf w_\mu\mathbf w_\mu} = \|\mathbf w\| \sum_j G_{\bw j}G_{j\bw_i}, \quad \bw:= \sum_\mu \overline \Pi_{\bw_\mu\bw_\mu} \bw_\mu,$$
and then use \eqref{eq_comp_r2}.) \nc
Thus for the first term in \eqref{2terms}, we have
\begin{align}
&\quad \, \left|\frac{1}{N^2}\sum_{i,\mu}( G^{\times 2} )_{\mathbf w_\mu\mathbf w_i}(G^{\times 2})_{\mathbf w_i\mathbf w_i} G_{\mathbf w_\mu\mathbf w_\mu}\right| \nonumber\\
& \le \left|\frac{1}{N^2}\sum_{i,\mu}( G^{\times 2} )_{\mathbf w_\mu\mathbf w_i}(G^{\times 2})_{\mathbf w_i\mathbf w_i} \wt G_{\mathbf w_\mu\mathbf w_\mu}\right| + \left|\frac{1}{N^2}\sum_{i,\mu}( G^{\times 2} )_{\mathbf w_\mu\mathbf w_i}(G^{\times 2})_{\mathbf w_i\mathbf w_i} \Pi_{\mathbf w_\mu\mathbf w_\mu}\right| \nonumber\\
& \prec N(q+\Psi)\left( \Psi^2 + \frac{q}{N\eta}\right)\left(\frac{1}{N^2}\sum_{i,\mu}\left|( G^{\times 2} )_{\mathbf w_\mu\mathbf w_i}\right|^2\right)^{1/2} +N^{3/2}\left( \Psi^2 + \frac{q}{N\eta}\right)^2 \nonumber\\
& \prec N\eta^{-1}(q+\Psi)\left( \Psi^2 + \frac{q}{N\eta}\right)^{3/2} +N^{3/2}\left( \Psi^2 + \frac{q}{N\eta}\right)^2,\label{4term}
\end{align}
where in the last step we used the bound in \eqref{3term}. Now using \eqref{3term}, \eqref{4term} and \eqref{tildeS}, we get
$$b_N N^{-2}\left|\sum_{i\in\mathcal I_1}\sum_{\mu\in\mathcal I_2} \prod_{t=1}^{2}\left(\frac{1}{n}\sum_{j\in\sI_1}A_{ \mathbf e_j, i, \mu}(w_t)\right)\right| \prec \left(q^2+\Psi^2\right)^{2} + \left( \frac{N^{-\fa/2}}{N\eta}\right)^2 .$$

Finally, in the $l=3$ case, $\prod_{t=1}^{3}\left(\frac{1}{N}\sum_{j\in\sI_1}A_{ \mathbf e_j, i, \mu}(w_t)\right)$ is of the form 
${N^{-3}}( G^{\times 2})^3_{\mathbf w_i\mathbf w_\mu}$, or an expression obtained by exchanging $\mathbf w_i$ and $\mathbf w_\mu$ in some of the three factors. Using \eqref{3.5term} and the bound in \eqref{3term}, we can estimate that 
$$\frac{1}{N^3}\left|\sum_{i,\mu}( G^{\times 2})^3_{\mathbf w_i\mathbf w_\mu}\right| \prec \left( \Psi^2 + \frac{q}{N\eta}\right)\frac{1}{N^2}\sum_{i,\mu}\left|( G^{\times 2} )_{\mathbf w_\mu\mathbf w_i}\right|^2 \prec \eta^{-2}\left(\Psi^2 +\frac{q}{N\eta}\right)^2,$$
Then the LHS of (\ref{eq_comp_goalAvg_genX}) is bounded by 
$$O_\prec\left(\left(q^2 + \Psi^2\right) \left(\frac{N^{-\fa/2}}{N\eta}\right)^2\right).$$

Combining the above three cases $l=1,2,3$, we conclude \eqref{comp_avg_geX_self-improving-bound}, which finishes the proof of \eqref{aver_in1} and \eqref{aver_out1}. 

If $A$ or $B$ is diagonal, then by the remark at the end of Section \ref{subsec_3moment}, the anisotropic local law \eqref{aniso_law} holds for all $z\in S(c_0,C_0,\e)$ even in the case with $b_N=N^{1/2}$ in \eqref{assm_3moment}. Then with \eqref{aniso_law} and the self-consistent comparison argument in \cite[Section 9]{Anisotropic}, we can prove \eqref{aver_in1} and \eqref{aver_out1} for $z\in S(c_0,C_0,\e)$. Again most of the arguments are the same as the ones in \cite[Section 9]{Anisotropic}, hence we omit the details.

\section{Proof of Lemma \ref{thm_largebound}, Theorem \ref{thm_largerigidity} and Theorem \ref{lem_comparison}}\label{sec_Lindeberg}

With Lemma \ref{lem_decrease}, given $X$ satisfying the assumptions in Theorem \ref{LEM_SMALL}, we can construct a matrix $\wt{X}$ with support $q= N^{-1/2}$ and have the same first four moments as $X$. By Theorem \ref{LEM_SMALL}, the averaged local laws \eqref{aver_in} and \eqref{aver_out0} hold for $G(\wt X,z)$. Thus it is easy to see that Theorem \ref{thm_largerigidity} is implied by the following lemma. 

\begin{lemma} \label{lem_comgreenfunction}
Let $X$, $\wt{X}$ be two matrices as in Lemma \ref{lem_decrease}, and $G\equiv G(X,z)$, $\wt G\equiv G(\wt X,z)$ be the corresponding resolvents. We denote $m(z)\equiv m(X,z)$ and $\wt m(z)\equiv m(\wt X,z)$. Fix any constant $\e>0$. For any $z \in \wt S(c_0, C_0, \fa, \epsilon)$, if there exist deterministic quantities $J\equiv J(N)$ and $K\equiv K(N)$ such that
\begin{equation}
 \wt{G}(z)-\Pi =\OO_\prec (J), \quad \vert \wt m(z)-m_{c}(z) \vert \prec K, \quad J+K\prec 1,\label{KEYBOUNDS}
\end{equation}
then for any fixed $p \in 2\mathbb{N}$, we have
\begin{equation}
\mathbb{E} \vert m(z)-m_{c}(z) \vert^p \prec \mathbb{E} \vert \wt m(z) - m_{c}(z) \vert ^p+ \left(\Psi^2(z) + J^2+K\right)^p. \label{KEYEYEYEY}
\end{equation}
\end{lemma}
\begin{proof}[Proof of Theorem \ref{thm_largerigidity}]
By Theorem \ref{LEM_SMALL}, one can choose $J= \Psi(z)$ and 
\begin{equation*}
K=\frac{1}{N \eta}, \quad \text{ or }  \quad  \nc \frac{N^{-\fa/2}}{N\eta} \nc + \frac{1}{N(\kappa +\eta)} + \frac{1}{(N\eta)^2\sqrt{\kappa +\eta}} \ \text{ outside of the spectrum}.
\end{equation*}
Then using \eqref{KEYEYEYEY}, \eqref{psi2} and Markov's inequality, we can prove \eqref{aver_in} and \eqref{aver_out0}.

The eigenvalues rigidity results \eqref{Kdist} and \eqref{rigidity} follow from \eqref{aver_in} and \eqref{aver_out0} through a standard argument, see e.g.\;the proofs for \cite[Theorems 2.12-2.13]{EKYY1}, \cite[Theorem 2.2]{EYY} or \cite[Theorem 3.3]{PY}. \nc More precisely, the estimate \eqref{Kdist} is implied by the local law \eqref{aver_in}. Then the rigidity result \eqref{rigidity} follows from \eqref{Kdist} together with the following upper bound on the largest eigenvalue: for any constant $\e>0$,
\be\label{upper1}
\lambda_1 \le \lambda_+ + N^{-2/3+\e} \quad \text{with high probability.}
\ee
 In \cite{EYY}, this follows from the averaged local law \eqref{aver_out0} without the ${N^{-\fa/2}}/({N\eta})$ term. Now we would like to show that even with this extra term, the bound \eqref{aver_out0} is sufficient to give \eqref{upper1}. First, we have $\lambda_1 \le C$ with high probability for some constant $C>0$ by e.g. \cite[Lemma 3.12]{DY}. Now we pick $\e$ to be a sufficiently small constant such that $0<\e\le c/4$, and $C_0$ to be a sufficiently large constant such that $C_0 \lambda_+ > C$. Set $\eta=N^{-2/3}$ and choose $E=\lambda_+ + \kappa$ outside of the spectrum with some $\kappa\ge N^{-2/3+2\e} \gg N^\e \eta$. Then using \eqref{Immc}, \eqref{eq_defpsi} and $b\le N^{1/3-c}$, we can verify that
$$z=E+\ii\eta \in \wt S(c_0,C_0,c,\e)\cap \{z=E+\ii\eta: E\ge \lambda_+, N\eta\sqrt{\kappa + \eta} \ge N^\epsilon\}.$$
Then using \eqref{aver_out0}, we get that
\be\label{add1}|\im m(z) - \im m_c(z)| \prec \frac{N^{-c/2} + N^{-\e}}{N\eta} = \OO\left(\frac{N^{-\e}}{N\eta}\right).\ee
On the other hand, if there is an eigenvalue $\lambda_j$ satisfying $|\lambda_j - E| \le \eta$ for some $1\le j \le n$, then 
\be\label{add2}\im m(z) = \frac1n\sum_{i=1}^n \frac{\eta}{|\lambda_i-E|^2 + \eta^2} \gtrsim \frac{1}{N\eta}.\ee
On the other hand, by \eqref{Immc} we have 
$$\im m_c(z) =\OO\left( \frac{\eta}{\sqrt{\kappa + \eta}}\right) = \OO\left( \frac{N^{-\e}}{N\eta}\right).$$
Together with \eqref{add2}, this contradicts \eqref{add1}. Hence we obtain that $\lambda_1 \le \lambda_+ + N^{-2/3+2\e}$ with high probability. Since $\e$ can be arbitrarily chosen, we conclude \eqref{upper1}. With \eqref{aver_in} and \eqref{upper1}, the rest of the proof for \eqref{Kdist} and \eqref{rigidity} is the same as \cite{EYY}, so we omit the details.
\nc
\end{proof}

In order to prove Lemma \ref{thm_largebound} and Lemma \ref{lem_comgreenfunction}, we will extend the resolvent comparison method developed in \cite{DY,LY}. The basic idea is still to use the Lindeberg replacement strategy for $G(X,z)$. On the other hand, the main difference is that the resolvent estimates are only obtained from the entrywise local law in \cite{DY,LY}, while in our case we need to use the more general anisotropic local law \eqref{aniso_law}. (We will use the anisotropic local law in \eqref{KEYBOUNDS} when proving Lemma \ref{lem_comgreenfunction}. However, for simplicity of presentation, we will always mention \eqref{aniso_law} instead.)

\nc We remark that the following proof is similar to the one in \cite[Section 6]{DY}, and involves some tedious notations bookkeeping. We shall first give the proper notations and definitions that are adapt to our setting. The proof of the results is then a straightforward extension of the one in \cite{DY} by using the correct notations and applying the stronger anisotropic local law \eqref{aniso_law}. Hence we will only state several key lemmas that are needed for the argument without presenting all the details of the proof. \nc
 
Let $X=(x_{i\mu})$ and $\wt{X}=(\wt x_{i\mu})$ be two matrices as in Lemma \ref{lem_decrease}.
Define a bijective ordering map $\Phi$ on the index set of $X$ as
\begin{equation*}
\Phi: \{(i,\mu):1 \leq i \leq n, \ n+1 \leq \mu \leq n+N \} \rightarrow \{1,\ldots,\gamma_{\max}=nN\}.
\end{equation*} 
For any $1\le \gamma \le \gamma_{\max}$, we define the matrix $X^{\gamma}= (x^{\gamma}_{i\mu})$ such that $x_{i\mu}^{\gamma} =x_{i\mu} $ if $\Phi(i,\mu)\leq \gamma$, and $x_{i\mu}^{\gamma} =\wt{x}_{i\mu}$ otherwise. Then we have $X^0=\wt X$, $X^{\gamma_{\max}}=X$, and $X^\gamma$ has bounded support $q=N^{-\phi}$ for all $0\le \gamma \le \gamma_{\max}$. Correspondingly, we define
 \begin{equation}\label{Hgamma}
   H^{\gamma} := \left( {\begin{array}{*{20}c}
   { 0 } &  Y^\gamma   \\
   {(Y^\gamma)^*} & {0}  \\
   \end{array}} \right), \ \ \ G^\gamma:= \left( {\begin{array}{*{20}c}
   { - I_{n\times n}} & Y^\gamma  \\
   {(Y^\gamma)^*} & { - zI_{N\times N}}  \\
\end{array}} \right)^{-1},
 \end{equation}
where $Y^\gamma:=\Sig^{1/2} U^{*}X^\gamma V\wt \Sig^{1/2}$.
Then we define the $(n+N)\times (n+N)$ matrices $V^\gamma$ and $W^\gamma$ by (recall \eqref{deltaimu})
$$V^\gamma=\Delta_{(i\mu)}^{x_{i\mu}},\quad W^\gamma:=\Delta_{(i\mu)}^{\wt x_{i\mu}} ,$$
so that $H^{\gamma}$ and $H^{\gamma-1}$ can be written as
\begin{equation}\label{R0}
H^\gamma= Q^\gamma + V^\gamma, \quad H^{\gamma-1} = Q^\gamma+W^\gamma,
\end{equation}
for some matrix $Q^\gamma$ that is independent of $x_{i\mu}$ and $\wt x_{i\mu}$. For simplicity of notations, for any $\gamma$ we denote 
\begin{equation}
S^\gamma:=G^\gamma, \quad  T^\gamma:=G^{\gamma-1}, \quad R^\gamma:=\left(Q^\gamma-\left( {\begin{array}{*{20}c}
   {  I_{n\times n}} & 0  \\
   {0} & { z I_{N\times N}}  \\
\end{array}} \right)\right)^{-1}.  \label{R}
\end{equation}
For convenience, we sometimes drop the superscript from $R,S,T$ if $\gamma$ is fixed. Under the above definitions, we can write
\begin{align}
S&= \left(Q^\gamma -\left( {\begin{array}{*{20}c}
   {  I_{M\times M}} & 0  \\
   {0} & { z I_{N\times N}}  \\
\end{array}} \right) + V^\gamma\right)^{-1}=(I+RV^\gamma)^{-1}R  \label{RESOLVENT}\\
&=R-RV^\gamma R+(RV^\gamma)^2R+\cdots+(-1)^l(RV^\gamma)^l R+(-1)^{l+1}(RV^\gamma)^{l+1}S \label{RESOLVENTEXPANSION}
\end{align}
for $l\in \N$.
On the other hand, we can also expand $R$ in terms of $S$,
\begin{equation}
R=(I-SV^\gamma)^{-1}S=S+SV^\gamma S+(SV^\gamma)^2 S+\ldots+ (SV^\gamma)^l S+ (SV^\gamma)^{l+1}R. \label{RESOLVENTEXPANSION2}
\end{equation}
We have similar expansions for $T$ and $R$ by replacing $(V^\gamma,S)$ with $(W^\gamma,T)$ in (\ref{RESOLVENTEXPANSION}) and (\ref{RESOLVENTEXPANSION2}). By the bounded support condition, we have
\begin{equation}
\max_{\gamma} \| V^\gamma\| = \OO(|x_{i\mu}|) \prec N^{-\phi}, \quad \max_{\gamma} \| W^\gamma\| = \OO(|\wt x_{i\mu}|) \prec N^{-1/2}. \label{5BOUND2}
\end{equation}
Note that $S$, $R$ and $T$ satisfy the following deterministic bounds by (\ref{eq_gbound}):
\begin{equation}
\sup_{z\in S(c_0,C_0,\e)} \max_{\gamma} \max \left\{ \| S^\gamma \|, \|T^\gamma \|, \|R^\gamma \| \right\}\lesssim \sup_{z\in  S(c_0,C_0,\e)} \eta^{-1} \le N. \label{5BOUNDT}
\end{equation}
Then using expansion (\ref{RESOLVENTEXPANSION2}) in terms of $T,W^\gamma$ with $l=3$, the anisotropic local law (\ref{aniso_law}) for $T$, and the bound (\ref{5BOUNDT}) for $R$, we can get that for any deterministic unit vectors $\mathbf u,\mathbf v \in \mathbb C^{\mathcal I}$, 
\begin{equation}
\sup_{z\in \wt S(c_0,C_0,\fa,\e)} \max_{\gamma} |R^\gamma_{\mathbf u \mathbf v}| =\OO(1) \quad \text{ with high probability.} \label{5BOUND1}
\end{equation}

From the definitions of $V^\gamma$ and $W^\gamma$, one can see that it is helpful to introduce the following notations to simplify the expressions.


\begin{definition}[Matrix operators $*_\gamma$] \label{def_operator1}
For any two $(n+N) \times (n+N)$ matrices $A$ and $B$, we define 
\begin{equation}
A *_{\gamma} B :=A I_{\gamma} B, \quad I_{\gamma}:  = \Delta_{(i\mu)}^{1},\quad \Phi(i,\mu)=\gamma.
\end{equation} 
In other words, we have
$$A*_{\gamma}B =A \mathbf w_i \mathbf w_\mu^* B +A\mathbf w_\mu \mathbf w_i^* B,\quad \mathbf w_i:= \Sigma^{1/2} \bu_i, \quad \mathbf w_\mu:=\wt \Sigma^{1/2} \bv_\mu.$$ 
We denote the $l$-th power of $A$ under the $*_\gamma$-product by $A^{*_\gamma l}$, i.e.
\begin{equation}
A^{*_\gamma m}:=\underbrace{A*_\gamma A*_\gamma A*_\gamma \ldots*_\gamma A}_l.
\end{equation}

\end{definition}


\begin{definition} [$\mathcal{P}_{\gamma, \mathbf{k}}$ and $\mathcal{P}_{\gamma,k}$] \label{def_operator2}
For $k \in \mathbb{N}$, $\mathbf{k}=(k_1, \cdots, k_s) \in \mathbb{N}^s$ with \nc $s\in \N$\nc, and $1\le  \gamma \le \gamma_{\max}$, we define
\begin{equation}
\mathcal{P}_{\gamma, k} G_{\mathbf u\mathbf v} := G_{\mathbf u\mathbf v}^{*_{\gamma}(k+1)}, \ \ \mathcal{P}_{\gamma,\mathbf{k}}\left(\prod_{t=1}^p G_{\mathbf u_t \mathbf v_t}\right):=\prod_{t=1}^p \mathcal{P}_{\gamma, k_t} G_{\mathbf u_t \mathbf v_t},
\end{equation}
where we abbreviate $G_{\mathbf u\mathbf v}^{*_{\gamma}(k+1)} \equiv (G^{*_{\gamma}(k+1)})_{\mathbf u\mathbf v}$. If $\mathfrak G_1$ and $\mathfrak G_2$ are products of resolvent entries as above, then we define
\begin{equation}
 \mathcal{P}_{\gamma,\mathbf{k}} (\mathfrak G_1+\mathfrak G_2):= \mathcal{P}_{\gamma,\mathbf{k}}\mathfrak G_1 +  \mathcal{P}_{\gamma,\mathbf{k}}\mathfrak G_2.
\end{equation}
Note that $ \mathcal{P}_{\gamma, k}$ and $ \mathcal{P}_{\gamma,\mathbf{k}}$ are not linear operators, but just notations we use for simplification.
Similarly, for the product of the entries of $G-\Pi$, we define
 \begin{equation}
{\mathcal{P}}_{\gamma, \mathbf{k}}\left(\prod_{t=1}^p (G-\Pi)_{\mathbf u_t\mathbf v_t}\right):=\prod_{t=1}^p {\mathcal{P}}_{\gamma, k_t}(G-\Pi)_{\mathbf u_t\mathbf v_t} ,
 \end{equation}
 where 
 \[ {\mathcal{P}}_{\gamma, k}(G-\Pi)_{\mathbf u \mathbf v}:=\begin{cases} 
      (G-\Pi)_{\mathbf u\mathbf v}, & \text{if }  k=0, \\
       G_{\mathbf u\mathbf v}^{*_\gamma(k+1)}, &  \text{otherwise}.
   \end{cases}\]
\end{definition}

\begin{remark}
It is easy to see that for any fixed $k\in \mathbb N$, $\mathcal{P}_{\gamma, k} G_{\mathbf u\mathbf v} $ is a sum of finitely many products of $(k+1)$ resolvent entries of the form $G_{\mathbf x \mathbf y}$, $\mathbf x,\mathbf y\in \{\mathbf u,\mathbf v, \mathbf w_i,\mathbf w_\mu\}$. Hence by \eqref{aniso_law} and \eqref{5BOUND1}, we can bound $\mathcal{P}_{\gamma, k} G_{\mathbf u\mathbf v} $ by $\OO_\prec(1)$. This is one of the main reasons why we need to prove the stronger anisotropic local law for $G$, rather than the entrywise local law only as in \cite{DY,LY}.
\end{remark}

Now we begin to perform the resolvent comparison strategy. The basic idea is to expand $S$ and $T$ in terms of $R$ using the resolvent expansions as in (\ref{RESOLVENTEXPANSION}) and (\ref{RESOLVENTEXPANSION2}), and then compare the two expressions. 
The key of the comparison argument is the following Lemma \ref{Greenfunctionrepresent}. Its proof is almost the same as the one for \cite[Lemma 6.5]{LY}. In fact, we can copy their arguments almost verbatim, except for some notational differences. Hence we omit the details. 



\begin{lemma} 
\label{Greenfunctionrepresent} 
Given $z\in \wt S(c_0,C_0,\fa,\e)$ and $\Phi(i,\mu)=\gamma$. \nc Let $r>0$ be a fixed constant and $p\in \N$ be a fixed integer. \nc Then for $S,R$ in (\ref{R}), we have 
\begin{equation}\label{ONLYPROVE}
\begin{split}
\mathbb{E}\prod_{t=1}^p S_{\mathbf u_t \mathbf v_t} & = \sum_{0 \leq k \leq 4} A_{k} \mathbb{E}\left[(- x_{i\mu})^{k}\right]  +\sum_{5 \leq \vert \mathbf{k} \vert \leq {r}/{\phi}, \mathbf k\in \mathbb N^p } \mathcal{A}_{  \mathbf{k}  }\mathbb{E} \, \mathcal{P}_{\gamma,\mathbf{k} }\prod_{t=1}^p S_{\mathbf u_t \mathbf v_t}+\OO_\prec(N^{-r}), 
\end{split}
\end{equation}
where $A_k$, $0\le k \le 4$, depend only on $R$, $\mathcal A_{\mathbf k}$'s do not depend on the deterministic unit vectors $(\mathbf u_t, \mathbf v_t)$, $1\leq t \leq p$, and we have bounds
\begin{equation}
\vert \mathcal{A}_{  \mathbf{k} } \vert \le N^{-{\vert \mathbf{k} \vert  \phi} /{10}-2} . \label{INDEXBOUND}
\end{equation}
Similarly, we have
\begin{equation}\label{ONLYPROVE2}
\begin{split}
& \mathbb{E}\prod_{t=1}^p (S-\Pi)_{\mathbf u_t\mathbf v_t} = \sum_{0 \leq k \leq 4} \wt A_{k} \mathbb{E}\left[(-x_{i\mu})^{k}\right] +\sum_{5 \leq \vert \mathbf{k} \vert \leq {r}/{\phi}, \mathbf k \in \mathbb N^p} \mathcal{A}_{  \mathbf{k}  }\mathbb{E} \, \mathcal{P}_{\gamma,\mathbf{k}}\prod_{t=1}^p (S-\Pi)_{\mathbf u_t\mathbf v_t}+\OO_\prec(N^{-r}), 
\end{split} 
\end{equation}
where $\wt{A}_k$, $0\leq k \leq 4$, again depend only on $R$. Finally, we have
\begin{equation}
\begin{split}
\mathbb{E}\prod_{t=1}^p S_{\mathbf u_t\mathbf v_t} &=  \mathbb{E}\prod_{t=1}^p R_{\mathbf u_t\mathbf v_t} + \sum_{1 \leq \vert  \mathbf{k}\vert \leq r/{\phi}, \mathbf k \in \mathbb N^p} \wt{\mathcal{A}}_{\mathbf{k}}\mathbb{E} \, \mathcal{P}_{\gamma,\mathbf{k}} \prod_{t=1}^p S_{\mathbf u_t \mathbf v_t}+\OO_\prec (N^{-r}), 
\end{split}
\label{ONLYPROVE3}
\end{equation}
where $\wt{\mathcal A}_{\mathbf k}$'s do not depend on $(\mathbf u_t, \mathbf v_t)$, $1\leq t \leq p$, and 
\begin{equation}\label{tildeA}
\vert \mathcal{\wt{A}}_{\mathbf{k}} \vert \leq N^{-\vert \mathbf k \vert \phi /10} .
\end{equation}
Note that the terms $A_k$, $\wt A_k$, $\mathcal A_{\mathbf k}$ and $\wt{\mathcal A}_{\mathbf k}$ do depend on $\gamma$ and we have omitted this dependence in the above expressions.
\end{lemma}

Next we use Lemma \ref{Greenfunctionrepresent} to finish the proof of Lemma \ref{thm_largebound} and Lemma \ref{lem_comgreenfunction}. It is obvious that a result similar to Lemma \ref{Greenfunctionrepresent} also holds for the product of $T$ entries. As in (\ref{ONLYPROVE}), we define the notation $\mathcal{A}^{\gamma,a}$, $a=0,1$ as follows:
\begin{equation} \label{greenpowers}
\begin{split}
&\mathbb{E}\prod_{t=1}^p S_{\mathbf u_t\mathbf v_t} = \sum_{0 \leq k \leq 4} A_{k} \mathbb{E}\left[(-x_{i\mu})^{k}\right] +\sum_{5 \leq \vert \mathbf{k} \vert \leq {r}/{\phi}, \mathbf k\in \mathbb N^p } \mathcal{A}_{  \mathbf{k}}^{\gamma, 0}\mathbb{E} \, \mathcal{P}_{\gamma,\mathbf{k} }\prod_{t=1}^p S_{\mathbf u_t\mathbf v_t}+\OO_\prec(N^{-r}),
\end{split}
\end{equation}
\begin{equation} \label{greenpowert}
\begin{split}
&\mathbb{E}\prod_{t=1}^p T_{\mathbf u_t\mathbf v_t} = \sum_{0 \leq k \leq 4} A_{k} \mathbb{E}\left[(-\wt{x}_{i\mu})^{k}\right] +\sum_{5 \leq \vert \mathbf{k} \vert \leq {r}/{\phi}, \mathbf k\in \mathbb N^p } \mathcal{A}_{  \mathbf{k}}^{\gamma, 1}\mathbb{E} \, \mathcal{P}_{\gamma,\mathbf{k} }\prod_{t=1}^p T_{\mathbf u_t\mathbf v_t} +\OO_\prec(N^{-r}).
\end{split}
\end{equation}
Since $A_{k}$, $0 \leq k \leq 4$, depend only on $R$ and $x_{i \mu}$, $\wt{x}_{i \mu}$ have the same first four moments, we get from (\ref{greenpowers}) and (\ref{greenpowert}) that
\begin{equation} \label{telescoping_2}
\begin{split}
& \mathbb{E}\prod_{t=1}^p G_{\mathbf u_t\mathbf v_t} -\mathbb{E}\prod_{t=1}^p \wt G_{\mathbf u_t\mathbf v_t} =\sum_{\gamma=1}^{\gamma_{\max}} \left(\mathbb{E}\prod_{t=1}^p G^\gamma_{\mathbf u_t\mathbf v_t} -\mathbb{E}\prod_{t=1}^p G^{\gamma-1}_{\mathbf u_t\mathbf v_t}\right) \\
& = \sum_{\gamma=1}^{\gamma_{\max}} \sum_{\mathbf k\in \mathbb N^p}^{5 \leq \vert \mathbf{k} \vert \leq {r}/{\phi}}\left(\mathcal{A}^{\gamma,0}_{  \mathbf{k}  }\mathbb{E} \,\mathcal{P}_{\gamma,\mathbf{k} }\prod_{t=1}^p G^\gamma_{\mathbf u_t\mathbf v_t} -\mathcal{A}_{  \mathbf{k}}^{\gamma,1}\mathbb{E} \, \mathcal{P}_{\gamma,\mathbf{k}}\prod_{t=1}^p G^{\gamma-1}_{\mathbf u_t\mathbf v_t}\right) +\OO_\prec(N^{-r+2}).
\end{split}
\end{equation}
where we abbreviate $G:=G(X,z)$ and $\wt G:=G(\wt X,z)$. With a similar argument, we also have
\begin{equation} \label{telescoping_1}
\begin{split}
& \mathbb{E}\prod_{t=1}^p (G-\Pi)_{\mathbf u_t\mathbf v_t} -\mathbb{E}\prod_{t=1}^p (\wt G-\Pi)_{\mathbf u_t\mathbf v_t}\\
&= \sum_{\gamma=1}^{\gamma_{\max}} \sum_{\mathbf k\in \mathbb N^p}^{5 \leq \vert \mathbf{k} \vert \leq {r}/{\phi}}\left(\mathcal{A}^{\gamma,0}_{  \mathbf{k}  }\mathbb{E} \,\mathcal{P}_{\gamma,\mathbf{k} }\prod_{t=1}^p (G^\gamma-\Pi)_{\mathbf u_t\mathbf v_t} -\mathcal{A}_{  \mathbf{k}}^{\gamma,1}\mathbb{E} \, \mathcal{P}_{\gamma,\mathbf{k}}\prod_{t=1}^p (G^{\gamma-1}-\Pi)_{\mathbf u_t\mathbf v_t}\right)+\OO_\prec(N^{-r+2}).
\end{split}
\end{equation}

\nc Next, we notice that $\mathcal{P}_{\gamma,\mathbf{k}}\prod_{t=1}^p G^{\gamma-a}_{\mathbf u_t \mathbf v_t}$ is also a sum of the products of $G$ entries. Hence we can apply (\ref{telescoping_2}) to 
$$\mathbb{E} \mathcal{P}_{\gamma,\mathbf{k}}\prod_{t=1}^p G^{\gamma-a}_{\mathbf u_t \mathbf v_t}-\mathbb{E} \mathcal{P}_{\gamma,\mathbf{k}}\prod_{t=1}^p \wt G_{\mathbf u_t \mathbf v_t}$$ 
with $\gamma_{\max}$ replaced by $\gamma-a$. 
Iterating this process for $l = 2r/{\phi}$ times, we finally can obtain that \nc
\begin{align} 
 \left \vert \mathbb{E}\prod_{t=1}^p G^{\gamma_{\max}}_{\mathbf u_t\mathbf v_t}  \right\vert & \leq    \left\vert \mathbb{E}\prod_{t=1}^p G^{0}_{\mathbf u_t\mathbf v_t} \right\vert   + \OO_\prec\left( \max_{\mathbf{k}, l}(N^{-2})^l(N^{-\phi/10})^{\sum_{i}\vert \mathbf{k}_i \vert}\sum_{\gamma_1, \cdots, \gamma_l}  \left\vert \mathbb{E}\mathcal{P}_{\gamma_l, \mathbf{k}_l} \cdots \mathcal{P}_{\gamma_1, \mathbf{k}_1} \prod_{t=1}^p G^{0}_{\mathbf u_t\mathbf v_t} \right\vert \right)\nonumber \\
&+\OO_\prec(N^{-r+2}). \label{teles_inequality3}
\end{align}  
where 
\begin{equation} \label{indexassumption}
1\le l \leq 2r/{\phi}, \quad \mathbf{k}_1 \in \mathbb{N}^p, \  \ \mathbf{k}_2 \in \mathbb{N}^{p + \vert \mathbf{k}_1 \vert}, \ \ \mathbf{k}_3 \in \mathbb{N}^{p + \vert \mathbf{k}_1 \vert+ \vert \mathbf{k}_2 \vert}, \ \ldots, \ \text{ and } \ 5 \leq \vert \mathbf{k}_i \vert \leq \frac{r}{\phi}.
\end{equation}
\nc For the details of the above derivation, we refer the reader to the arguments between (6.25) and (6.31) in \cite{DY}. \nc The above estimate still holds if we replace some of the $G$ entries with $\overline G$ entries, since we only need to use the absolute bounds for the resolvent entries. Of course, using \eqref{telescoping_1} instead of \eqref{telescoping_2}, we can obtain a similar estimate
\begin{align} 
& \left \vert \mathbb{E}\prod_{t=1}^p \left(G^{\gamma_{\max}}-\Pi\right)_{\mathbf u_t\mathbf v_t}  \right\vert  \leq    \left\vert \mathbb{E}\prod_{t=1}^p \left(G^{0}-\Pi\right)_{\mathbf u_t\mathbf v_t} \right\vert  \nonumber\\
&  + \OO_\prec\left( \max_{\mathbf{k}, l}(N^{-2})^l(N^{-\phi/10})^{\sum_{i}\vert \mathbf{k}_i \vert}\sum_{\gamma_1, \cdots, \gamma_l}  \left\vert \mathbb{E}\mathcal{P}_{\gamma_l, \mathbf{k}_l} \cdots \mathcal{P}_{\gamma_1, \mathbf{k}_1} \prod_{t=1}^p \left(G^{0}-\Pi\right)_{\mathbf u_t\mathbf v_t} \right\vert \right)+\OO_\prec(N^{-r+2}). \label{teles_inequality4}
\end{align}  

Now we use Lemma \ref{Greenfunctionrepresent}, (\ref{teles_inequality3}) and (\ref{teles_inequality4}) to complete the proof of Lemma \ref{thm_largebound} and Lemma \ref{lem_comgreenfunction}.

\begin{proof}[Proof of Lemma \ref{thm_largebound}]
\nc The proof of this lemma is similar to the one for \cite[Lemma 3.17]{DY}, where the main difference lies in the estimate \eqref{bound_similar} below. \nc 
We apply (\ref{teles_inequality4}) to $(G-\Pi)_{\mathbf u\mathbf v}\overline{(G-\Pi)_{\mathbf u\mathbf v}}$ with $p=2$ and $r=3$. Recall that $\wt{X}$ is of bounded support $q= N^{-1/2}$. Then by (\ref{aniso_law}) and Lemma \ref{lem_stodomin}, we have 
\be\label{need_initial}
\mathbb E\vert (\wt G-\Pi)_{\mathbf u\mathbf v}  \vert^2 \prec \Psi^2(z).
\ee
Moreover, by \eqref{psi12} the remainder term $\OO_\prec( N^{-r+2})=\OO_\prec(N^{-1})$ in (\ref{teles_inequality4}) is negligible.
Hence it remains to handle the second term on the right-hand side of (\ref{teles_inequality4}), i.e.
\be\label{need_bound}
(N^{-2})^l \sum_{\gamma_1, \cdots, \gamma_l}  \left\vert \mathcal{P}_{\gamma_l, \mathbf{k}_l} \cdots \mathcal{P}_{\gamma_1, \mathbf{k}_1} \left| \left(G^{0}-\Pi\right)_{\mathbf u\mathbf v}\right|^2 \right\vert. 
\ee
For each product in (\ref{need_bound}), $\mathbf v$ appears exactly twice in the indices of $G$. These two $\mathbf v$'s appear as $G_{\mathbf v \mathbf w_a} G_{\mathbf w_b \mathbf v}$ in the product, where $\mathbf w_a, \mathbf w_b$ come from some $\gamma_k$ and $\gamma_{l}$ via ${\mathcal{P}}$. 
Let $\mathbf v=\left( {\begin{array}{*{20}c}
   {\mathbf v}_1   \\
   {\mathbf v}_2 \\
   \end{array}} \right)$ for ${\mathbf v}_1 \in\mathbb C^{\mathcal I_1}$ and ${\mathbf v}_2\in\mathbb C^{\mathcal I_2}$. By Lemma \ref{lem_comp_gbound}, after taking the averages $N^{-2}\sum_{\gamma_k}$ and $N^{-2}\sum_{\gamma_l}$, the term $G_{\mathbf v \mathbf w_a} G_{ \mathbf w_b\mathbf v}$ contributes a factor 
\begin{align}\label{bound_similar}
 \OO_\prec\left(\frac{\im \left(z^{-1}G^0_{\mathbf v_1\mathbf v_1}\right) + \im \left(G^0_{\mathbf v_2\mathbf v_2}\right) +\eta \left|G^0_{\mathbf v_1\mathbf v_1}\right|+\eta \left|G^0_{\mathbf v_2\mathbf v_2}\right|}{N\eta}\right) =\OO _\prec\left( \frac{\im m_{2c} +\Psi(z) }{N\eta}\right)=\OO_\prec ( \Psi^2(z) ),
\end{align}
where we used \eqref{aniso_law}. For all the other $G$ factors in the product, we control them by $\OO_\prec(1)$ using \eqref{5BOUND1}. Thus we have proved that $\eqref{need_bound}\prec \Psi^2(z)$. Together with (\ref{teles_inequality4}) and \eqref{need_initial}, this proves Lemma \ref{thm_largebound}.
\end{proof}

\begin{proof}[Proof of Lemma \ref{lem_comgreenfunction}]
\nc The proof of this lemma is similar to the one for \cite[Lemma 5.2]{DY}, where the main difference lies in \eqref{bound_similar3} below. \nc  For simplicity, we shall prove that
\begin{equation}
\begin{split}
&\vert \mathbb{E} \left(m(z)-m_{c}(z) \right)^p| \prec \vert \mathbb{E} \left(\wt m(z) - m_{c}(z)\right)^p \vert + \left(\Psi^2(z)+J^2+K \right)^p. 
\end{split}
\label{simpler_pf}
\end{equation}
\nc The proof for (\ref{KEYEYEYEY}) is exactly the same but with slightly heavier notations (in the product of $p$ terms, half of them are normal and the other half are complex conjugates). \nc

\nc We define a function of coefficients 
$$f(I,J)= n^{-p} \prod \delta _{i_t j_t}, \quad I=(i_1,i_2,\cdots,i_p)\in \mathcal I_1^p, \ \ J=(j_1,j_2,\cdots,j_p)\in \mathcal I_1^p.$$ 
\nc It is easy to check that
\begin{equation}
 \mathbb{E}\sum_{I,J}f(I,J)\prod_{t=1}^p (G^{\alpha}-\Pi)_{i_t j_t}= \mathbb{E}(m^{\alpha}-m_{c})^p, \ \ \alpha=0, \ \gamma_{\max} .
\end{equation}
Since $\mathcal{A}$'s do not depend on $i_t$ and $j_t$, we may consider a linear combination of (\ref{teles_inequality4}) with coefficients $f(I,J)$ and $r=p+2$: 
\begin{equation}\label{KEY22}
\begin{split}
& \left\vert \mathbb{E}\sum_{I,J}f(I,J)\prod_{t=1}^p \left(G - \Pi\right)_{i_t j_t} \right\vert =  \left\vert \mathbb{E}\sum_{I,J}f(I,J)\prod_{t=1}^p (\wt G-\Pi)_{i_t j_t} \right\vert \\
& +\OO_\prec \left(\max_{\mathbf{k}, l,\gamma}(N^{-\phi/10})^{\sum_{i}\vert \mathbf{k}_i \vert} \left\vert \mathbb{E}\sum_{I,J}f(I,J){\mathcal{P}}_{\gamma_l, \mathbf{k}_l } \cdots  { \mathcal{P}}_{\gamma_1, \mathbf{k}_1}\prod_{t=1}^p (\wt G-\Pi)_{i_t j_t} \right\vert \right) + \OO(N^{-p})
\end{split}
\end{equation}
Now to conclude (\ref{simpler_pf}), it suffices to control the second term on the RHS of (\ref{KEY22}). We consider the terms
\begin{flalign}\label{average_comparison1}
{\mathcal{P}}_{\gamma_l, \mathbf{k}_l  } \cdots {\mathcal{P}}_{\gamma_1, \mathbf{k}_1}\prod_{t=1}^p (\wt G-\Pi)_{i_t i_t} ,
\end{flalign}
for $\mathbf k_1,\ldots, \mathbf k_l$ satisfying (\ref{indexassumption}). 
For each product in (\ref{average_comparison1}) and any $1\le t \le p$, there are two $i_t$'s in the indices of $G$. These two $i_t$'s can only appear as (1) $(\wt{G}-\Pi)_{i_t i_t}$ in the product, or (2) $\wt G_{i_t \mathbf w_a} \wt G_{\mathbf w_b i_t}$, where $\bw_a, \bw_b$ come from some $\gamma_k$ and $\gamma_{l}$ via ${\mathcal{P}}$. Then after averaging over $n^{-p} \sum_{i_1,\cdots, i_p}$, this term becomes either (1) $\wt{m}-m_{c}$, which is bounded by $K$ by (\ref{KEYBOUNDS}), or (2) $n^{-1}\sum_{i_t} G_{i_t \mathbf w_a} G_{\mathbf w_b i_t}$, which is bounded as in \eqref{bound_similar} by 
\be\label{bound_similar3} \OO _\prec\left( \frac{\im m_{2c} +J }{N\eta}\right)=\OO_\prec \left( \Psi^2(z) + J^2\right).\ee
For other $G$ entries in the product with no $i_t$, we simply bound them by $\OO_\prec(1)$ using \eqref{KEYBOUNDS}. Then for any fixed $\gamma_1,\ldots,\gamma_l$, $\mathbf k_1,\ldots, \mathbf k_l$, we have proved that
\begin{equation}
\left\vert \frac{1}{n^p}\sum_{i_1,\ldots, i_p} \mathbb{E} {\mathcal{P}}_{\gamma_l, \mathbf{k}_l } \cdots  {\mathcal{P}}_{\gamma_1, \mathbf{k}_1}\prod_{t=1}^p \left(\wt{G}-\Pi\right)_{i_t i_t} \right\vert \prec \left(\Psi^2(z) + J^2+K\right)^p . \label{KEY444}
\end{equation}
Together with (\ref{KEY22}), this concludes (\ref{simpler_pf}).
\end{proof}

Finally, we give the proof of Theorem \ref{lem_comparison}. \nc Its proof is similar to the one for \cite[Theorem 3.16]{DY}. We only outline the proof by stating the key lemmas one can prove. \nc  
For the matrix $\wt{X}$ constructed in Lemma \ref{lem_decrease}, it satisfies the edge universality by the following lemma. 

\begin{lemma} 
\label{lem_smallcomp}
Let $X^{(1)}$ and $X^{(2)}$ be two separable covariance matrices satisfying the assumptions in Theorem \ref{LEM_SMALL} and the bounded support condition \eqref{eq_support} with $q= N^{-1/2}$. Suppose $b_N \le N^{1/3-c}$ for some constant $c>0$. 
Then there exist constants $\epsilon,\delta >0$ such that for any $s\in \mathbb R$, we have
\be 
\begin{split}
\mathbb{P}^{(1)} \left(N^{{2}/{3}}(\lambda_1-\lambda_{+}) \leq s - N^{-\epsilon}\right) - N^{-\delta}  \leq
\mathbb{P}^{(2)} \left(N^{{2}/{3}}(\lambda_1-\lambda_{+})\leq s \right) \\ \leq \mathbb{P}^{(1)} \left(N^{{2}/{3}}(\lambda_1-\lambda_{+}) \leq s+N^{-\epsilon}\right)+N^{-\delta}   ,
\end{split}
\ee
where $\mathbb{P}^{(1)}$ and $\mathbb{P}^{(2)}$ denote the laws of $X^{(1)}$ and $X^{(2)}$, respectively.
\end{lemma}

\begin{proof}
The proof of this lemma is similar to the ones in \cite[Section 6]{EKYY}, \cite[Section 6]{EYY}, \cite[Section 4]{PY} and \cite[Section 10]{Anisotropic}. The main argument involves a routine application of the Green's function comparison method (as the one in Lemma \ref{lem_compdiffsupport}) near the edge developed in \cite[Section 6]{EYY} and \cite[Section 4]{PY}. The proofs there can be easily adapted to our case using the anisotropic local law (Theorem \ref{LEM_SMALL}), the rigidity of eigenvalues (Theorem \ref{thm_largerigidity}), and the resolvent identities in Lemma \ref{lemm_resolvent} and Lemma \ref{lem_comp_gbound}. 
\end{proof}

Now it is easy to see that Theorem \ref{lem_comparison} follows from the following comparison lemma.

\begin{lemma}\label{eq_edgeuniv} 
Let $X$ and $\wt{X}$ be two matrices as in Lemma \ref{lem_decrease}. Suppose $b_N \le N^{1/3-c}$ for some constant $c>0$. Then there exist constants $\epsilon,\delta >0$ such that, for any $s \in \mathbb R$ we have
\begin{equation}\label{edgeXX}
\begin{split}
\mathbb{P}^{\wt X} \left(N^{{2}/{3}}(\lambda_1-\lambda_+)\leq s-N^{-\epsilon}\right) -N^{-\delta}  \leq \mathbb{P}^{X}(N^{{2}/{3}}\left({\lambda}_1-\lambda_+)\leq s\right) \\
\leq \mathbb{P}^{\wt X} \left(N^{{2}/{3}}(\lambda_1-\lambda_+)\leq s+ N^{-\epsilon}\right)+N^{-\delta},
\end{split}
\end{equation}
where $\mathbb{P}^X$ and $\mathbb{P}^{\wt X}$ are the laws for $X$ and $\wt{X}$, respectively.
\end{lemma}

To prove Lemma \ref{eq_edgeuniv}, it suffices to prove the following Green's function comparison result. \nc Its proof is the same as the one for \cite[Lemma 5.5]{DY}, so we skip the details. \nc

\begin{lemma}
\label{lem_compdiffsupport} 
Let $X$ and $\wt{X}$ be two matrices as in Lemma \ref{lem_decrease}. Suppose $F:\mathbb R\to \mathbb R$ is a function whose derivatives satisfy
\begin{equation}
\sup_{x} \vert {F^{(k)}(x)}\vert {(1+\vert x \vert)^{-C_1}} \leq C_1, \quad k=1,2,3, \label{FCondition}
\end{equation}
for some constant $C_1>0$. Then for any sufficiently small constant $\delta>0 $ and for any 
\begin{equation}\label{Eeta}
E,E_1,E_2 \in I_{\delta}:=\left\{x: \vert x -\lambda_{+}\vert \leq N^{-{2}/{3}+\delta}\right\} \ \ \text{and} \ \ \eta:=N^{-{2}/{3}-\delta},
\end{equation}
we have
\begin{equation}
\left| \mathbb{E}F\left(N\eta \im m(z)\right)-\mathbb{E}F\left(N\eta \im \wt{m} (z)\right) \right| \leq N^{-\phi + C_2 \delta}, \ \ z=E+\ii\eta,  \label{BDBD}
\end{equation}
and
\begin{align}
\left| \mathbb{E}F\left(N \int_{E_1}^{E_2} \operatorname{Im} m(y+\ii\eta)dy\right)- \mathbb{E} F\left(N \int_{E_1}^{E_2} \operatorname{Im} \wt{m}(y+\ii\eta)dy\right) \right| \leq N^{-\phi+ C_2\delta}, \label{BDBD1}
\end{align}
where $\phi $ is as given in Theorem \ref{LEM_SMALL} and $C_2>0$ is some constant. 
\end{lemma}

\begin{proof}[Proof of Lemma \ref{eq_edgeuniv}]
Although not explicitly stated, it was shown in \cite{EYY} that if Theorem \ref{thm_largerigidity} and Lemma \ref{lem_compdiffsupport} hold, then the edge universality \eqref{edgeXX} holds. More precisely, in Section 6 of \cite{EYY}, the edge universality problem was reduced to proving Theorem 6.3 of \cite{EYY}, which corresponds to our Lemma \ref{lem_compdiffsupport}. In order for this conversion to work, only the the averaged local law and the rigidity of eigenvalues are used, which correspond to the statements in our Theorem \ref{thm_largerigidity}.
\end{proof}

\begin{proof}[Proof of Theorem \ref{lem_comparison}]
Theorem \ref{lem_comparison} follows immediately combining Lemma \ref{lem_decrease}, Lemma \ref{lem_smallcomp} and Lemma \ref{eq_edgeuniv} .
\end{proof}

\end{document}